\documentclass{article}
\usepackage{amsmath,amsthm,amsfonts,amssymb,enumerate}
\usepackage{graphicx} 
\usepackage{color}
\theoremstyle{plain}
\newtheorem{theorem}{Theorem}
\newtheorem{lemma}[theorem]{Lemma}
\newtheorem{proposition}[theorem]{Proposition}
\newtheorem{cor}[theorem]{Corollary}
\theoremstyle{definition}
\newtheorem{rem}[theorem]{Remark}

\title{Walk/Zeta correspondence for quantum walks on path graphs: a pre-geometric structure implied by 
a generalized reflection principle}
\author{Jir\^o Akahori, Riki Kitano, Norio Konno, and Iwao Sato}
\date{October 2026}

\begin{document}

\maketitle

\begin{abstract}
We study the zeta functions of the reflected random walk and of the
reflected Szegedy walk on the finite path graph $P_N$, and show that both
admit an exact decomposition into local factors indexed by the infinite
dihedral group $\mathcal{G}_{0,N+1}$ generated by the two boundary
reflections.

The decomposition rests on two ingredients. The first is a generalized
reflection principle: for an arbitrary asymmetry $p\neq q$ the reflecting
boundary is implemented by an explicit involution of $\mathbb{R}^{\mathbb{Z}}$
which is \emph{not} a point reflection, and the transition probabilities of
the reflected walk are recovered by summing the free-walk kernel over
$\mathcal{G}_{0,N+1}$. 
This yields a discrete trace formula, that is, a
product expansion of $\det(I-u\mathbf{P}^{(N)})^{-1}$ into a boundary factor
and one local zeta factor for each group element. 
The second is a local
Konno--Sato correspondence: after a symmetrizing similarity transformation
the classical and the quantum local factors are related term by term, which
promotes the discrete trace formula to a product expansion of
$\det(I-u\mathbf{U}^{(N)})^{-1}$.

We then show that this discrete decomposition already carries the
information that survives the continuum limits. First, in the thermodynamic
limit the normalized zeta functions converge to a Mahler measure, and this
Mahler measure is exactly the local factor of the identity element: in the
limit, no other group element and no boundary term contributes. 
Second,
under the diffusive scaling $n\asymp N^{2}$ the classical trace formula
reproduces the Poisson summation formula, whereas the quantum one reproduces
the Dirac comb identity under the ballistic scaling $n\asymp N$; the two
identities are therefore not two forms of one statement but live on
different time scales, and in the quantum case the discrete identity is
already exact at finite $N$. 

Third, under a double scaling limit in which
the asymmetry is tuned as $p-q=M/N$, we obtain a continuum trace formula of massive Dirac type, and the discrete local decomposition
is inherited factor by factor: the even words converge to the
periodic-orbit Euler product, the identity to the free-field factor, 
and the odd words together with the boundary to a finite scattering
phase. The three limits then multiply to a classical object: the limiting
determinant is the zeta-regularized determinant of $-\mathrm{d}^{2}/\mathrm{d}x^{2}+\varkappa^{2}$
on $[0,1]$ with Dirichlet ends, and the three classes of words are precisely
its periodic-orbit factor, its vacuum factor and its Gelfand--Yaglom
normalization.

Taken together, these results indicate that the discrete trace formula is
not merely a combinatorial identity but a pre-geometric structure: the
features that become the length spectrum, the heat/wave dichotomy and the
mass term in the continuum are already present, exactly, at finite $N$.
\end{abstract}

\section{Introduction}

The intricate relationship between spectral geometry, probability theory,
and quantum mechanics has been profoundly illuminated by the study of zeta
functions on graphs. Originally introduced for regular graphs as an analogue
of the Selberg zeta function \cite{Selberg1956}, the Ihara zeta function
\cite{Ihara1966} bridges the topological length spectrum of closed paths
with the algebraic spectrum of adjacency operators. This framework was
extended to quantum information theory by the Konno--Sato theorem
\cite{KS2012QIP, KKS2021QIP}, which established a determinant-based exact
correspondence---the Konno--Sato correspondence---between the zeta functions
of classical random walks and those of their quantum counterparts, such as
the Szegedy walk \cite{Sze}. The development of the ``Walk/Zeta
correspondence'' has since been a central theme in a series of works by
Konno, Sato and their collaborators \cite{KKS2021QIP, KKS2023JSP,
KonnoTamura2021}, mostly for Grover walks or for walks on regular lattices,
that is, on graphs without boundary.

The present paper is a new branch of that program: we work on a graph
\emph{with} boundary, the finite path $P_N$ with two reflecting endpoints,
and we allow an arbitrary spatial asymmetry $p\neq q$. Two features are
specific to this setting. The first is that the global zeta function is no
longer a single spectral object but splits, in an exact and canonical way,
into contributions indexed by a reflection group; the second is that the
asymmetry, which has no counterpart on a vertex-transitive graph, acts in
the continuum limit as a mass. The purpose of the paper is to make both
statements precise, and to argue that the resulting discrete identity
already contains the structure of the continuum trace formulas it degenerates
to.

\medskip
\noindent\textbf{The discrete trace formula.}
Our principal object is an exact factorization of the zeta function of the
reflected walk on $P_N$. Its algebraic source is a generalized reflection
principle (Proposition \ref{TraceFormula-P}). For $p\neq q$ the reflecting
boundary at $0$ cannot be implemented by the point reflection
$x\mapsto -x$; instead we construct in Section \ref{sec:4} an explicit
infinite matrix $\mathsf{M}$ whose associated operator $\mathcal{M}_0$ is an
involution of $\mathbb{R}^{\mathbb{Z}}$ and whose invariant subspace is
exactly the set of functions compatible with the reflecting boundary
condition. Its anti-diagonal is a weighted point reflection and it carries,
in addition, a tail supported beyond the mirror image; the two boundary
involutions generate a group $\mathcal{G}_{0,N+1}$ isomorphic to the
infinite dihedral group. Summing the free-walk kernel over this group
reproduces the reflected transition probabilities, and hence, after taking
traces, yields the discrete trace formula of Theorem \ref{DiscreteTrace}:
\begin{align*}
    \det(I_{N+2}-u\mathbf{P}^{(N)})^{-1}
    =\zeta_{\partial X^{(N)}}(u)\prod_{W\in\mathcal{G}_{0,N+1}}\zeta_{W,X^{(N)}}(u),
\end{align*}
a boundary factor times one local zeta factor per group element. We also
prove the absolute convergence of the expansion (Lemma
\ref{prop:absolute-summability}), which is what makes the factorization an
identity of analytic functions rather than a formal rearrangement.

\medskip
\noindent\textbf{From the classical to the quantum walk.}
The second ingredient is a \emph{local} Konno--Sato correspondence (Lemma
\ref{local-konno-sato}). A diagonal similarity transformation symmetrizes
both the transition matrix and the reflection operators, and in the
symmetrized picture the generating function of the free Szegedy walk and
that of the free random walk differ only on the diagonal. Consequently the
classical and quantum local factors attached to one and the same group
element are related by the elementary factor $(1+u^{2})^{-c(\tilde W)}$,
with $c(\tilde W)$ non-zero only for the identity. Combining this with the
global Konno--Sato correspondence upgrades the discrete trace formula to its
quantum counterpart (Theorem \ref{QWDiscreteTrace}), and shows that the
arithmetic relation between the classical and the quantum zeta functions
holds bulk-term by bulk-term, the whole discrepancy being absorbed into the
boundary factor.

\medskip
\noindent\textbf{What the decomposition sees in the continuum limits.}
The rest of the paper tests the discrete decomposition against three
different limits, and in each of them the outcome is governed by an
identifiable part of the group.

\emph{(i) Thermodynamic limit: only the identity survives.} In Section
\ref{sec:3} the normalized zeta functions of $P_N$ are shown to converge, as
$N\to\infty$, to integrals of Mahler-measure type, namely to $\zeta_X(u)$
and $\zeta_Y(u)$, the local return zeta functions of the free walks on
$\mathbb{Z}$ (Theorems \ref{M-Th3} and \ref{Th3}). The trace formula then
identifies this limit: the local factor of the identity element is
$\zeta_{id,X^{(N)}}(u)=\zeta_X(u)^{N}$, respectively
$\zeta_{id,Y^{(N)}}(u)=\zeta_Y(u)^{N}$, so the Mahler measure
\emph{is} the identity factor, and the whole of
$\mathcal{G}_{0,N+1}\setminus\{id\}$ together with the boundary factor
contributes nothing after normalization. The infinite-graph limit is thus
not a global spectral statement but a purely local one.

\emph{(ii) Scaling limits: two summation formulas on two time scales.}
Section \ref{sec:5} evaluates the discrete trace formula for $p=q$ in the
diffusive regime $n=t(N+1)^{2}$ and recovers the classical Poisson summation
formula (Corollary \ref{cor:poisson}); here the translations of
$\mathcal{G}_{0,N+1}$ produce the Gaussian side and the reflections produce
the constant term. Section \ref{sec:7} performs the same computation for the
Szegedy walk, but the relevant regime is now ballistic, $n=v(N+1)$, and the
output is the Dirac comb identity (Theorem \ref{thm:dirac-comb}). The
contrast is the point: the same discrete identity degenerates to the heat
trace on $S^{1}$ on the time scale $n\asymp N^{2}$ and to the wave trace on
$S^{1}$ on the time scale $n\asymp N$, so the classical and quantum Poisson
formulas are not two readings of one limit. Moreover the quantum
computation produces an identity which is already exact for every $N\ge2$,
the boundary anomaly cancelling the reflection terms exactly, whereas the
classical one requires a genuine passage to the limit.

\emph{(iii) Double scaling limit: the decomposition is inherited.} In
Section \ref{sec:8} we let the asymmetry vanish at the same rate as the
lattice spacing, $p-q=M/N$ and $u=e^{-s/N}$. The spatial asymmetry then
survives as an effective mass: 
the limiting determinant exhibits a massive Dirac-type dispersion relation, with $\varkappa=\sqrt{s^2+M^2}$ in place of $s$.
The essential point is not
the limit itself but that the discrete decomposition passes to it factor by
factor (Theorem \ref{thm:double_scaling}): the even words converge, term by
term, to the local factors of the iterates of the unique primitive periodic
orbit, and hence their product to the continuum Euler product; the identity
converges to the free-field factor $e^{s-\varkappa}$; the odd words and the
boundary factor separately diverge and vanish, but their product converges
to the finite scattering phase $\varkappa/s$. Two exact algebraic facts make
this possible and are of independent interest: the tails of the boundary
involutions cancel identically along even words (Lemma
\ref{lem:translations}), so that translations act as genuine translations on
the bulk for every $p,q$ and every $N$; along odd words they do not cancel,
and their surviving contribution is exactly the factor $\varkappa^{2}/s^{2}$
which is later absorbed by the boundary (Lemma \ref{lem:reflections}).
Finally, the three limits are not independent: their product is the
zeta-regularized determinant of $\mathcal{D}^{2}+s^{2}$ on $[0,1]$ with
Dirichlet ends, where $\mathcal{D}$ is the massive Dirac operator of the
limit (Theorem \ref{thm:identification}). Each class of words is thereby
identified with one of the three canonical constituents of a
Gelfand--Yaglom determinant, and in particular the boundary factor
$\varkappa/s$ is its Wronskian normalization, equivalently the Jacobian
$\mathrm{d}s/\mathrm{d}\varkappa$ of the massive dispersion relation.

\medskip
\noindent\textbf{A pre-geometric reading.}
The Selberg trace formula \cite{Selberg1956} and the wave trace identities
\cite{Duistermaat1975, Chavel1984} are normally obtained in the continuum,
from differential operators on Riemannian manifolds, and the length spectrum
appears there as a property of geodesics. The three limits above suggest the
opposite order of logic. The length spectrum is already visible at finite
$N$ as the word length in $\mathcal{G}_{0,N+1}$; the heat/wave dichotomy is
already visible as a choice of time scale in one and the same discrete
identity; the mass is already visible as the asymmetry $p-q$. In this sense
the discrete trace formula is a pre-geometric skeleton: the continuum trace
formulas are degenerations of it, and the decomposition into local zeta
factors is the structure that is transported along the degeneration.

\medskip
The paper is organized as follows. Section \ref{sec:2} fixes notation and
defines the walks and their zeta functions. Section \ref{sec:3} establishes
the thermodynamic limits and their Mahler measure representations. Section
\ref{sec:4} constructs the generalized reflections and proves that they are
involutions. Section \ref{sec:5} proves the generalized reflection principle
and the discrete trace formula for the reflected random walk, and derives
the Poisson summation formula from it. Section \ref{sec:6} establishes the
local Konno--Sato correspondence and the trace formula for the Szegedy walk.
Section \ref{sec:7} derives the quantum (Dirac comb) counterpart of the
Poisson summation formula. Section \ref{sec:8} treats the double scaling
limit and the massive Dirac-type trace formula. Section \ref{sec:9} collects
concluding remarks.

\section{Zeta Functions of Reflected and Free Walks}\label{sec:2}

\subsection{Notations}
\begin{itemize}
    \item For the $ N \times N $ unit matrix we write 
$ I_N $. 
    \item We denote the set of eigenvalues of a given matrix $ A $ by $\mathrm{Spec} (A)$. 
\end{itemize}

\subsection{Reflected Random and Quantum Walks}\label{sec:2.2}
Let $ N \geq 2 $ and 
$ P_N := \{ 0,1,\cdots, N-1, N, N+1\} $ 
and $ p, q \in (0,1) $ with $ p+q =1 $, and 
consider 
\begin{enumerate}
    \item (Reflected Random Walk)
    Let $ X^{(N)} $ be a Markov chain on $ P_N $
    (defined on a probability space $ (\Omega, \mathcal{F}, \mathbb{P}) $)
    such that 
    \begin{align*}
        \begin{aligned}
            \mathbb{P} (X^{(N)}_{t+1} =k | X^{(N)}_t)
            = \begin{cases}
                p & X^{(N)}_t = k+1, k=0, 1, \cdots, N-1 \\
                q & X^{(N)}_t = k-1, k=2, \cdots, N+1 \\
                1 & X^{(N)}_t= N +1, k=N, \, \& \, X^{(N)}_t= 0, k=1 \\
                0 & \text{otherwise}.
            \end{cases}
        \end{aligned}
    \end{align*}
    We denote the transpose of the transition matrix by $ \mathbf{P}^{(N)}$, i.e. 
the $(i,j)$ entry of $ \mathbf{P}^{(N)} $ is 
$ \mathbb{P} (X^{(N)}_{t+1} =i | X^{(N)}_t=j )$ irrespective of $ t $, 
to be consistent with 
the quantum walk convention.   
    \item (Reflected Quantum Walk a.k.a. Szegedy Walk \cite{Sze})
    Write $ |i\rangle $
as $ (0,\cdots,\overset{i}{\check{1}}, \dots, 0)^\top $ for $ i \in \mathbb{Z} $
and $ |\mathrm{L}\rangle := (1,0)^\top $ and $ |\mathrm{R}\rangle := (0,1)^\top $.
Let 
    \begin{align*}
    \begin{aligned}
        & \mathcal{H}_N := \\
       & \mathrm{span}_{\mathbb{C}} \{       
        |i\rangle \otimes | \mathrm{L} \rangle, |i\rangle \otimes | \mathrm{R} \rangle : i =1, 2, \cdots, N \} \oplus \mathbb{C} |0\rangle \otimes |\mathrm{R}\rangle 
        \oplus \mathbb{C} |N+1\rangle \otimes |\mathrm{L}\rangle 
    \end{aligned}
    \end{align*}
which is a subspace of $ \mathbb{C}^{P_N} \otimes \mathbb{C}^2 $.
Let $ \mathrm{C}^{(N)} \in \mathrm{End} (\mathcal{H}_N) $ be defined by 
\begin{align*}
        \begin{aligned}
            \mathrm{C}^{(N)} &:= |0\rangle \langle 0| \otimes |\mathrm{R}\rangle \langle \mathrm{R}|
            + \sum_{j=1}^N |j\rangle \langle j| \otimes ( 2 |\phi \rangle \langle \phi | - I_2 ) \\
            & \hspace{2cm} + |N+1 \rangle \langle N+1| \otimes |\mathrm{L} \rangle \langle \mathrm{L} |,
        \end{aligned}
    \end{align*}
where 
    \begin{align*}
        |\phi \rangle = \sqrt{p} |\mathrm{L}\rangle + \sqrt{q} |\mathrm{R}\rangle,
    \end{align*}
    and let $ \mathrm{S}^{(N)}\in \mathrm{End} (\mathcal{H}_N) $
    be the shift operator 
    defined by
        \begin{align*}
        \begin{aligned}
            \mathrm{S}^{(N)}  | i\rangle \otimes |J \rangle &:= 
            \begin{cases}
                | i+1\rangle \otimes| \mathrm{L} \rangle & (J=\mathrm{R}) \\
                | i-1\rangle \otimes | \mathrm{R}\rangle & (J=\mathrm{L}). 
            \end{cases}
        \end{aligned}
    \end{align*}  
The linear operator 
    $ \mathbf{U}^{(N)} = S^{(N)} C^{(N)} $
is easily seen to be a unitary one 
and defines 
the quantum walk $ Y^{(N)} $ by 
    \begin{align*}
      \mathbb{P}_\psi (Y^{(N)}_n=k)
      = \left\Vert (|k\rangle \langle k| \otimes I_2 ) (\mathbf{U}^{(N)})^n |\psi\rangle \right\Vert^2, \quad  k \in P_N,  
    \end{align*}
for $ \psi \in \mathcal{H}_N$
with $ \Vert \psi \Vert = 1$, 
where we regard $ \{ Y^{(N)}_n \}_{n=0}^\infty $, starting at $ \psi $, as a family of random variables, without specifying joint distributions. 
\end{enumerate}

\subsection{Zeta Functions Associated with the Reflected Walks}
We associate a (normalized) zeta function with each walk by
\begin{align*}
\begin{aligned}
    \bar{\zeta}_{X^{(N)}} (u) &:= \left( \det (I_{N+2} - u \mathbf{P}^{(N)}) \right)^{-\frac{1}{N+1}} \\
    &=: e^{-\frac{1}{N+1}\mathrm{Log}\det (I_{N+2} - u \mathbf{P}^{(N)}) }, \quad u \in \mathbb{C}
    \end{aligned}
\end{align*}
and 
\begin{align*}
    \bar{\zeta}_{Y^{(N)}} (u) &:= \left( \det (I_{2N+2} - u \mathbf{U}^{(N)}) \right)^{-\frac{1}{N+1}} \\
    &=: e^{-\frac{1}{N+1}\mathrm{Log}\det (I_{2N+2} - u \mathbf{U}^{(N)}) }, \quad u \in \mathbb{C},
\end{align*}
respectively. Here,  
we take $ \mathrm{Log} z = \log |z| + i \theta $ with $ -\pi < \theta \leq \pi $ for $ \theta = \arg z $, and we understand them to be functions that take values in $ \bar{\mathbb{C}} = \mathbb{C} \cup \{ \infty\} $. 
Further, we also associate a zeta function 
to the free random walk $ X $ on $ \mathbb{Z} $ by
\begin{align}\label{FreeZeta}
    \begin{aligned}
    \zeta_X (u) = \exp \sum_{n=1}^\infty\frac{ P(X_n=0|X_0=0) }{n} u^n.
    \end{aligned}
\end{align}
Here $ X $ is defined as usual by 
     \begin{align*}
        \begin{aligned}
            \mathbb{P} (X_{t+1} =k | X_t)
            = \begin{cases}
                p & X_t = k+1 \\
                q & X_t = k-1 \\
                0 & \text{otherwise},
            \end{cases}
        \end{aligned}
    \end{align*}
and we denote the transpose of the transition matrix by $ \mathbf{P}^{(\infty)}$.  
This local return trace approach for infinite graphs aligns with the framework of $L^2$-zeta functions \cite{Luck1994} and is conceptually related to the regularized determinants and local $L$-factors discussed in number theory \cite{Deninger1992, Kurokawa2005}.

We also consider the free Szegedy walk $\mathbf{U}^{(\infty)}$ on the infinite line $\mathbb{Z}$.
The state space is defined as the Hilbert space $\mathcal{H}_{\infty} = \ell^2(\mathbb{Z}) \otimes \mathbb{C}^2$, with the orthonormal basis given by $\{ |x\rangle \otimes |J\rangle : x \in \mathbb{Z}, J \in \{\mathrm{L}, \mathrm{R}\} \}$.

The evolution operator $\mathbf{U}^{(\infty)} = \mathrm{S}_{\infty} \mathrm{C}_{\infty}$ is the product of the coin operator $\mathrm{C}_{\infty}$ and the shift operator $\mathrm{S}_{\infty}$, defined as:
\begin{align*}
    \mathrm{C}_{\infty} &:= \sum_{x \in \mathbb{Z}} |x\rangle \langle x| \otimes (2 |\phi \rangle \langle \phi | - I_2), \\
    \mathrm{S}_{\infty} (|x\rangle \otimes |J\rangle) &:=
    \begin{cases}
        |x+1\rangle \otimes |\mathrm{L}\rangle & (J=\mathrm{R}) \\
        |x-1\rangle \otimes |\mathrm{R}\rangle & (J=\mathrm{L}),
    \end{cases}
\end{align*}
where $|\phi \rangle = \sqrt{p} |L\rangle + \sqrt{q} |R\rangle$.
The free Szegedy walk $Y$ on $\mathbb{Z}$ is 
defined by:
\begin{align*}
      \mathbb{P}_\psi (Y_n=k)
      = \left\Vert (|k\rangle \langle k| \otimes I_2 ) (\mathbf{U}^{(\infty)})^n |\psi\rangle \right\Vert^2, \quad k \in \mathbb{Z} 
    \end{align*}
for $\psi \in \mathcal{H}_\infty$
with $ \Vert \psi \Vert = 1$.
Furthermore, analogous to the free random walk $X$, we associate a zeta function to $ Y $. 
Since the global trace diverges on the infinite graph, we define it via the local return trace to the origin:
\begin{align}\label{TraceDefinition}
    \zeta_Y (u) := \exp \left( \sum_{n=1}^\infty \frac{u^n}{n} \mathrm{Tr}_{\mathbb{C}^2} \langle 0 | (\mathbf{U}^{(\infty)})^n | 0 \rangle \right), \quad |u| < 1.
\end{align}
As will be shown in Theorem \ref{Th3} below, this  definition perfectly recovers the thermodynamic limit of the normalized zeta functions of the finite path graphs:
\begin{align*}
    \lim_{N \to \infty} \bar{\zeta}_{Y^{(N)}} (u) = \zeta_Y(u).
\end{align*}

\subsection{Fourier Transformation of the Free Walks}
We recall some basic facts in the Fourier transformation to apply it our analysis. 
The $ 2 $-dimensional Fourier transform $ \mathcal{F}: l_0 (\mathbb{Z})\otimes \mathbb{C}^2  \to L^2([0,2\pi)) \otimes  \mathbb{C}^2  $ is defined as usual by 
\begin{align*}
   \mathcal{F}(\Psi) (k) \equiv \widehat{\Psi}(k) = \sum_{x \in \mathbb{Z}} e^{-ikx}  \Psi (x), \quad k \in [0, 2\pi), \quad \Psi \in l_0 (\mathbb{Z}) \otimes \mathbb{C}^2, 
\end{align*}
and its inversion is given by
\begin{align*}
    \begin{aligned}
      \mathcal{F}^{-1} (\widehat{\Psi} )(y) =  \int_0^{2\pi} \widehat{\Psi} (k) e^{iky} \frac{dk}{2\pi} 
        = \sum_{x \in \mathbb{Z}} 
       \int_0^{2\pi} e^{-ik (x-y)}
       \Psi (x) \frac{dk}{2\pi} = \Psi (y), \quad y \in \mathbb{Z}.
    \end{aligned}
\end{align*}
Here $ l_0 (\mathbb{Z}) := \{ f : \mathbb{Z} \to \mathbb{C} | \text{$f(z) = 0$ for all but finite $ z $} \} $.  

We extend it to the operators on 
$ l_0 (\mathbb{Z}) \otimes \mathbb{C}^2 $
by the following: for a linear operator $A: l_0 (\mathbb{Z}) \otimes \mathbb{C}^2 \to l_0(\mathbb{Z}) \otimes \mathbb{C}^2 $, 
\begin{align*}
    \begin{aligned}
        \widehat{A} (k) 
        = \sum_{x \in \mathbb{Z}} e^{-ikx} \langle x | A | 0 \rangle, \quad k \in [0,2\pi).
    \end{aligned}
\end{align*}

\begin{lemma}\label{FourLemma}
Suppose that the linear operator $ A: l_0 (\mathbb{Z}) \otimes \mathbb{C}^2 \to l_0(\mathbb{Z}) \otimes \mathbb{C}^2$ is translation-invariant, that is, 
\begin{align}\label{Transinv0}
   A \mathrm{T}_a = \mathrm{T}_a A \quad \iff \quad \mathrm{T}_{-a} A \mathrm{T}_a = A \quad (\forall a \in \mathbb{Z})
\end{align}
where $ \mathrm{T}_a $ is the shift operator defined by 
$$\mathrm{T}_a (|x\rangle \otimes |\psi\rangle) = |x+a\rangle \otimes |\psi\rangle \quad (x, a \in \mathbb{Z}, \, |\psi\rangle \in \mathbb{C}^2). $$
Then, 
for $x, y \in \mathbb{Z}$, 
\begin{align}\label{InvFourier0}
    \langle x | A | y \rangle = \int_0^{2\pi} e^{ik(x-y)} \widehat{A}(k) \frac{dk}{2\pi} \in \mathrm{End}(\mathbb{C}^2).
\end{align}
Moreover, 
for $ \Psi \in l_0 (\mathbb{Z}) \otimes \mathbb{C}^2 $
\begin{align}\label{multi-op}
    \begin{aligned}
       & \widehat{A\Psi}(k) =\widehat{A}(k) \widehat{\Psi}(k), \quad k \in [0,2\pi). 
    \end{aligned}
\end{align}
\end{lemma}
\begin{proof}
By \eqref{Transinv0}, we have 
$$\langle x | A| y \rangle = \langle x-y |A | 0 \rangle=:K(x-y). $$
Thus, the inversion formula \eqref{InvFourier0} is obvious.  
Furthermore, since 
$$(A\Psi)(x) = \sum_{y \in \mathbb{Z}} \langle x | A | y \rangle \Psi(y) = \sum_{y \in \mathbb{Z}} K(x-y) \Psi(y),$$
we have, for $ k \in [0,2\pi) $, 
\begin{align*}
    \begin{aligned}
       & \widehat{(A \Psi)}(k) = \sum_{x \in \mathbb{Z}} e^{-ikx} \left( \sum_{y \in \mathbb{Z}} K(x-y) \Psi(y) \right) \\
 &= \sum_{z \in \mathbb{Z}} \sum_{y \in \mathbb{Z}} e^{-ik(y+z)} K(z) \Psi(y)
  = \left( \sum_{z \in \mathbb{Z}} e^{-ikz} K(z) \right) \left( \sum_{y \in \mathbb{Z}} e^{-iky} \Psi(y) \right) \\
  &   = \widehat{A}(k) \widehat{\Psi}(k).
    \end{aligned}
\end{align*}
\end{proof}

\begin{rem}
    The free Szegedy walk operator $ \mathbf{U}^{(\infty)} $,
    as well as the     
    transition probability $ \mathbf{P}^{(\infty)}  $ of the free random walk, is translation-invariant. 
\end{rem}

\section{Thermodynamical Limits}\label{sec:3}

This section is devoted to the presentation of our first observation: 
the normalized zeta function of the reflected walk on the path graph $ P_N $ converges to the zeta function of the free walk, both in the random walk case and the Szegedy walk case. 
\begin{theorem}\label{M-Th3}
We have
\begin{align}\label{Th1}
\begin{aligned}
    \lim_{N \to \infty} \bar{\zeta}_{X^{(N)}} (u) ^{-1}
    &=\exp \left( \int_0^{2\pi}
    \mathrm{Log} \left( 1- 2 u \sqrt{pq}\cos\theta  \right) \frac{d\theta}{2\pi}\right) \\
    &= \frac{1 + \sqrt{1- 4pqu^2 }}{2}
    \equiv \frac{1 + e^{\frac{1}{2}\mathrm{Log} (1- 4pqu^2)}}{2} \\
    &= \zeta_X (u)^{-1}, 
    \quad |u|<1, 
    \end{aligned}
\end{align}
and the convergence is uniform in every compact set of $ \{ z \in \mathbb{C} | |z| <1 \} $.
\end{theorem}
\begin{proof}
Since (see \cite[Lemma 1]{IKS2012QIP})
\begin{align}\label{spec}
\begin{aligned}
    & \mathrm{Spec} (\mathbf{P}^{(N)}) \\
   & = \left\{ \lambda_0 =1, \lambda_m = 2 \sqrt{pq} \cos \left( \frac{m \pi}{N+1}\right)\,(m=1, \cdots, N), \lambda_{N+1} = -1 
    \right\}, 
    \end{aligned}
\end{align}
we have, for $ |u| < 1 $, 
\begin{align*}
\begin{aligned}
&\bar{\zeta}_{X^{(N)}} (u)^{-1}
     = \exp \bigg( \frac{1}{N+1} \mathrm{Log} (1-u) 
     + \frac{1}{N+1} \mathrm{Log} (1+u) \\
     & \hspace{2cm} + \frac{1}{N+1} \sum_{m=1}^N 
     \mathrm{Log}\left( 1- 2 \sqrt{pq} u \cos \left( \frac{m \pi}{N+1}\right)\right)
     \bigg) \\
     &\longrightarrow \exp \left( \int_0^{2 \pi}
    \mathrm{Log} \left( 1- 2 u \sqrt{pq}\cos\theta  \right)
    \frac{\mathrm{d}\theta}{2 \pi}
    \right) \, (N \to \infty) = \frac{1 + \sqrt{1- 4pqu^2 }}{2}. 
\end{aligned}
\end{align*}
\end{proof}

\begin{cor}\label{recurrence-est}
For $ n =1, 2, \cdots $, 
\begin{align*}
 \lim_{N \to \infty} \frac{\mathrm{Tr} ((\mathbf{P}^{(N)})^n)}{N+1}
 = 
 \begin{cases}
     (pq)^l 
     \binom{2l}{l}
     & n=2l, l=1, 2, \cdots, \\
    0 & n=2l-1, l=1,2, \cdots. 
 \end{cases}
\end{align*}
\end{cor}
\begin{proof}
It follows from the following relations
implied by \eqref{Th1}; 
for $ |u| < 1 $, 
\begin{align*}
    \begin{aligned}
        -\mathrm{Log} \left(\frac{1 + \sqrt{1- 4pqu^2 }}{2}\right) 
       & = \sum_{l=1}^\infty 
        (pq)^l  \binom{2l}{l}
     \frac{u^{2l}}{2l} = \lim_{N \to \infty} \sum_{n=1}^\infty 
     \frac{ \mathrm{Tr} ((\mathbf{P}^{(N)})^n)}{(N+1)n}{u^n}.     
    \end{aligned}
\end{align*}
\end{proof}

\begin{rem}
We note that 
\begin{align*}
    \det \left( I_{N+2} - u \mathbf{P}^{(N)} \right)
    = (1-u^2) u^N (pq)^{N/2} U_N \left((2u \sqrt{pq})^{-1}\right),
\end{align*}
where  $U_N(x)$ is the Chebyshev polynomial of the second kind; 
$$ U_N (x) := 2^N \prod_{m=1}^N \left(x - \cos\frac{m\pi}{N+1}\right). $$
\end{rem}

For the quantum walk, we have the following.
\begin{theorem}\label{Th3}
The sequence of the normalized zeta functions $\{ \bar{\zeta}_{Y^{(N)}} (u) \}_{N \in \mathbb{N}}$ for the reflected Szegedy walks on $P_N$ converges to the zeta function $\zeta_Y(u)$ of the free Szegedy walk on $\mathbb{Z}$ as $N \to \infty$:
\begin{align*}
    \lim_{N \to \infty} \bar{\zeta}_{Y^{(N)}} (u) = \zeta_Y(u), \quad |u| < 1. 
\end{align*}
The limit admits the following integral representation:
\begin{align}\label{QWlimzeta_integral}
    \zeta_Y (u)^{-1} = \exp \left( \int_0^{2 \pi} \mathrm{Log} \left( 1- 4 u \sqrt{pq}\cos\theta + u^2 \right) \frac{\mathrm{d}\theta}{2 \pi} \right).
\end{align}
The convergence is uniform on every compact set in $\{ u \in \mathbb{C} : |u| < 1 \}$.
\end{theorem}

\begin{proof}
The proof consists of two parts: establishing the integral limit of the finite determinants and proving the Fourier trace identity for the free walk.

\textit{Part 1: Convergence of the normalized determinants.}
Since (see \cite[Lemma 2]{IKS2012QIP}), 
\begin{align*}
    \begin{aligned}
        & \mathrm{Spec} (\mathbf{U}^{(N)}) 
        = \left\{ \mu_0 =1, \mu_{\pm m} = e^{\pm i \varphi_m} 
\,(m=1, \cdots, N), \mu_{N+1} = -1 
    \right\}, 
    \end{aligned}
\end{align*}
where 
\begin{align*}
    \varphi_m := \arccos \left(2 \sqrt{pq} \cos \left(\frac{m\pi}{N+1} \right)\right), m=1, \cdots, N, 
\end{align*}
the normalized zeta function satisfies:
\begin{align}\label{Formula1}
 \begin{aligned}
  &  \bar{\zeta}_{Y^{(N)}} (u)^{-1} = \exp \left( \frac{1}{N+1} \sum_{\mu \in \mathrm{Spec}(\mathbf{U}^{(N)})} \mathrm{Log}(1 - \mu u) \right) \\
    &= \exp \bigg( \frac{\mathrm{Log}(1-u) + \mathrm{Log}(1+u)}{N+1} + \frac{1}{N+1} \sum_{m=1}^N \mathrm{Log}(1 - 2u \cos \varphi_m + u^2) \bigg).
 \end{aligned}
\end{align}
As $N \to \infty$, the boundary terms $\frac{1}{N+1}\mathrm{Log}(1 \pm u)$ vanish. The summation term is a Riemann sum for the following integral:
\begin{align*}
    \lim_{N \to \infty} \frac{1}{N+1} \sum_{m=1}^N \mathrm{Log}(1 - 2u \cos \varphi_m + u^2) 
    &= \int_0^1 \mathrm{Log}(1 - 2u \cos(\varphi(\theta)) + u^2) d\theta \\
    &= \int_0^\pi \mathrm{Log}(1 - 4\sqrt{pq}u \cos\theta + u^2) \frac{d\theta}{\pi},
\end{align*}
where we used $\cos(\varphi(\theta)) = 2\sqrt{pq}\cos(\pi\theta)$. Since the integrand is symmetric over $[0, 2\pi]$, we obtain \eqref{QWlimzeta_integral}.

\textit{Part 2: The Inverse Fourier Transform Identity.}
We now prove that the integral in \eqref{QWlimzeta_integral} coincides with the definition \eqref{TraceDefinition} of $\zeta_Y(u)$. 

We first note that the free Szegedy walk operator $\mathbf{U}^{(\infty)}$ is translation-invariant, 
and, for any $n \in \mathbb{N}$, the operator $(\mathbf{U}^{(\infty)})^n$ is also translation-invariant.
Then, by \eqref{InvFourier0} and \eqref{multi-op} of Lemma \ref{FourLemma}, 
\begin{align*}
    \langle x | (\mathbf{U}^{(\infty)})^n | y \rangle = \int_0^{2\pi} e^{ik(x-y)} \widehat{\mathbf{U}^{(\infty)}}(k)^n \frac{dk}{2\pi}, \quad x, y \in \mathbb{Z}.
\end{align*}
Setting $x=y=0$ and taking the trace $\mathrm{Tr}_{\mathbb{C}^2}$ on both sides and using the linearity of the integral:
\begin{align*}
\mathrm{Tr}_{\mathbb{C}^2} \langle 0 | (\mathbf{U}^{(\infty)})^n | 0 \rangle = \int_0^{2\pi} \mathrm{Tr}_{\mathbb{C}^2} (\widehat{\mathbf{U}^{(\infty)}}(k)^n) \frac{dk}{2\pi}.
\end{align*}
Then, 
\begin{align}\label{LogZetaY}
\begin{aligned}
    \mathrm{Log}\, \zeta_Y(u)^{-1} &= -\sum_{n=1}^\infty \frac{u^n}{n} \int_0^{2\pi} \mathrm{Tr}_{\mathbb{C}^2} (\widehat{\mathbf{U}^{(\infty)}}(k)^n) \frac{dk}{2\pi} \\
    &=- \int_0^{2\pi} \left( \sum_{n=1}^\infty \frac{u^n}{n} \mathrm{Tr}_{\mathbb{C}^2} (\widehat{\mathbf{U}^{(\infty)}}(k)^n) \right) \frac{dk}{2\pi} \\
    &=  \int_0^{2\pi} \mathrm{Log} \det(I_2 - u \hat{\mathbf{U}}(k)) \frac{dk}{2\pi}.
    \end{aligned}
\end{align}
Since 
\begin{align*}
    \hat{\mathbf{U}}(k) = \begin{pmatrix} 0 & e^{ik} \\ e^{-ik} & 0 \end{pmatrix} \begin{pmatrix} p-q & 2\sqrt{pq} \\ 2\sqrt{pq} & q-p \end{pmatrix} = \begin{pmatrix} 2\sqrt{pq}e^{ik} & (q-p)e^{ik} \\ (p-q)e^{-ik} & 2\sqrt{pq}e^{-ik} \end{pmatrix},
\end{align*}
we have 
\begin{align}\label{detUk}
    \det(I_2 - u\widehat{\mathbf{U}^{(\infty)}}(k)) = 1 - 4\sqrt{pq}u\cos k + u^2.
\end{align}
Combining \eqref{detUk} with \eqref{LogZetaY}, we obtain \eqref{QWlimzeta_integral}.
\end{proof}

\begin{cor}
For $ n =1, 2, \cdots $, 
\begin{align}\label{TracesQW}
 \lim_{N \to \infty} \frac{\mathrm{Tr} ((\mathbf{U}^{(N)})^n)}{N+1}
= 2 
\int_0^{2\pi} T_n ( 2 \sqrt{pq} \cos \theta ) \frac{\mathrm{d}\theta}{2 \pi}, 
\end{align}
where 
$ T_n, n=1, 2, \cdots $ are the Chebyshev polynomials of the first kind; polynomials defined by the relation 
$ T_n (\cos \theta) = \cos (n \theta) $. 
\end{cor}
\begin{proof}
    By expanding the functions of the form $ \log (1- z) $
    in power series in $ z $
    in \eqref{Formula1}, 
    we obtain 
    \begin{align*}
        \begin{aligned}
\frac{\mathrm{Tr} ((\mathbf{U}^{(N)})^n)}{N+1}
    = \frac{1}{N+1}
    \bigg(
    1 + (-1)^n + 2 \sum_{m=1}^N \cos (n \varphi_m)
    \bigg), 
        \end{aligned}
    \end{align*}
which leads to the relation 
    \eqref{TracesQW}. 
\end{proof}

\begin{rem}
Note that 
we can directly retrieve the ``Konno-Sato correspondence" between $ \mathbf{U}^{(N)}$ and $ \mathbf{P}^{(N)} $:
\begin{align}\label{KSforUNPN}
    \det (I_{2N+2} - u \mathbf{U}^{(N)})^{-1}
    &= \frac{1-u^2}{(1+u^2)^{N+2}} \det \left( I_{N+2} - \frac{2u}{1+u^2} \mathbf{P}^{(N)} \right)^{-1}.
\end{align}
In fact, by the explicit expression of $
 \mathrm{Spec} (\mathbf{P}^{(N)})$
 and $ \mathrm{Spec} (\mathbf{U}^{(N)})$, we have 
\begin{align*}
    \det \left( I_{N+2} - u \mathbf{P}^{(N)} \right)
    = (1-u^2) \prod_{m=1}^N \left(1- 2u \sqrt{pq} \cos \frac{m \pi}{N+1} \right)
\end{align*}
and 
\begin{align}\label{explicitUN}
    \det (I_{2N+2} - u \mathbf{U}^{(N)})
    =  (1-u^2) \prod_{m=1}^N \left(1-4 u \sqrt{pq} \cos \frac{m \pi}{N+1} +u^2 \right),
\end{align}
establishing \eqref{KSforUNPN}. 
Furthermore, via the expressions \eqref{Th1} and
\eqref{QWlimzeta_integral}, we also have a Konno-Sato type correspondence in the limit:
\begin{align*}
    \zeta_{Y}^{-1} (u) = (1+u^2) \zeta_X \left( 2(u+\frac{1}{u})^{-1} \right)^{-1}.
\end{align*}
\end{rem}

\section{
Generalized Reflections for Reflected Random Walks}\label{sec:4}

In this section, we express the transition probability $\mathbf{P}^{(N)}$ of the reflected random walk $X^{(N)}$ on $P_N$ using ``reflections" 
derived from the free random walk on $\mathbb{Z}$.

Let $ \mathsf{M} \equiv (m_{i,j}: i, j \in \mathbb{Z} ) $
be the infinite matrix given by
\begin{align}\label{matrixM}
   m_{i,j}
    = \begin{cases}
        (\frac{p}{q})^{k-1} & i=-j=k,  k \in \mathbb{N}, \\
        p^{(j+1)_-} q^{j-1} (q-p) & i+j =2k, j \leq 0, k \in \mathbb{N},  \\
        1 & i=j=0 \\
         (\frac{q}{p})^{k-1} & -i=j=k,  k \in \mathbb{N}, \\
           q^{(j-1)_+} p^{-j-1} (p-q) & -i-j =2k, j \geq 0, k \in \mathbb{N}, \\
        0 & \text{otherwise}.
    \end{cases}
\end{align}
Here we write
$a_+:=\max\{a,0\}$, and $ a_-:=\max\{-a,0\}$.
We define $ \mathcal{M}_K : \mathbb{R}^\mathbb{Z} \to \mathbb{R}^\mathbb{Z} $ for each $ K \in \mathbb{Z} $ by
\begin{align}\label{MKdef}
    \begin{aligned}
        \mathcal{M}_{K} f (K+y) = \sum_{z\in \mathbb{Z}} m_{y,z} f (K +z). 
    \end{aligned}
\end{align}
We note that 
for each $ y $, $ m_{y,z} =0 $ for all but finite $ z $  so that the operator is well-defined. 
Their adjoint operators $ \mathcal{M}_K^*$, $ K \in \mathbb{Z} $ are defined on $ \mathbb{R}^{(\mathbb{Z})} $ by 
\begin{align*}
    \begin{aligned}
        \mathcal{M}^*_{K} f (K+y) = \sum_{z\in \mathbb{Z}} m_{z,y} f (K +z). 
    \end{aligned}
\end{align*}
Here, $\mathbb{R}^{(\mathbb{Z})}$ denotes the space of real sequences with finite support (i.e., the set of functions $g: \mathbb{Z} \to \mathbb{R}$ such that $g(x) = 0$ for all but finitely many $x \in \mathbb{Z}$). 
From the definition, we have 
\begin{align*}
     \langle \mathcal{M}_K f,g \rangle_K := \sum_{z \in \mathbf{Z}} \mathcal{M}_K f(K+z) g (K+z) =  
    \langle  f,\mathcal{M}^*_K g \rangle, \quad f \in \mathbb{R}^\mathbb{Z}, g \in \mathbb{R}^{(\mathbb{Z})}.
\end{align*}

\begin{lemma}
For any $ K \in \mathbb{Z} $, the operator $ \mathcal{M}_K $ and $ \mathcal{M}^*_K $ are involutions on $ \mathbb{R}^{\mathbb{Z}}$
   and $\mathbb{R}^{(\mathbb{Z})}$, i.e., $ \mathcal{M}_K^2 = id_{\mathbb{R}^{\mathbb{Z}}} $ and $ (\mathcal{M}^*_K)^2 = id_{\mathbb{R}^{(\mathbb{Z})}} $.
\end{lemma} 

\begin{proof}
The case $ p=q =1/2 $, it is trivial. We concentrate on the case where $ p \ne q $. 
By the definition \eqref{MKdef},
we see that $ \mathsf{M}^2 = I $ implies 
$\mathcal{M}_K^2 = \text{id}_{\mathbb{R}^\mathbb{Z}}$ for all $ K \in \mathbb{Z} $. 

Let us partition the index set $\mathbb{Z}$ into three sets: $\mathbb{Z}_{++} = \{1, 2, \dots\}$, $\mathbb{Z}_0 = \{0\}$, and $\mathbb{Z}_{--} = \{-1, -2, \dots\}$. We analyze the non-zero entries of $\mathsf{M} = (m_{i,j})$:
\begin{itemize}
    \item {Row $i=0$}: $m_{0,0} = 1$. For $j \neq 0$,  $m_{0,j} = 0$. 
    \item {Rows $i \in \mathbb{Z}_{++}$}: For $m_{i,j}$ to be non-zero, we must be in Case 1 ($j = -i < 0$) or Case 2 ($i+j = 2k > 0$ and $j \le 0$). Thus, $m_{i,j} = 0$ for all $j > 0$.
    \item {Rows $i \in \mathbb{Z}_{--}$}: For $m_{i,j}$ to be non-zero, we must be in Case 4 ($j = -i > 0$) or Case 5 ($-i-j = 2k > 0$ and $j \ge 0$). Thus, $m_{i,j} = 0$ for all $j < 0$.
\end{itemize}
This forces $\mathsf{M}$ to have a block structure:
\begin{align}\label{Def-Block-M}
\mathsf{M} = \begin{pmatrix} 0 & A & B \\ 0 & 1 & 0 \\ C & D & 0 \end{pmatrix} 
\end{align}
where the blocks correspond to $\mathbb{Z}_{++}, \mathbb{Z}_0, \mathbb{Z}_{--}$. 
Squaring this matrix yields:
\[ 
\mathsf{M}^2 = \begin{pmatrix} 0 & A & B \\ 0 & 1 & 0 \\ C & D & 0 \end{pmatrix}^2 = \begin{pmatrix} BC & A+BD & 0 \\ 0 & 1 & 0 \\ 0 & CA+D & CB \end{pmatrix} 
\]
To prove $ \mathsf{M}^2 = I$, we must show:
\begin{enumerate}
    \item $BC = I_{\mathbb{Z}_+}$
    \item $A + BD = 0$
\end{enumerate}
The identities $CB = I_{\mathbb{Z}_-}$ and $CA+D=0$ follow by the symmetry under exchanging $p \leftrightarrow q$ and $\mathbb{Z}_{++} \leftrightarrow \mathbb{Z}_{--}$.

Let $x, y \in \mathbb{Z}_{++}$. 
The block $B$ maps $\mathbb{Z}_{--}$ to $\mathbb{Z}_{++}$. 
Its entries are $B_{x,y} = m_{x, -y}$. 
\begin{itemize}
    \item If $x = y$, Case 1 gives $B_{x,x} = (p/q)^{x-1}$.
    \item If $x > y$ and $x \equiv y \pmod 2$, Case 2 gives $B_{x,y} = p^{(-y+1)_-} q^{-y-1} (q-p) = p^{y-1} q^{-y-1} (q-p)$.
    \item Otherwise, $B_{x,y} = 0$.
\end{itemize}
The block $C$ maps $\mathbb{Z}_{++}$ to $\mathbb{Z}_{--} $. Its entries are $C_{x,y} = m_{-x, y}$. 
\begin{itemize}
    \item If $x = y$, Case 4 gives $C_{x,x} = (q/p)^{x-1}$.
    \item If $x > y$ and $x \equiv y \pmod 2$, Case 5 gives $C_{x,y} = q^{(y-1)_+} p^{-y-1} (p-q) = q^{y-1} p^{-y-1} (p-q)$.
    \item Otherwise, $C_{x,y} = 0$.
\end{itemize}
Since $B, C$ are lower triangular, $(BC)_{x,y} = \sum_{z=y}^x B_{x,z} C_{z,y}$.
For $x = y$:
\[ 
(BC)_{x,x} = B_{x,x} C_{x,x} = (p/q)^{x-1} (q/p)^{x-1} = 1. 
\]
For $x > y$ and $x \not\equiv y \pmod 2$, we trivially have $(BC)_{x,y} = 0$.
For $x > y$ with $x - y = 2m \ge 2$, the sum has indices $z = y, y+2, \dots, x$. 
The first term ($z=y$) becomes 
\[ 
B_{x,y} C_{y,y} = p^{y-1} q^{-y-1} (q-p) (q/p)^{y-1} = q^{-2}(q-p). 
\]
The last term ($z=x$) is given by
\[ 
B_{x,x} C_{x,y} = (p/q)^{x-1} q^{y-1} p^{-y-1} (p-q) = p^{x-y-2} q^{-x+y} (p-q). 
\]
The middle terms ($z = y+2k$ for $1 \le k \le m-1$) are obtained as 
\begin{align*}
B_{x,z} C_{z,y} &= \left( p^{z-1} q^{-z-1} (q-p) \right) \left( q^{y-1} p^{-y-1} (p-q) \right) \\
&= - (p-q)^2 p^{z-y-2} q^{y-z-2}.
\end{align*}
Summing the middle terms gives,
by writing $ \alpha = p/q $, 
\[ 
S_{\text{mid}}:=\sum_{k=1}^{m-1} B_{x,y+2k} C_{y+2k,y} = -(p-q)^2 q^{-4} \sum_{k=1}^{m-1} (\alpha^2)^{k-1} = -(p-q)^2 q^{-4} \frac{1 - \alpha^{2m-2}}{1 - \alpha^2}. 
\]
Using the condition $p+q=1$, we simplify the denominator: $1 - \alpha^2 = \frac{q^2 - p^2}{q^2} = \frac{q-p}{q^2}$. Thus
\[ 
S_{\text{mid}} = -(p-q)^2 q^{-4} \frac{q^2}{q-p} (1 - \alpha^{2m-2}) = q^{-2}(p-q) (1 - \alpha^{2m-2}). 
\]
Adding the first, middle, and last terms:
\begin{align*}
(BC)_{x,y} &= q^{-2}(q-p) + q^{-2}(p-q) - q^{-2}(p-q)\alpha ^{2m-2} + \alpha^{2m-2} q^{-2}(p-q) \\
&= q^{-2} (q-p + p-q) = 0.
\end{align*}
Thus, $BC = I_{\mathbb{Z}_+}$.

Next, we prove that $A + BD = 0$, 
For $x, y \in \mathbb{Z}_+$:
\begin{itemize}
    \item $A_x = m_{x, 0} = q^{-1}(q-p)$ if $x$ is even, and $0$ if odd.
    \item $D_y = m_{-y, 0} = p^{-1}(p-q)$ if $y$ is even, and $0$ if odd.
\end{itemize}
We compute $(BD)_x = \sum_{y \ge 1} B_{x,y} D_y$. If $x$ is odd, $B_{x,y}$ requires $y$ to be odd, but $D_y = 0$ for odd $y$. 
Hence $(BD)_x = 0 = -A_x$.
If $x = 2n$, the non-zero terms occur for $y = 2m$ where $1 \le m \le n$. 
For $m < n$:
\[ 
B_{2n, 2m} D_{2m} = \left( p^{2m-1} q^{-2m-1} (q-p) \right) \left( p^{-1}(p-q) \right) = -q^{-3} (p-q)^2 \alpha^{2m-2}. 
\]
For $m = n$:
\[ 
B_{2n, 2n} D_{2n} = \alpha^{2n-1} p^{-1} (p-q) = q^{-1} (p-q) \alpha^{2n-2}. 
\]
Summing the terms for $m < n$ gives a geometric series identical in structure to Step 3:
\[ 
\sum_{m=1}^{n-1} -q^{-3} (p-q)^2 \alpha^{2m-2} = -q^{-3} (p-q)^2 \frac{1 - \alpha^{2n-2}}{1 - \alpha^2} = q^{-1}(p-q) (1 - \alpha^{2n-2}). 
\]
Adding the $m=n$ term, the $\alpha^{2n-2} $ dependencies cancel:
\[ 
(BD)_{2n} = q^{-1}(p-q) (1 - \alpha^{2n-2}) + q^{-1}(p-q) \alpha^{2n-2} = q^{-1}(p-q). 
\]
Since $A_{2n} = q^{-1}(q-p)$, we immediately get:
\[ 
A_{2n} + (BD)_{2n} = q^{-1}(q-p) + q^{-1}(p-q) = 0. 
\]
This completes the proof that $\mathsf{M}^2 = I$.

Furthermore, since $ \mathcal{M}_K^* $ is the formal adjoint of $ \mathcal{M}_K $ under the bilinear pairing $ \langle \cdot, \cdot \rangle_K $, it automatically inherits the involution property on the dual space $ \mathbb{R}^{(\mathbb{Z})} $, yielding $ (\mathcal{M}_K^*)^2 = id_{\mathbb{R}^{(\mathbb{Z})}} $.
\end{proof}

\section{The Trace Formula for the Reflected Random Walk}\label{sec:5}

In this section, we construct an exact spatial trace formula for the Markov transition matrix of the reflected random walk. 
By expanding the transition probabilities over the discrete reflection group $\mathcal{G}_{0,N+1}$,
the group generated by $ \mathcal{M}_0$ and $ \mathcal{M}_{N+1} $, 
we obtain a product representation for the global zeta function. 
Furthermore, we demonstrate that the diffusive scaling limit of this
discrete formula naturally recovers the classical Poisson summation formula
in the symmetric case; the asymmetric case is taken up in Section
\ref{sec:8}, where the drift is tuned with the system size and produces an
effective mass rather than a diffusive limit.

\subsection{A Generalized Reflection Principle for the Reflected Random Walk}

The following is a generalization of the so-called reflection principle.
\begin{proposition}\label{TraceFormula-P}
(i) For any $ f \in  \mathbb{R}^{\mathbb{Z}} $
such that  $ f(y) = 0 $ for $ y \notin [1,N] $, 
\begin{align}\label{Tr1}
    \mathbb{E} [f(X^{(N)}_t) |X^{(N)}_0=x ]
    = \sum_{W \in \mathcal{G}_{0,N+1}} \mathbb{E} [ (W f)(X_t) | X_0 =x] 
\end{align}
for $ t \in \mathbb{Z}_+ $, and $ x \in \mathbb{Z} \cap [0,N+1] $.

(ii) For $ x, y \in \mathbb{Z} \cap [1,N]$ and $ t \in \mathbb{Z}_+ $, 
\begin{align}\label{TraceFormula00}
    \mathbb{P} (X^{(N)}_t =y \,|X^{(N)}_0=x) 
    = \sum_{W \in \mathcal{G}_{0,N+1}} W^* \mathbb{P} (X_t = \cdot \,\,|X_0=x) (y) 
\end{align}
\end{proposition}

\begin{proof}
(i) For a function $f: \mathbb{Z} \to \mathbb{R}$ supported on $[1, N]$, define the extended function $\tilde{f} : \mathbb{Z} \to \mathbb{R}$ by
\begin{align*}
    \tilde{f}(x) = \sum_{W \in \mathcal{G}_{0,N+1}} (W f)(x).
\end{align*}

First, we show that this sum is well-defined and evaluates to a finite number of non-zero terms for any $x \in \mathbb{Z}$. 
From the definition of $\mathsf{M}$, the matrix elements $m_{y,z}$ govern the reflection $\mathcal{M}_0$. For $f$ supported on $[1,N]$, applying $\mathcal{M}_0$ yields $\mathcal{M}_0 f(y) = \sum_{z=1}^N m_{y,z} f(z)$. The coefficients $m_{y,z}$ for $z \ge 1$ are non-zero only when $y \le -1$ (since $m_{0,z} = 0$ for $z \neq 0$). Thus, $\mathrm{supp}(\mathcal{M}_0 f) \subseteq (-\infty, -1]$. 
Symmetrically, $\mathrm{supp}(\mathcal{M}_{N+1} f) \subseteq [N+2, \infty)$. 
Applying alternating words $W \in \mathcal{G}_{0,N+1}$ of length $\ge 2$ strictly pushes the support further away from the interval $[0, N+1]$. Therefore $\tilde{f}(x)$ is a well-defined finite sum.

Since the group $\mathcal{G}_{0,N+1}$ is generated by the involutions $\mathcal{M}_0$ and $\mathcal{M}_{N+1}$, the function $\tilde{f}$ is invariant under both reflections:
\begin{align*}
    \mathcal{M}_0 \tilde{f} = \tilde{f} \quad \text{and} \quad \mathcal{M}_{N+1} \tilde{f} = \tilde{f}.
\end{align*}

Furthermore, since $(Wf)(x) = 0$ for all $W \neq id$ when $x \in [1, N]$, we have $\tilde{f}(x) = f(x)$ for $x \in [1, N]$. At the boundaries $x \in \{0, N+1\}$, we have $f(x) = 0$, and as verified above, $(Wf)(x) = 0$ for all $W$. Thus, $\tilde{f}(x) = f(x) = 0$ holds on $[0, N+1]$.

We proceed by induction on $t$ to prove \eqref{Tr1}. 
For $t=0$, we have $$\mathbb{E}[f(X^{(N)}_0)|X^{(N)}_0=x] = f(x). $$
Since $\tilde{f}(x) = f(x)$ for $x \in [0, N+1]$, the base case holds.

Suppose the claim holds for all $k < t$ and all $x \in [0, N+1]$.
For the interior points $x \in [1, N]$, let $\tau = \inf\{ s > 0 : X_s \in \{0, N+1\} \}$
be the first hitting time of the boundary for the free walk $X_s$. Since the reflected walk $X^{(N)}_s$ coincides with the free walk up until it hits a boundary, $\tau$ is also the first boundary hitting time for $X^{(N)}_s$. 
Moreover, since $X_s \in [1, N]$ for $ s < \tau $, we have $f(X_t) = \tilde{f}(X_t)$ on $\{\tau > t\}$. 
Thus, 
\begin{align}\label{fstdc}
    \mathbb{E}[f(X^{(N)}_t) | X_0=x] = \mathbb{E}[\tilde{f}(X_t) 1_{\{\tau > t\}} | X_0=x] + \mathbb{E}[f(X^{(N)}_t) 1_{\{\tau \le t\}} | X_0=x].
\end{align}
Furthermore, we have 
\begin{align}\label{summ}
\begin{aligned}
    & \mathbb{E}[f(X^{(N)}_t) 1_{\{\tau \le t\}} | X_0=x] \\
    &= \sum_{s=1}^t \sum_{y \in \{0, N+1\}} \mathbb{P}(\tau=s, X_s=y | X_0=x) \, \mathbb{E}[f(X^{(N)}_t) | \tau=s, X_s^{(N)}=y].
    \end{aligned}
\end{align}
By the Markov property of $X^{(N)}$, 
\begin{align*}
    \mathbb{E}[f(X^{(N)}_t) | \tau=s, X_s^{(N)}=y] = \mathbb{E}[f(X^{(N)}_{t-s}) | X^{(N)}_0=y].
\end{align*}
Since $t-s < t$, we can apply the induction hypothesis:
\begin{align*}
    \mathbb{E}[f(X^{(N)}_{t-s}) | X^{(N)}_0=y] = \mathbb{E}[\tilde{f}(X_{t-s}) | X_0=y].
\end{align*}
Substituting this back into the summation of \eqref{summ} yields:
\begin{align}\label{scddc}
\begin{aligned}
    & \mathbb{E}[f(X^{(N)}_t) 1_{\{\tau \le t\}} | X_0=x] \\
    & =\sum_{s=1}^t \sum_{y \in \{0, N+1\}} \mathbb{P}(\tau=s, X_s=y | X_0=x) \, \mathbb{E}[\tilde{f}(X_{t-s}) | X_0=y] \\
    &= \mathbb{E}[\tilde{f}(X_t) 1_{\{\tau \le t\}} | X_0=x],
    \end{aligned}
\end{align}
where the second equality is guaranteed directly by the strong Markov property of the free walk $X_t$.
The equations \eqref{fstdc} and \eqref{scddc} 
lead to \eqref{Tr1}. 

We also need to show \eqref{Tr1}
for $ x =0 $ and $ x =N+1 $. Let $x=0 $.
Then, 
\begin{align*}
    \begin{aligned}
        \mathbb{E}[f(X^{(N)}_t) \mid X_0^{(N)}=0] = \mathbb{E}[f(X^{(N)}_{t-1}) \mid X_0^{(N)}=1] = (P^{t-1}\tilde{f})(1)
    \end{aligned}
\end{align*}
by the induction hypothesis, where 
$ P $ is the operator on $ \mathbb{R}^\mathbb{Z} $ defined by
$(Pf)(x) = p f(x-1) + q f(x+1)$.
On the free walk side, 
\begin{align*}
    \begin{aligned}
        \mathbb{E}[\tilde{f}(X_t) \mid X_0=0] = q(P^{t-1}\tilde{f})(1) + p(P^{t-1}\tilde{f})(-1).
    \end{aligned}
\end{align*}
So it remains to prove 
\begin{align}\label{PP}
(P^{t-1}\tilde{f})(-1) = (P^{t-1}\tilde{f})(1).
\end{align}

Let $V_{\mathrm{inv}} = \{ f \in \mathbb{R}^{\mathbb{Z}} \mid \mathcal{M}_0 f = f \}$. 
Our goal is to show that $P(V_{\mathrm{inv}}) \subseteq V_{\mathrm{inv}}$
since for any $f \in V_{\mathrm{inv}}$, 
\begin{equation*}
    f(1) = (\mathcal{M}_0 f)(1) = \sum_{z \in \mathbb{Z}} m_{1,z} f(z) = f(-1).
\end{equation*}
To prove that $Pf \in V_{\mathrm{inv}}$, we must verify that $(\mathcal{M}_0 (Pf))(x) = (Pf)(x)$ for all $x \in \mathbb{Z}$. 
Since the matrix $\mathsf{M}$ defining $\mathcal{M}_0$ is an involution ($\mathcal{M}_0^2 = id$) that maps positive coordinates to negative coordinates and vice versa (see \eqref{Def-Block-M}), verifying this equality for $x > 0$ is sufficient to guarantee it holds for all $x \in \mathbb{Z}$.
By definition, for $x > 0$, the left-hand side is:
\begin{equation*}
    (\mathcal{M}_0 (Pf))(x) = \sum_{z \in \mathbb{Z}} m_{x,z} (Pf)(z) = \sum_{z \in \mathbb{Z}} m_{x,z} (p f(z-1) + q f(z+1)).
\end{equation*}
The right-hand side is:
\begin{equation*}
    (Pf)(x) = p f(x-1) + q f(x+1) = p \sum_{z \in \mathbb{Z}} m_{x-1,z} f(z) + q \sum_{z \in \mathbb{Z}} m_{x+1,z} f(z).
\end{equation*}

By matching the coefficients of $f(z)$ for $z \le 0$ and applying the property $f(1) = f(-1)$, the condition $(\mathcal{M}_0 (Pf))(x) = (Pf)(x)$ reduces to three algebraic identities for the matrix elements $m_{i,j}$:
\begin{enumerate}
    \item For $z \le -2$: $p \, m_{x-1, z} + q \, m_{x+1, z} = p \, m_{x, z+1} + q \, m_{x, z-1}$.
    \item For $z = -1$: $p \, m_{x-1, -1} + q \, m_{x+1, -1} = m_{x, 0} + q \, m_{x, -2}$. The shift in the coefficient of $m_{x,0}$ to $1$ occurs because $(Pf)(0) = p f(-1) + q f(1) = (p+q)f(-1) = f(-1)$.
    \item For $z = 0$: $p \, m_{x-1, 0} + q \, m_{x+1, 0} = q \, m_{x, -1}$.
\end{enumerate}

To prove the three algebraic identities for the matrix elements $m_{i,j}$ given $i > 0$, we first note that $m_{i,j}$ is non-zero only if $i=-j$ or $i+j$ is an even positive integer. 
Thus, both sides of the identities are trivially zero unless $x+z$ is odd. 
We assume $x+z$ is odd throughout.
For $i > 0$ and $j \le -1 $, we define the constant part for $i+j > 0$ as $ \mathfrak{m}(j) := p^{-j-1} q^{j-1} (q-p)$. 
\begin{enumerate}
    \item For $z \le -2$: 
    We prove $p \, m_{x-1, z} + q \, m_{x+1, z} = p \, m_{x, z+1} + q \, m_{x, z-1}$. 
    \begin{itemize}
        \item {Case 1a ($x-1+z > 0$)}: All index sums are strictly positive. 
        The left-hand side (LHS) is $p \mathfrak{m} (z) + q \mathfrak{m}(z) = (p+q) \mathfrak{m}(z) = \mathfrak{m}(z)$. 
        The right-hand side (RHS) expands as:
        \begin{align*}
            p \mathfrak{m} (z+1) + q \mathfrak{m}(z-1) &= p (p^{-z-2} q^z (q-p)) + q (p^{-z} q^{z-2} (q-p)) \\
            &= p^{-z-1} q^z (q-p) + p^{-z} q^{z-1} (q-p) \\
            &= p^{-z-1} q^{z-1} (q-p) [q + p] = \mathfrak{m} (z).
        \end{align*}
        LHS = RHS holds.
        \item {Case 1b ($x-1+z = 0$, meaning $x = 1-z$)}: 
        The LHS involves the anti-diagonal element $m_{-z, z} = (p/q)^{-z-1}$:
        \begin{align*}
            \text{LHS} &= p \, m_{-z, z} + q \, m_{-z+2, z} = p \left(\frac{p}{q}\right)^{-z-1} + q \mathfrak{m} (z) \\
            &= p^{-z} q^{z+1} + p^{-z-1} q^z (q-p) = p^{-z-1} q^z (pq + q - p).
        \end{align*}
        Substituting $q = 1-p$, this yields $p^{-z-1} q^z (1-p-p^2)$.
        The RHS expands as:
        \begin{align*}
            \text{RHS} &= p \, m_{1-z, z+1} + q \, m_{1-z, z-1} = p \mathfrak{m} (z+1) + q \left(\frac{p}{q}\right)^{-z} \\
            &= p^{-z-1} q^z (q-p) + p^{-z} q^{z+1} = p^{-z-1} q^z (q-p + pq).
        \end{align*}
        This also simplifies to $p^{-z-1} q^z (1-p-p^2)$. LHS = RHS holds.
    \end{itemize}
    \item {For $z = -1$}: 
    We prove $p \, m_{x-1, -1} + q \, m_{x+1, -1} = m_{x, 0} + q \, m_{x, -2}$. Since $x+z$ is odd, $x$ must be even ($x \ge 2$).
    \begin{itemize}
        \item {Case 2a ($x \ge 4$)}: 
        $\text{LHS} = p \mathfrak{m} (-1) + q \mathfrak{m}(-1) = \mathfrak{m}(-1) = p^0 q^{-2} (q-p) = \frac{q-p}{q^2}$.
        Using $m_{x,0} = \frac{q-p}{q}$, the RHS is:
        \begin{align*}
            \text{RHS} &= \frac{q-p}{q} + q \mathfrak{m}(-2) = \frac{q-p}{q} + q \left( p^1 q^{-3} (q-p) \right) \\
            &= \frac{q(q-p) + p(q-p)}{q^2} = \frac{(p+q)(q-p)}{q^2} = \frac{q-p}{q^2}.
        \end{align*}
        LHS = RHS holds.
        \item {Case 2b ($x = 2$)}: 
        Using $m_{1,-1} = 1$, the LHS is:
        \begin{align*}
            \text{LHS} &= p(1) + q \mathfrak{m}(-1) = p + \frac{q-p}{q} = \frac{pq + q - p}{q} = \frac{1-p-p^2}{q}.
        \end{align*}
        Using $m_{2,-2} = p/q$, the RHS is:
        \begin{align*}
            \text{RHS} &= m_{2,0} + q \, m_{2,-2} = \frac{q-p}{q} + q \left(\frac{p}{q}\right) = \frac{q-p+pq}{q} = \frac{1-p-p^2}{q}.
        \end{align*}
        LHS = RHS holds.
    \end{itemize}
    \item {For $z = 0$}: 
    We prove $p \, m_{x-1, 0} + q \, m_{x+1, 0} = q \, m_{x, -1}$. Since $x+z$ is odd, $x$ must be odd ($x \ge 1$).
    \begin{itemize}
        \item {Case 3a ($x \ge 3$)}: 
        $\text{LHS} = p \left(\frac{q-p}{q}\right) + q \left(\frac{q-p}{q}\right) = (p+q)\frac{q-p}{q} = \frac{q-p}{q}$.
        $\text{RHS} = q \mathfrak{m}(-1) = q \left(\frac{q-p}{q^2}\right) = \frac{q-p}{q}$. LHS = RHS holds.
        \item {Case 3b ($x = 1$)}: 
        Using $m_{0,0} = 1$, the LHS is:
        \begin{align*}
            \text{LHS} &= p(1) + q \, m_{2,0} = p + q \left(\frac{q-p}{q}\right) = p + q - p = q.
        \end{align*}
        The RHS is $q \, m_{1,-1} = q(1) = q$. LHS = RHS holds.
    \end{itemize}
\end{enumerate}
Thus we have shown that 
$(\mathcal{M}_0 (Pf))(x) = (Pf)(x)$, which establishes $Pf \in V_{\mathrm{inv}}$, and therefore \eqref{PP}.

(ii) Let $g_x(z) = \mathbb{P}(X_t = z | X_0=x) \in \mathbb{R}^{(\mathbb{Z})}$. For a fixed state $y \in [1, N]$, let $\delta_y \in \mathbb{R}^{\mathbb{Z}}$ be the Dirac delta function. By definition, $\mathbb{P}(X^{(N)}_t = y | X^{(N)}_0=x) = \mathbb{E}[\delta_y(X^{(N)}_t) | X^{(N)}_0=x]$. 
Since $y \in [1, N]$, $\delta_y$ satisfies the support condition. Substituting $f = \delta_y$ into the established result of (i) yields:
\begin{align*}
  &  \mathbb{P}(X^{(N)}_t = y | X^{(N)}_0=x) 
    = \sum_{W \in \mathcal{G}_{0,N+1}} \mathbb{E}[(W \delta_y)(X_t) | X_0=x] \\
    &= \sum_{W \in \mathcal{G}_{0,N+1}} \sum_{z \in \mathbb{Z}} (W \delta_y)(z) \mathbb{P}(X_t = z | X_0=x) = \sum_{W \in \mathcal{G}_{0,N+1}} \langle W \delta_y, g_x \rangle_0.
\end{align*}
Applying the adjoint property $\langle W \delta_y, g_x \rangle_0 = \langle \delta_y, W^* g_x \rangle_0$, we have:
\begin{align*}
    \sum_{W \in \mathcal{G}_{0,N+1}} \langle W \delta_y, g_x \rangle_0 
    &= \sum_{W \in \mathcal{G}_{0,N+1}} \langle \delta_y, W^* g_x \rangle_0 = \sum_{W \in \mathcal{G}_{0,N+1}} (W^* g_x)(y).
\end{align*}
Substituting $g_x(z) = \mathbb{P}(X_t = z | X_0=x)$ back into the final expression completes the proof:
\begin{align*}
    \mathbb{P}(X^{(N)}_t = y | X^{(N)}_0=x) = \sum_{W \in \mathcal{G}_{0,N+1}} W^* \mathbb{P} (X_t = \cdot \,\,|X_0=x) (y).
\end{align*}
\end{proof}

\begin{rem}\label{p=q=1/2}
When $p=q=1/2$ the operators $\mathcal{M}_0$ and $\mathcal{M}_{N+1}$ are the
point reflections $x\mapsto-x$ and $x\mapsto 2(N+1)-x$, so that
$\mathcal{G}_{0,N+1}$ acts on $\mathbb{Z}$ by
$x\mapsto \pm x+2k(N+1)$ $(k\in\mathbb{Z})$, and the reflected walk
$X^{(N)}$ is the image of the free walk $X$ under the folding map
$\mathbb{Z}\to\mathbb{Z}\cap[0,N+1]$ induced by this action.  Consequently
\begin{align}\label{TraceFormula00-boundary}
\begin{aligned}
      \mathbb{P}(X^{(N)}_t=y\mid X^{(N)}_0=x)
    &=\sum_{y'\in\mathcal{O}(y)}\mathbb{P}(X_t=y'\mid X_0=x), \\
&\text{where} \,\, \mathcal{O}(y):=\{\,\pm y+2k(N+1):k\in\mathbb{Z}\,\},
\end{aligned}
\end{align}
for all $x,y\in\mathbb{Z}\cap[0,N+1]$.  For $y\in[1,N]$ the map
$W\mapsto Wy$ is a bijection of $\mathcal{G}_{0,N+1}$ onto $\mathcal{O}(y)$
and \eqref{TraceFormula00-boundary} is \eqref{TraceFormula00}; for
$y\in\{0,N+1\}$ the stabilizer of $y$ has order two, so that the
group sum in \eqref{TraceFormula00} is twice the left-hand side of
\eqref{TraceFormula00-boundary}.
\end{rem}

\subsection{The Trace Formula for the Reflected Random Walk}

Recall from Section \ref{sec:2.2} that $\mathbf{P}^{(N)}$ denotes the transpose of the transition matrix for the reflected random walk $X^{(N)}_t$ on the finite path graph. We are now ready to establish the exact global-local spectral correspondence.
\begin{theorem}\label{DiscreteTrace}
For sufficiently small $u \in \mathbb{C}$, the zeta function of the reflected random walk admits the following absolutely convergent product expansion:
\begin{align}\label{DiscreteTraceFormula}
    \det (I_{N+2} - u \mathbf{P}^{(N)})^{-1} 
    = \zeta_{\partial X^{(N)}}(u) \prod_{W \in \mathcal{G}_{0,N+1}} \zeta_{W,X^{(N)}} (u).
\end{align}
Here, the boundary zeta function is defined by the return probabilities at the boundaries $x \in \{0, N+1\}$:
\begin{align*}
    \zeta_{\partial X^{(N)}} (u) := \exp \left( \sum_{n=1}^\infty \frac{u^n}{n} \sum_{x \in \{0, N+1\}} \mathbb{P}(X^{(N)}_n = x \mid X^{(N)}_0 = x) \right),
\end{align*}
and the local bulk zeta factor associated with each $W \in \mathcal{G}_{0,N+1}$ is defined by the spatial expansion over the interior $x \in [1, N]$:
\begin{align*}
\zeta_{W,X^{(N)}} (u) := \exp \left( \sum_{n=1}^\infty \frac{u^n}{n} \sum_{x=1}^{N} W^* \mathbb{P}(X_n = \cdot \mid X_0 = x) (x) \right).
\end{align*}
In particular, the contribution from the identity element $W = id$ is the $ N $-th power of the 
zeta of the free walk \eqref{FreeZeta}:
\begin{align*}
    \zeta_{id,X^{(N)}}(u) = \zeta_X(u)^{N}.
\end{align*}
\end{theorem}

\begin{proof}
By the definition of the determinant, we have the trace expansion:
\begin{align}\label{DecompTrace}
    - \mathrm{Log} \det (I_{N+2} - u \mathbf{P}^{(N)})
    = \sum_{n=1}^\infty \frac{u^n}{n} \mathrm{Tr} ((\mathbf{P}^{(N)})^n).
\end{align}
The trace of the $n$-th power of the transition matrix is the sum of the return probabilities over the state space $P_N = [0, N+1]$. 
We split this spatial sum into the boundary contributions and the bulk contributions:
\begin{align}\label{BulkTrace}
    \mathrm{Tr} ((\mathbf{P}^{(N)})^n) 
    &= \sum_{x \in \{0, N+1\}} \mathbb{P} (X^{(N)}_n =x \mid X^{(N)}_0 =x) 
    + \sum_{x=1}^{N} \mathbb{P} (X^{(N)}_n =x \mid X^{(N)}_0 =x).
\end{align}
Multiplying by $u^n / n$ and summing over $n \ge 1$, the first term (the boundary sum) forms the logarithmic exponent of the boundary zeta function $\zeta_{\partial X^{(N)}}(u)$. 

For the second term (the bulk sum), we apply the generalized reflection principle established in Proposition \ref{TraceFormula-P} (ii), which holds specifically for $x \in [1, N]$:
\begin{align*}
\begin{aligned}
    \sum_{x=1}^{N} \mathbb{P} (X^{(N)}_n =x \mid X^{(N)}_0 =x) 
    &= \sum_{x=1}^{N} \sum_{W \in \mathcal{G}_{0,N+1}} \left[ W^* \mathbb{P} (X_n = \cdot \mid X_0=x) \right] (x) \\
    &= \sum_{W \in \mathcal{G}_{0,N+1}} \sum_{x=1}^{N} \left[ W^* \mathbb{P} (X_n = \cdot \mid X_0=x) \right] (x). 
\end{aligned}
\end{align*}
Multiplying this by $u^n / n$, summing over $n \geq 1$, and taking the exponential on both sides alongside the boundary term yields the formal product representation \eqref{DiscreteTraceFormula}. 
The last part of the theorem is straightforward from the translation invariance. 

The absolute convergence follows from Lemma \ref{prop:absolute-summability} below. 
\end{proof}

\begin{lemma}\label{prop:absolute-summability}
Fix $N\geq2$. There exist constants
$C>0$ and $R>0$ such that, for every $n\geq1$,
\begin{equation}\label{eq:main-bound}
  \sum_{W\in \mathcal{G}_{0,N+1}}|\sum_{x=1}^{N} \left[ W^* \mathbb{P} (X_n = \cdot \mid X_0=x) \right] (x) |
  \leq C R^n.
\end{equation}
Consequently,
\begin{equation*}
  \sum_{n=1}^{\infty}\frac{|u|^n}{n}
  \sum_{W\in \mathcal{G}_{0,N+1}}|\sum_{x=1}^{N} \left[ W^* \mathbb{P} (X_n = \cdot \mid X_0=x) \right] (x)|<\infty
  \,\,\text{whenever } |u|<R^{-1}.
\end{equation*}
In particular, on this disk the order of summation in $n$ and $W$ may be
interchanged, and the product of the local zeta factors converges absolutely.
\end{lemma}

\begin{proof}
For $x\in\{1,\ldots,N\}$, define
\[
  \pi_{n,x}(y)
  :=\mathbb{P}(X_n=y\mid X_0=x),
  \qquad y\in\mathbb{Z}. 
\] 
Set $ \alpha := p/q $
and choose a positive number $ \sigma $ such that 
\begin{equation}\label{eq:sigma-choice}
  \sigma^2 >\rho := \max( \alpha, \alpha^{-1}).
\end{equation}
We will show that 
\begin{align}
  |(W^*\pi_{n,x})(x)|\leq B^{\ell(W)} \sigma^n \rho^{n/2}                
  \label{eq:pointwise-a-bound-2}
\end{align}
for a positive constant $ B $,
where $ \ell(W) $ is 
the number of generators occurring in the reduced expression of $ W $.
This together with an estimate that
\begin{align}\label{number-est}
    \# \{ W \in \mathcal{G}_{0,N+1} : (W^*\pi_{n,x})(x) \ne 0 \} \leq  \frac{2n} {N+1} +3
\end{align}
establishes \eqref{eq:main-bound}.

To prove \eqref{eq:pointwise-a-bound-2}, 
we introduce a weighted $ l^1$ norm 
\[
\|f\|_{1,\omega}:=\sum_{j\in\mathbb{Z}}|f(j)|\omega(j), \quad f \in l_0 (\mathbb{Z})
\]
where 
\begin{equation}\label{eq:weight}
\omega(j):=\alpha^{j/2}\sigma^{|j|},
  \qquad j\in\mathbb{Z}.
\end{equation}

Recall $ \mathsf{M} =(m_{i,j})_{i,j\in\mathbb{Z}}$
is given by 
\eqref{matrixM},
which implies the following row bounds.
For every $r\geq1$,
\begin{align}
  |m_{r,-r}|&=\alpha^{r-1},
  &
  |m_{r,-s}|&\leq c_0\alpha^s
  &&(0\leq s\leq r-2),                                      \label{eq:row-plus}\\
  |m_{-r,r}|&=\alpha^{-r+1},
  &
  |m_{-r,s}|&\leq c_0\alpha^{-s}
  &&(0\leq s\leq r-2),                         \label{eq:row-minus}
\end{align}
where, the entries are actually zero if $ r +s $ is odd, and
$  c_0 := |p-q|/(pq) $.
Indeed, for $s\geq1$ the two off-anti-diagonal expressions in
\eqref{eq:row-plus}
and \eqref{eq:row-minus} 
are equal respectively to
$c_0\alpha^s$ and $c_0\alpha^{-s}$, while the estimates for $s=0$ follow
because $p,q\leq1$.

For $r\geq1$, using \eqref{eq:weight} and \eqref{eq:row-plus}, we obtain
\begin{align*}
 &\frac{1}{\omega(r)}\sum_{j\in\mathbb{Z}}|m_{r,j}|\omega(j)
 \leq
 \alpha^{r-1}\frac{\omega(-r)}{\omega(r)}
 +c_0\sum_{s=0}^{r-2}\alpha^s\frac{\omega(-s)}{\omega(r)}\\
 &=\alpha^{-1}
 +c_0\sum_{s=0}^{r-2}(\sigma\sqrt{\alpha})^{s-r}\leq
 \alpha^{-1}+\frac{c_0}{\sigma\sqrt{\alpha}-1}.
\end{align*}
Similarly, \eqref{eq:row-minus} gives
\begin{align*}
 \frac{1}{\omega(-r)}\sum_{j\in\mathbb{Z}}|m_{-r,j}|\omega(j)
 &\leq
 \alpha+c_0\sum_{s=0}^{r-2}
       \left(\frac{\sigma}{\sqrt{\alpha}}\right)^{s-r}\leq
 \alpha+\frac{c_0}{\sigma/\sqrt{\alpha}-1}.
\end{align*}
The denominators are positive by \eqref{eq:sigma-choice}.  Since the zeroth
row consists only of $m_{0,0}=1$, it follows that
\begin{equation*}
  \| \mathcal{M}_0^*f\|_{1,\omega}
  \leq B_0\|f\|_{1,\omega},
\end{equation*}
where
\begin{equation*}
 B_0:=\max\left\{
  1,
  \alpha^{-1}+\frac{c_0}{\sigma\sqrt{\alpha}-1},
  \alpha+\frac{c_0}{\sigma/\sqrt{\alpha}-1}
 \right\}.
\end{equation*}

For a reflection centered at $K\in\mathbb{Z}$, introduce the translated weight
\[
  \omega_K(j):=\alpha^{(j-K)/2}\sigma^{|j-K|}.
\]
The same calculation gives
\begin{equation}\label{eq:shifted-reflection-bound}
  \| \mathcal{M}_K^*f\|_{1,\omega_K}
  \leq B_0\|f\|_{1,\omega_K}.
\end{equation}
On the other hand, the triangle inequality yields
\begin{equation}\label{eq:weight-equivalence}
  \alpha^{-K/2}\sigma^{-K}\omega(j)
  \leq \omega_K(j)
  \leq \alpha^{-K/2}\sigma^K\omega(j),
  \qquad j\in\mathbb{Z}.
\end{equation}
Combining \eqref{eq:shifted-reflection-bound} and
\eqref{eq:weight-equivalence}, we find
\[
  \| \mathcal{M}_K^*f\|_{1,\omega}
  \leq B_0\sigma^{2K}\|f\|_{1,\omega}.
\]
Consequently, with $  B:=B_0\sigma^{2(N+1)}\geq1 $, 
we have, for $K\in\{0,N+1\}$,
\begin{equation}\label{eq:generator-bound}
  \|\mathcal{M}_K^*f\|_{1,\omega}\leq B\|f\|_{1,\omega}.
\end{equation}
By repeated use of
\eqref{eq:generator-bound} therefore gives
\begin{equation}\label{eq:word-bound}
  \|W^*f\|_{1,\omega}\leq B^{\ell(W)}\|f\|_{1,\omega}.
\end{equation}
Since $\pi_{n,x}$ is a probability distribution and satisfies
\begin{equation}\label{eq:finite-propagation-2}
\mathrm{supp} \, \pi_{n,x}\subset[x-n,x+n],
\end{equation}
we have
\begin{equation}\label{eq:pi-weight-bound}
  \|\pi_{n,x}\|_{1,\omega}
  \leq \omega(x) \sigma^n \rho^{n/2}.
\end{equation}
Combining \eqref{eq:word-bound} and \eqref{eq:pi-weight-bound}, we obtain \eqref{eq:pointwise-a-bound-2}.

Next we proceed to prove \eqref{number-est}. 
The support pattern of the matrix $\mathsf{M}$ implies the elementary reflection rule
\begin{align}
  \mathrm{supp} f\subset[K+r,\infty)
  &\quad\Longrightarrow\quad
  \mathrm{supp}(\mathcal{M}_K f)\subset(-\infty,K-r],                         \label{eq:support-rule-1}\\
  \mathrm{supp} f\subset(-\infty,K-r]
  &\quad\Longrightarrow\quad
  \mathrm{supp}(\mathcal{M}_Kf)\subset[K+r,\infty),                         \label{eq:support-rule-2}
\end{align}
for every $r>0$.  
Every non-identity $ W \in \mathcal{G}_{0,N+1} $
has one of the following forms:
\begin{align*}
 \mathcal{T}_k &:=(\mathcal{M}_{N+1}\mathcal{M}_0)^k,
 &\mathcal{T}_{-k} &:=(\mathcal{M}_0 \mathcal{M}_{N+1})^k &&(k\geq1),\\
\mathcal{R}_{0,k}&:=\mathcal{M}_0(\mathcal{M}_{N+1} \mathcal{M}_0)^k,
 &\mathcal{R}_{N+1,k}&:=\mathcal{M}_{N+1} (\mathcal{M}_0 \mathcal{M}_{N+1})^k &&(k\geq0).
\end{align*}
Applying \eqref{eq:support-rule-1} and \eqref{eq:support-rule-2} inductively to
$\delta_x$, $1\leq x\leq N $, gives
\begin{align}
\begin{aligned}
& \mathrm{supp} (\mathcal{T}_k\delta_x)\subset[2k(N+1)+x,\infty), \quad 
 \mathrm{supp}(\mathcal{T}_{-k} \delta_x) \subset(-\infty,x-2k(N+1)],               
 \\
 & \mathrm{supp}(\mathcal{R}_{0,k}\delta_x)\subset(-\infty,-2k(N+1)-x], \\
& \mathrm{supp}(\mathcal{R}_{N+1,k}\delta_x)\subset[2(k+1)(N+1)-x,\infty).          \label{eq:supp-RL}
 \end{aligned}
\end{align}
Since 
\begin{equation*}
  (W^*\pi_{n,x})(x)
=\langle\delta_x,W^*\pi_{n,x}\rangle
  =\langle W\delta_x,\pi_{n,x}\rangle,
\end{equation*}
we see by 
\eqref{eq:finite-propagation-2}
that 
this quantity can be nonzero only if
$\mathrm{supp}(W\delta_x)$ intersects $[x-n,x+n]$.  
From
\eqref{eq:supp-RL}, this implies
\begin{equation*}
  \ell(W)\leq \frac{n}{N+1}+1.
\end{equation*}
This proves \eqref{number-est}.

Therefore, using the triangle inequality and summing first over
\(x=1,\ldots,N\), we obtain
\[
\begin{aligned}
&\sum_{W\in\mathcal G_{0,N+1}}
 \left|
 \sum_{x=1}^{N}(W^*\pi_{n,x})(x)
 \right|  \\
&\qquad\le
N\left(\frac{2n}{N+1}+3\right)
B^{\,n/(N+1)+1}\sigma^n\rho^{n/2}.
\end{aligned}
\]
Since the linear factor in \(n\) can be absorbed into an exponential,
there exist constants \(C,R>0\), depending only on \(N,p,q\), such that
the right-hand side is bounded by \(CR^n\).
This proves \eqref{eq:main-bound}.
\end{proof}

\begin{cor}[Poisson Summation Formula via Trace Formula]\label{cor:poisson}
For the symmetric reflected random walk ($p=q=1/2$), the exact evaluation of the discrete trace formula bridging Theorem \ref{M-Th3} and Theorem \ref{DiscreteTrace} under the diffusive continuum limit yields the classical Poisson summation formula.
Specifically, scaling $N, n \to \infty$ with $t := n / (N+1)^2 > 0$ constant, we obtain:
\begin{align}\label{PoissonFormula}
    \sum_{m=-\infty}^\infty e^{-t m^2 \pi^2} = \frac{1}{\sqrt{\pi t}} \sum_{k=-\infty}^\infty e^{- \frac{k^2}{t}}.
\end{align}
\end{cor}

\begin{proof}
Since the scaling relation $n=t(N+1)^2$ does not necessarily define an integer
$n$ for every $N$, we take a sequence $n=n_N\in\mathbb{N}$ such that
\begin{align*}
    t_N:=\frac{n_N}{(N+1)^2}\longrightarrow t \qquad (N\to\infty).
\end{align*}

Each element of the discrete reflection group $\mathcal{G}_{0,N+1} = \langle \mathcal{M}_0, \mathcal{M}_{N+1} \rangle$ can be uniquely decomposed into a translation operator $\mathsf{T}_k$ and a reflection operator $ \mathsf{R}_k$ defined for any $k \in \mathbb{Z}$ as:
\begin{align*}
    \mathsf{T}_k x := x + 2k(N+1), \quad \mathsf{R}_k x := -x + 2k(N+1).
\end{align*}
Using these operators, we evaluate $\mathrm{Tr} ((\mathbf{P}^{(N)})^{2n})$ in two distinct ways.

First, from \eqref{spec}, the exact trace for $p=q=1/2$ is given by:
\begin{align*}
    \mathrm{Tr} ((\mathbf{P}^{(N)})^{2n}) = 2 +  \sum_{m=1}^N \cos^{2n} \left( \frac{m \pi}{N+1} \right).
\end{align*}
Since
\begin{align*}
    \cos^{2n}\left(\frac{(N+1-m)\pi}{N+1}\right)
    =
    \cos^{2n}\left(\frac{m\pi}{N+1}\right),
\end{align*}
and the term $m=(N+1)/2$ is equal to zero when $N+1$ is even,
we obtain
\begin{align*}
    \mathrm{Tr}\bigl((\mathbf{P}^{(N)})^{2n}\bigr)
    =2+2\sum_{m=1}^{\lfloor N/2\rfloor}\cos^{2n}\left(\frac{m\pi}{N+1}\right).
\end{align*}

For fixed $\delta\in(0,\pi/2)$, $0\leq x\leq\delta$, we have
\begin{align*}
    -\log(\cos x)-\frac{x^2}{2} &= \int_0^x\int_0^s\tan^2 r\,dr\,ds \leq \frac{1}{\cos^2\delta}\int_0^x\int_0^s r^2\,dr\,ds  =\frac{x^4}{12\cos^2\delta}.
\end{align*}
Consequently,
\begin{align*}
    \exp\left(-\frac{x^2}{2}-\frac{x^4}{12\cos^2\delta}\right)
    \leq \cos x 
    \leq \exp\left(-\frac{x^2}{2}\right),
    \qquad 0\leq x\leq\delta.
\end{align*}
The upper bound also holds for every $x\in[0,\pi/2]$, because
\begin{align*}
    \frac{d}{dx}\left( \log(\cos x)+\frac{x^2}{2} \right) = x-\tan x \leq 0.
\end{align*}

For any fixed $m$, evaluating this at $x_m = \frac{m \pi}{N+1}$, we have $x_m \leq \delta$ for all sufficiently large $N$. Raising the
preceding inequality to the power $2n_N$ gives
\begin{align*}
    &\exp\left(-n_N x_m^2 - \frac{n_N x_m^4}{6\cos^2\delta} \right)
    \leq \cos^{2n_N}\left(x_{m}\right)
    \leq \exp\left(-n_N x_m^2\right).
\end{align*}
Since $t_N\to t$, the lower and upper bounds both converge to
$e^{-tm^2\pi^2}$. Hence, for every fixed $m$,
\begin{align*}
\lim_{N\to\infty}\cos^{2n_N}\left(\frac{m\pi}{N+1}\right) = e^{-tm^2\pi^2}.
\end{align*}
To justify passage of the limit through the sum, note that, for all
sufficiently large $N$, $t_N\geq \frac{t}{2}$.
Therefore, for
$1\leq m\leq\lfloor N/2\rfloor$,
\begin{align*}
    0\leq
    \cos^{2n_N}\left(\frac{m\pi}{N+1}\right)
    \leq \exp\left(-t_N m^2\pi^2\right)
    \leq \exp\left(-\frac{t\pi^2m^2}{2}\right).
\end{align*}
The last sequence is summable over $m\in\mathbb{N}$. 
Applying the dominated convergence theorem, we obtain
\begin{align}\label{SpectralLimit}
    \lim_{N\to\infty}
    \mathrm{Tr}\bigl((\mathbf{P}^{(N)})^{2n}\bigr)
    =2+2\sum_{m=1}^{\infty}e^{-tm^2\pi^2} 
    = 1+\sum_{m=-\infty}^{\infty}e^{-tm^2\pi^2}.
\end{align}

Second, we expand the trace using \eqref{DecompTrace} and \eqref{BulkTrace}:
\begin{align}\label{SpatialDecompExact}
    \begin{aligned}
    & \mathrm{Tr} ((\mathbf{P}^{(N)})^{2n}) \\
    &= \sum_{k=-\infty}^\infty \sum_{x=1}^{N} \mathbb{P}(X_{2n} = \mathsf{T}_k x \mid X_0 = x) + \sum_{k=-\infty}^\infty \sum_{x=1}^{N} \mathbb{P}(X_{2n} = \mathsf{R}_k x \mid X_0 = x)  \\
    & \qquad + \sum_{x \in \{0, N+1\}} \mathbb{P} (X^{(N)}_{2n} =x \mid X^{(N)}_0 =x) \\
    &= \sum_{k=-\infty}^\infty N \mathbb{P}(X_{2n} = 2k(N+1) \mid X_0 = 0)  \\
    &\quad + \sum_{k=-\infty}^\infty \sum_{x=1}^{N} \mathbb{P}(X_{2n} = -2x + 2k(N+1) \mid X_0 = 0)\\
    & \qquad + 2  \mathbb{P} (X^{(N)}_{2n} =0 \mid X^{(N)}_0 =0) 
    \end{aligned}
\end{align}
where we used the translation invariance $\mathbb{P}(X_{2n} = y \mid X_0 = x) = \mathbb{P}(X_{2n} = y-x \mid X_0 = 0)$.

We first note that, 
by \eqref{TraceFormula00-boundary} with $x=y=0$, by $ N+1 = \sqrt{n/t} $, and by a use of Stirling formula, there exists a constant $C > 0$ such that
\begin{align}\label{boundof P}
    \begin{aligned}
    &   \mathbb{P} (X^{(N)}_{2n} =0 \mid X^{(N)}_0 =0) 
       = \sum_{k \in \mathbb{Z} } \mathbb{P}(X_{2n} = 2k(N+1) \mid X_0 = 0) \\
      &  = \sum_{k \in  \mathbb{Z}\cap [-\sqrt{nt}, \sqrt{nt}]} 2^{-2n}
     \binom{2n}{n+k\sqrt{n/t}}
      \leq \frac{C}{\sqrt{n}} \sum_{k \in  \mathbb{Z}\cap [-\sqrt{nt}, \sqrt{nt}]} e^{-\frac{k^2}{t}}
    \end{aligned}
\end{align}
and therefore
\begin{align}\label{rec-term}
    \mathbb{P} (X^{(N)}_{2n} =0 \mid X^{(N)}_0 =0) \to 0 \quad (n,N \to \infty).
\end{align}

To evaluate the other terms, we utilize the Fourier inversion formula:
\begin{align*}
\mathbb{P}(X_{2n} = 2z \mid X_0 = 0) = \frac{1}{\pi} \int_{-\pi/2}^{\pi/2} e^{-i 2z \theta} \cos^{2n}\theta \, d\theta.
\end{align*}
We define the absolute error $R_n(z)$ between the discrete walk and the continuous Gaussian kernel as:
\begin{align*}
R_n(z) := \mathbb{P}(X_{2n} = 2z \mid X_0 = 0) - \frac{1}{\sqrt{\pi n}} e^{-\frac{z^2}{n}}.
\end{align*}
Using the integral identity $\frac{1}{\sqrt{\pi n}} e^{-z^2/n} = \frac{1}{\pi} \int_{-\infty}^{\infty} e^{-i 2z \theta} e^{-n\theta^2} d\theta$, the error function is expressed as:
\begin{align*}
R_n(z) = \frac{1}{\pi} \int_{-\pi/2}^{\pi/2} e^{-i 2z \theta} \left( \cos^{2n}\theta - e^{-n\theta^2} \right) d\theta - \frac{1}{\pi} \int_{|\theta| > \pi/2} e^{-i 2z \theta} e^{-n\theta^2} d\theta.
\end{align*}
The triangle inequality provides a uniform upper bound with respect to $z$:
\begin{align}\label{uniformbound}
|R_n(z)| \le \frac{1}{\pi} \int_{-\pi/2}^{\pi/2} \left| \cos^{2n}\theta - e^{-n\theta^2} \right| d\theta + \frac{1}{\pi} \int_{|\theta| > \pi/2} e^{-n\theta^2} d\theta.
\end{align}
The second integral decays exponentially as $e^{-n\pi^2/4}$. For the first integral, using the  inequality $0 \le e^{-n\theta^2} - \cos^{2n}\theta \le n \theta^4 e^{-n\theta^2}$ for $\theta \in [-\pi/2, \pi/2]$, the difference is bounded by $\int_{-\infty}^{\infty} n \theta^4 e^{-n\theta^2} d\theta$. Substituting $\phi = \sqrt{n}\theta$, this evaluates to $n^{-3/2} \int_{-\infty}^{\infty} \phi^4 e^{-\phi^2} d\phi = \frac{3\sqrt{\pi}}{4} n^{-3/2}$. Consequently, there exists a constant $C_1 > 0$ such that $\sup_{z \in \mathbb{Z}} |R_n(z)| \le C_1 n^{-3/2}$.

Under the scaling $n = t(N+1)^2$, we set 
\begin{align*}
\begin{aligned}
    \tilde{R}_N(k)& := R_{t(N+1)^2}(k(N+1)) \\
    &= \mathbb{P}(X_{2n} = 2k(N+1) \mid X_0 = 0) - \frac{1}{\sqrt{\pi t}(N+1)} e^{-\frac{k^2}{t}}.
\end{aligned}
\end{align*}
The uniform bound \eqref{uniformbound} implies $\sup_{z \in \mathbb{Z}} |R_{t(N+1)^2}(z)| \le \frac{C_1}{t^{3/2}(N+1)^3}$. 
Multiplying by $N$ and taking the limit yields:
\begin{align*}
\begin{aligned}
    &\lim_{N \to \infty} N\mathbb{P}(X_{2n} = 2k(N+1) \mid X_0 = 0) \\
    &= \lim_{N \to \infty} \left( \frac{N}{\sqrt{\pi t}(N+1)} e^{-k^2/t} + N\tilde{R}_N(k) \right) = \frac{1}{\sqrt{\pi t}} e^{-k^2/t}.
\end{aligned}
\end{align*}
Since $\mathbb{P}(X_{2n}=2k(N+1)\mid X_0=0)\leq\frac{C}{\sqrt{n}}e^{-k^2/t}$
by the termwise estimate used in \eqref{boundof P}, and
$\frac{N}{\sqrt{n}}=\frac{N}{\sqrt{t}\,(N+1)}\leq\frac{1}{\sqrt{t}}$, the
summands are dominated by $\frac{C}{\sqrt{t}}e^{-k^2/t}$ uniformly in $N$;
the dominated convergence theorem therefore implies
\begin{align}\label{TransLimitExact}
\lim_{N \to \infty} \sum_{k=-\infty}^\infty N\mathbb{P}(X_{2n} = 2k(N+1) \mid X_0 = 0) = \frac{1}{\sqrt{\pi t}} \sum_{k=-\infty}^\infty e^{-k^2/t}.
\end{align}

We decompose the double sum of the reflection component, which is the second term of \eqref{SpatialDecompExact}, as follows:
\begin{align*}
\begin{aligned}
&\sum_{k=-\infty}^\infty \sum_{x=1}^{N} \mathbb{P}(X_{2n} = -2x + 2k(N+1) \mid X_0 = 0) \\
&= \sum_{k=-\infty}^\infty \sum_{x=0}^{N} \mathbb{P}(X_{2n} = 2(-x + k(N+1)) \mid X_0 = 0) \\
& \qquad - \sum_{k=-\infty}^\infty \mathbb{P}(X_{2n} = 2(k-1)(N+1) \mid X_0 = 0).
\end{aligned}
\end{align*}
The mapping $(x, k) \mapsto -x + k(N+1)$ is a bijection from $\{0, \dots, N\} \times \mathbb{Z}$ to $\mathbb{Z}$. Then, since the simple random walk strictly has its support on even states at time $2n$, the first term is determined by the conservation of total probability as:
\begin{align*}
\sum_{k=-\infty}^\infty \sum_{x=0}^{N} \mathbb{P}(X_{2n} = 2(-x + k(N+1)) \mid X_0 = 0) = \sum_{z=-\infty}^\infty \mathbb{P}(X_{2n} = 2z \mid X_0 = 0) = 1.
\end{align*}
The remaining second term is a single sum corresponding to the boundary
point $x=0$; after the shift $k\mapsto k+1$ it equals
$\mathbb{P}(X^{(N)}_{2n}=0\mid X^{(N)}_0=0)$, which tends to $0$ by
\eqref{rec-term}. Therefore the limit of the entire reflection component is
strictly determined as:
\begin{align*}
&\lim_{N \to \infty} \sum_{k=-\infty}^\infty \mathbb{P}(X_{2n} = 2(k-1)(N+1) \mid X_0 = 0) \\
&\le \lim_{N \to \infty} \frac{1}{N+1} \left( \frac{1}{\sqrt{\pi t}} \sum_{j=-\infty}^\infty e^{-j^2/t} \right) = 0.
\end{align*}
Therefore, the limit of the entire reflection component is strictly determined as:
\begin{align}\label{RefLimitExact}
\lim_{N \to \infty} \sum_{k=-\infty}^\infty \sum_{x=1}^{N} \mathbb{P}(X_{2n} = \mathsf{R}_k x \mid X_0 = x) =  1.
\end{align}

Combining the spectral limit \eqref{SpectralLimit} with the spatial limits \eqref{rec-term}, \eqref{TransLimitExact}, 
and \eqref{RefLimitExact}, 
\begin{align*}
1 + \sum_{m=-\infty}^\infty e^{-t m^2 \pi^2} = \frac{1}{\sqrt{\pi t}} \sum_{k=-\infty}^\infty e^{-k^2/t} + 1
\end{align*}
is derived. Subtracting $1$ from both sides yields the classical Poisson summation formula \eqref{PoissonFormula}.
\end{proof}

\begin{rem}[Historical Context and Probabilistic Proofs of the Poisson Summation Formula]
The deep connection between the Poisson summation formula (the inversion formula for the Jacobi theta function) and the transition densities of random walks has a rich history in probability theory and analytic number theory. Classically, the theta function emerges in the asymptotic distribution of various functionals of Brownian motion. For instance, the celebrated Kolmogorov-Smirnov statistic \cite{Kolmogorov1933, Smirnov1939} evaluates the maximum deviation of an empirical distribution function, which Feller \cite{Feller1948} elegantly proved converges to the distribution of the supremum of a Brownian bridge. The equivalence of evaluating this probability via the method of images (spatial reflection) and via eigenfunction expansion (spectral sum) naturally yields the theta inversion formula.

More recently, Smith and Diaconis \cite{SmithDiaconis1988} investigated the maximum reach of honest Bernoulli excursions, explicitly demonstrating that the functional equation of the Riemann zeta function can be derived by applying the Mellin transform to the continuous limit of these excursion probabilities. This perspective was profoundly generalized by Biane, Pitman, and Yor \cite{BianePitmanYor2001}, who systematically surveyed how the probability laws of Brownian excursions and meanders inherently encode the automorphic properties of the Jacobi theta and Riemann zeta functions.

However, the derivation presented in Corollary \ref{cor:poisson} offers a fundamentally different and discrete geometric perspective:
The standard probabilistic proofs \cite{Feller1948, BianePitmanYor2001} assume the continuum and the smooth topological structure (e.g., the heat equation on a manifold or the Wiener measure) \textit{a priori}. In contrast, our starting point is a strictly finite, discrete algebraic object—the exact trace formula of a transition matrix on a finite path graph. The topological concept of ``winding numbers" in the continuum formula emerges purely as an asymptotic shadow of the word length in the discrete reflection group $\mathcal{G}_{0,N+1}$.
While Smith and Diaconis \cite{SmithDiaconis1988} applied the Mellin transform to the continuous limit to uncover the zeta function's properties, our discrete trace formula (Theorem \ref{DiscreteTrace}) demonstrates that the local Euler-product structure of the zeta function is already
    encoded in the combinatorial geometry of the finite graph, prior to any continuum scaling.
In this light, the discrete trace formula for the reflected random walk can be viewed as a pre-geometric skeleton of these classical probabilistic identities, bridging discrete graph theory and continuous wave mechanics.
\end{rem}

\section{Trace Formula for the Reflected Szegedy Walk}\label{sec:6}

This section establishes a quantum analog of the discrete trace formula for the reflected Szegedy walk. By introducing a similarity transformation that symmetrizes the reflection operators, we uncover a term-by-term local Konno-Sato correspondence between the local zeta functions of the classical random walk and the quantum walk. Combining this local correspondence with the global Konno-Sato theorem, we  decompose the determinant of the finite Szegedy walk into a product over the reflection group.

\subsection{Symmetrized Reflections and Trace Formula}

Let $D = \mathrm{diag}(d_0, d_1, \dots, d_{N+1})$ be a diagonal matrix, 
where 
\begin{align*}
    d_0 &= 1, \\
    d_x &= q^{(x-1)/2} p^{-x/2} \quad (1 \le x \le N), \\
    d_{N+1} &= q^{N/2} p^{-N/2}.
\end{align*}
Using this $D$, the matrix $\tilde{\mathbf{P}}^{(N)} = D^{-1} \mathbf{P}^{(N)} D$ is symmetric ($\tilde{\mathbf{P}}^{(N)}_{x,y} = \tilde{\mathbf{P}}^{(N)}_{y,x}$):
\begin{align*}
    \begin{aligned}
& \tilde{\mathbf{P}}^{(N)}_{x, x+1} = \tilde{\mathbf{P}}^{(N)}_{x+1, x} = \sqrt{pq} \quad (1 \le x \le N-1), \\
& \tilde{\mathbf{P}}^{(N)}_{0, 1} = \tilde{\mathbf{P}}^{(N)}_{1, 0} = \sqrt{p}, \quad \tilde{\mathbf{P}}^{(N)}_{N, N+1} = \tilde{\mathbf{P}}^{(N)}_{N+1, N} = \sqrt{q}\\
& \tilde{\mathbf{P}}^{(N)}_{x, y} =0 \quad \text{if $ |x-y| \ne 1 $}.
    \end{aligned}
\end{align*}

Though (the transpose of) $ \tilde{\mathbf{P}}^{(N)}$ is not Markov for general $ p, q$, we can construct a spatial trace formula for it by redefining the reflection operators via the same similarity transformation.
Let $\mathcal{D}_\infty : \mathbb{R}^{\mathbb{Z}} \to \mathbb{R}^{\mathbb{Z}}$ be the infinite diagonal operator defined by 
$(\mathcal{D}_\infty f)(x) = (q/p)^{x/2} f(x)$. 
The free symmetric random walk matrix is given by $\tilde{\mathbf{P}}^{(\infty)} := \mathcal{D}_\infty^{-1} \mathbf{P}^{(\infty)} \mathcal{D}_\infty$.

We define the symmetrized reflection operators acting on functions by
\begin{align*}
    \tilde{\mathcal{M}}_K := \mathcal{D}_\infty \mathcal{M}_K \mathcal{D}_\infty^{-1}, \quad K \in \{0, N+1\},
\end{align*}
and let $\tilde{\mathcal{G}}_{0,N+1} = \langle \tilde{\mathcal{M}}_0, \tilde{\mathcal{M}}_{N+1} \rangle$.
By definition, their adjoint operators satisfy $\tilde{\mathcal{M}}^*_K = \mathcal{D}_\infty^{-1} \mathcal{M}^*_K \mathcal{D}_\infty$.
\begin{lemma}\label{trace-sym}
For all $t \in \mathbb{Z}_+$ and interior points $x, y \in [1,N] \cap \mathbb{Z}$, the transition elements of the symmetrized matrix $\tilde{\mathbf{P}}^{(N)}$ satisfy:
\begin{align}\label{SymTraceFormula00}
    \langle y | (\tilde{\mathbf{P}}^{(N)})^t | x \rangle = \sum_{\tilde{W} \in \tilde{\mathcal{G}}_{0,N+1}} \langle y | \tilde{W}^* (\tilde{\mathbf{P}}^{(\infty)})^t | x \rangle.
\end{align}
Consequently, the trace over the interior is given by:
\begin{align*}
    \sum_{x=1}^{N} \langle x | (\tilde{\mathbf{P}}^{(N)})^t | x \rangle
    = \sum_{\tilde{W} \in \tilde{\mathcal{G}}_{0,N+1}} 
    \sum_{x=1}^{N} \langle x| \tilde{W}^*(\tilde{\mathbf{P}}^{(\infty)})^t |x \rangle.
\end{align*}
\end{lemma}

\begin{proof}
Recall the generalized reflection principle 
\eqref{TraceFormula00} of 
Proposition \ref{TraceFormula-P} (ii), which holds exactly for the interior points $x, y \in [1,N]$:
\begin{align*}
    \langle y | (\mathbf{P}^{(N)})^t | x \rangle = \sum_{W \in \mathcal{G}_{0,N+1}} \langle y | W^* (\mathbf{P}^{(\infty)})^t | x \rangle.
\end{align*}
The left-hand side is related to the symmetrized matrix by $\mathbf{P}^{(N)} = D \tilde{\mathbf{P}}^{(N)} D^{-1}$, yielding:
\begin{align*}
    \langle y | (\mathbf{P}^{(N)})^t | x \rangle &= d_y \langle y | (\tilde{\mathbf{P}}^{(N)})^t | x \rangle d_x^{-1} \\
    &= (q/p)^{y/2} \langle y | (\tilde{\mathbf{P}}^{(N)})^t | x \rangle (p/q)^{x/2}.
\end{align*}
For the right-hand side, substituting $\mathbf{P}^{(\infty)} = \mathcal{D}_\infty \tilde{\mathbf{P}}^{(\infty)} \mathcal{D}_\infty^{-1}$ and $W^* = \mathcal{D}_\infty \tilde{W}^* \mathcal{D}_\infty^{-1}$, we have
\begin{align*}
    W^* (\mathbf{P}^{(\infty)})^t 
    &= (\mathcal{D}_\infty \tilde{W}^* \mathcal{D}_\infty^{-1}) (\mathcal{D}_\infty (\tilde{\mathbf{P}}^{(\infty)})^t \mathcal{D}_\infty^{-1}) \\
    &= \mathcal{D}_\infty \tilde{W}^* (\tilde{\mathbf{P}}^{(\infty)})^t \mathcal{D}_\infty^{-1}.
\end{align*}
Evaluating this product at the states $\langle y |$ and $| x \rangle$:
\begin{align*}
    \langle y | W^* (\mathbf{P}^{(\infty)})^t | x \rangle 
    &= \langle y | \mathcal{D}_\infty \tilde{W}^* (\tilde{\mathbf{P}}^{(\infty)})^t \mathcal{D}_\infty^{-1} | x \rangle \\
    &= (q/p)^{y/2} \langle y | \tilde{W}^* (\tilde{\mathbf{P}}^{(\infty)})^t | x \rangle (p/q)^{x/2}.
\end{align*}
Equating both sides and dividing by the common boundary factors $(q/p)^{y/2}$ and $(p/q)^{x/2}$ yields \eqref{SymTraceFormula00}. Setting $y=x$ and summing over $x \in [1, N]$ directly produces the bulk trace formula.
\end{proof}
\subsection{Konno-Sato Correspondence between Local Bulk Zeta Functions}

We define the local zeta functions associated with the symmetrized matrix $\tilde{\mathbf{P}}^{(N)}$ and the quantum walk $Y^{(N)}$ over the symmetric reflection group $\tilde{W} \in \tilde{\mathcal{G}}_{0,N+1}$ by  limiting the spatial expansion to the bulk interior $x \in [1, N]$: for sufficiently small $ z $
and $ u $, 
\begin{align*}
\zeta_{\tilde{W},\tilde{\mathbf{P}}^{(N)}} (z) &= 
    \exp \left( \sum_{n=1}^\infty \frac{z^n}{n} \sum_{x=1}^{N} \langle x | \tilde{W}^* (\tilde{\mathbf{P}}^{(\infty)})^n | x\rangle \right), \\
\zeta_{\tilde{W},Y^{(N)}} (u) &= 
 \exp \left( \sum_{n=1}^\infty \frac{u^n}{n} \sum_{x=1}^{N}\mathrm{Tr}_{\mathbb{C}^2}\langle \tilde{W} \delta_x |  (\mathbf{U}^{(\infty)})^n | x \rangle \right).
\end{align*}

\begin{lemma}\label{local-konno-sato}
For sufficiently small $ u \in \mathbb{C} $, the term-by-term local Konno-Sato correspondence holds:
\begin{align*}
\zeta_{\tilde{W}, Y^{(N)}} (u) = (1+u^2)^{-c(\tilde{W})} \zeta_{\tilde{W}, \tilde{\mathbf{P}}^{(N)}} \left(\frac{2u}{1+u^2}\right),
\end{align*}
where $c(\tilde{W}) := \sum_{x=1}^{N} \langle x | \tilde{W}^* | x \rangle$. Specifically, 
\begin{align*}
    c(\tilde{W}) = \begin{cases}
        N & \text{if } \tilde{W} = id, \\
        0 & \text{otherwise}.
    \end{cases}
\end{align*}
\end{lemma}

\begin{proof}
Let $z = \frac{2u}{1+u^2}$. We first examine the local generating functions of the free walks on $\mathbb{Z}$. By the inverse Fourier transform identity 
of Lemma \ref{FourLemma}, the generating function for the free Szegedy walk is:
\begin{align*}
    F_{x,y}(u) &:= \sum_{n=1}^\infty \frac{u^n}{n} \mathrm{Tr}_{\mathbb{C}^2} \langle y | (\mathbf{U}^{(\infty)})^n | x \rangle \\
    &= - \int_0^{2\pi} e^{ik(y-x)} \mathrm{Log} \det (I_2 - u \widehat{\mathbf{U}^{(\infty)}}(k)) \frac{\mathrm{d}k}{2\pi}\\
    &= - \int_0^{2\pi} e^{ik(y-x)} \mathrm{Log} (1 - 4u\sqrt{pq}\cos k + u^2) \frac{\mathrm{d}k}{2\pi}, 
\end{align*}
where we have used the equation \eqref{detUk}. 

Similarly, for the symmetrized free random walk:
\begin{align*}
    G_{x,y}(z) &:= \sum_{n=1}^\infty \frac{z^n}{n} \langle y | (\tilde{\mathbf{P}}^{(\infty)})^n | x \rangle \\
     &= - \int_0^{2\pi} e^{ik(y-x)} \mathrm{Log} (1 - z \widehat{\tilde{\mathbf{P}}^{(\infty)}}(k) ) \frac{\mathrm{d}k}{2\pi} \\
&= - \int_0^{2\pi} e^{ik(y-x)} \mathrm{Log} (1 - z (e^{-ik} \tilde{\mathbf{P}}^{(\infty)}_{1,0} + e^{ik} \tilde{\mathbf{P}}^{(\infty)}_{-1,0} ) ) \frac{\mathrm{d}k}{2\pi} \\
    &= - \int_0^{2\pi} e^{ik(y-x)} \mathrm{Log} (1 - 2z\sqrt{pq}\cos k) \frac{\mathrm{d}k}{2\pi}.
\end{align*}
Substituting $z = \frac{2u}{1+u^2}$, we observe that $1 - 2z\sqrt{pq}\cos k = \frac{1 - 4u\sqrt{pq}\cos k + u^2}{1+u^2}$. Taking the logarithm algebraically separates the integral:
\begin{align*}
    G_{x,y} \left(\frac{2u}{1+u^2}\right) &= - \int_0^{2\pi} e^{ik(y-x)} \left[ \mathrm{Log} (1 - 4u\sqrt{pq}\cos k + u^2) - \mathrm{Log}(1+u^2) \right] \frac{\mathrm{d}k}{2\pi} \\
    &= F_{x,y}(u) + \mathrm{Log}(1+u^2) \delta_{x,y}.
\end{align*}
Thus, the generating functions exactly match on the off-diagonal terms, isolating the disparity purely to the diagonal: $F_{x,y}(u) = G_{x,y}(z) - \mathrm{Log}(1+u^2) \delta_{x,y}$.

We apply this relation to the local zeta functions. Expanding the bulk trace yields:
\begin{align*}
    \mathrm{Log} \zeta_{\tilde{W},Y^{(N)}} (u) &= \sum_{x=1}^{N} \sum_{y \in \mathbb{Z}} \langle x | \tilde{W}^* | y \rangle F_{x,y}(u) \\
    &= \sum_{x=1}^{N} \sum_{y \in \mathbb{Z}} \langle x| \tilde{W}^* | y \rangle G_{x,y}(z) - \mathrm{Log}(1+u^2) \sum_{x=1}^{N} \langle x | \tilde{W}^* | x \rangle.
\end{align*}
The first term equals to $\mathrm{Log} \zeta_{\tilde{W},\tilde{\mathbf{P}}^{(N)}} (z)$. Exponentiating both sides yields the desired multiplicative factor $(1+u^2)^{-c(\tilde{W})}$.

Finally, we evaluate $c(\tilde{W})$ over the bulk $x \in [1, N]$. Since $\tilde{W}^* = \mathcal{D}_\infty^{-1} W^* \mathcal{D}_\infty$, its diagonal elements are identical to those of the original operator $W^*$:
\begin{align*}
    \langle x | \tilde{W}^* | x \rangle = (q/p)^{x/2} \langle x | W^* | x \rangle (p/q)^{x/2} = \langle x | W^* | x \rangle.
\end{align*}
For the identity $W^* = id$, $\langle x | id | x \rangle = 1$ for all $x \in [1, N]$, yielding $N$. 
Crucially, for the boundary generators, $\langle 0 | \mathcal{M}^*_{0} | 0 \rangle = 1$ and $\langle N+1 | \mathcal{M}^*_{N+1} | N+1 \rangle = 1$. Since the summation is restricted to the interior $x \in [1, N]$, we have $\langle x | \mathcal{M}^*_{0} | x \rangle = 0$ and $\langle x | \mathcal{M}^*_{N+1} | x \rangle = 0$ inside the sum. For any alternating word of length $\geq 2$, the reflection strictly bounds the support away from $x$, ensuring $\langle x | W^* | x \rangle = 0$. Thus, $c(\tilde{W})$ is determined to be non-zero solely for the identity, completing the proof.
\end{proof}

\subsection{The Trace Formula for the Zeta Function of the Szegedy Walk}

\begin{theorem}\label{QWDiscreteTrace}
For $|u|$ sufficiently small, the determinant associated with the Szegedy walk $\mathbf{U}^{(N)}$ on the path $P_N$ admits the following  product expansion over the reflection group $\widetilde{\mathcal{G}}_{0,N+1}$:
\begin{align}\label{QWDiscreteTraceFormulaFixedPoint}
    \det (I_{2N+2} - u \mathbf{U}^{(N)})^{-1} 
    = \zeta_{\partial Y^{(N)}}(u) \prod_{\tilde{W} \in \tilde{\mathcal{G}}_{0,N+1}} \zeta_{\tilde{W}, Y^{(N)}} (u).
\end{align}
Here, 
\begin{align*}
    \zeta_{\partial Y^{(N)}}(u) := \frac{1-u^2}{(1+u^2)^2} \zeta_{\partial X^{(N)}} \left(\frac{2u}{1+u^2}\right).
\end{align*}
\end{theorem}

\begin{proof}
The proof is established by combining the global Konno-Sato theorem with the bulk-separated discrete trace formula for random walks and the local correspondence lemma.

First, recall the global Konno-Sato correspondence given in \eqref{KSforUNPN}:
\begin{align*}
    \det (I_{2N+2} - u \mathbf{U}^{(N)})^{-1} 
    &= \frac{1-u^2}{(1+u^2)^{N+2}} \det \left( I_{N+2} - \frac{2u}{1+u^2} \mathbf{P}^{(N)} \right)^{-1}.
\end{align*}
Let $z = \frac{2u}{1+u^2}$. By applying the discrete trace formula for the reflected random walk (Theorem \ref{DiscreteTrace}) to the determinant on the right-hand side, we obtain the product expansion where the boundary and bulk are explicitly separated. Since $\tilde{\mathbf{P}}^{(N)}$ and $\mathbf{P}^{(N)}$ are similar, their determinants and bulk traces are identical, allowing us to use the symmetrized group $\tilde{\mathcal{G}}_{0,N+1}$ and its associated bulk zeta functions:
\begin{align*}
    \det (I_{2N+2} - u \mathbf{U}^{(N)})^{-1} &= \frac{1-u^2}{(1+u^2)^{N+2}} \zeta_{\partial X^{(N)}}(z) \prod_{\tilde{W} \in \tilde{\mathcal{G}}_{0,N+1}} \zeta_{\tilde{W}, \tilde{\mathbf{P}}^{(N)}} (z).
\end{align*}

Next, we apply the local Konno-Sato correspondence from Lemma \ref{local-konno-sato} to each bulk factor in the product. Rearranging the formula gives:
\begin{align*}
    \zeta_{\tilde{W}, \tilde{\mathbf{P}}^{(N)}} (z) = (1+u^2)^{c(\tilde{W})} \zeta_{\tilde{W}, Y^{(N)}} (u).
\end{align*}
Substituting this into our product expansion, we have:
\begin{align*}
   & \det (I_{2N+2} - u \mathbf{U}^{(N)})^{-1} \\
   &= \frac{1-u^2}{(1+u^2)^{N+2}} \zeta_{\partial X^{(N)}}(z) \prod_{\tilde{W} \in \tilde{\mathcal{G}}_{0,N+1}} (1+u^2)^{c(\tilde{W})} \zeta_{\tilde{W}, Y^{(N)}} (u) \\
    &= \frac{1-u^2}{(1+u^2)^{N+2}} \zeta_{\partial X^{(N)}}(z) (1+u^2)^{\sum c(\tilde{W})} \prod_{\tilde{W} \in \tilde{\mathcal{G}}_{0,N+1}} \zeta_{\tilde{W}, Y^{(N)}} (u).
\end{align*}

The sum of the diagonal evaluation factors $c(\tilde{W})$ over the entire reflection group is determined by the identity element, as established in Lemma \ref{local-konno-sato}:
\begin{align*}
    \sum_{\tilde{W} \in \tilde{\mathcal{G}}_{0,N+1}} c(\tilde{W}) = c(id) = N.
\end{align*}
Substituting this sum back into the expression, the rational function of $(1+u^2)$ simplifies:
\begin{align*}
    \frac{(1+u^2)^N}{(1+u^2)^{N+2}} = \frac{1}{(1+u^2)^2}.
\end{align*}
Consequently, 
\begin{align*}
    \frac{1-u^2}{(1+u^2)^2} \zeta_{\partial X^{(N)}}(z) = \zeta_{\partial Y^{(N)}}(u).
\end{align*}
Combining these results yields the desired factorized formula \eqref{QWDiscreteTraceFormulaFixedPoint}, completing the proof.
\end{proof}

\begin{rem}\label{rem:identity-is-Mahler}
The identity element alone accounts for the thermodynamic limits of Section
\ref{sec:3}, for the Szegedy walk as well as for the random walk. Let
$F_{x,y}(u)$ and $G_{x,y}(z)$ be the generating functions of the free walks
on $\mathbb{Z}$ introduced in the proof of Lemma \ref{local-konno-sato}.
Comparing $F_{0,0}(u)$ with the definition \eqref{TraceDefinition} of
$\zeta_Y$ gives $\log\zeta_{Y}(u)=F_{0,0}(u)$, while $c(id)=N$ and the
translation invariance of $\mathbf{U}^{(\infty)}$ give
\begin{align*}
    \log\zeta_{id, Y^{(N)}}(u)=\sum_{x=1}^{N}F_{x,x}(u)=N\,F_{0,0}(u).
\end{align*}
Hence, for $|u|$ sufficiently small,
\begin{align*}
    \zeta_{id, Y^{(N)}}(u)=\zeta_{Y}(u)^{N},
\end{align*}
which is the exact quantum counterpart of $\zeta_{id,X^{(N)}}(u)=\zeta_X(u)^N$
in Theorem \ref{DiscreteTrace}; the same computation with $G_{0,0}$ in place
of $F_{0,0}$ gives the latter.

Consequently, raising \eqref{QWDiscreteTraceFormulaFixedPoint} to the power
$1/(N+1)$ and using Theorem \ref{Th3},
\begin{align*}
    \left[\zeta_{\partial Y^{(N)}}(u)
      \prod_{\tilde{W}\neq id}\zeta_{\tilde{W},Y^{(N)}}(u)\right]^{\frac{1}{N+1}}
    =\frac{\bar{\zeta}_{Y^{(N)}}(u)}{\zeta_Y(u)^{\frac{N}{N+1}}}
    \ \longrightarrow\ 1
    \qquad (N\to\infty),
\end{align*}
and likewise in the classical case by Theorem \ref{M-Th3}. In other words,
the Mahler measure representations \eqref{Th1} and \eqref{QWlimzeta_integral}
of the thermodynamic limit are exactly the local zeta function of the
identity element: after the normalization by $1/(N+1)$, every non-trivial
group element and the boundary factor become invisible. The infinite-graph
limit is thus a purely local statement about the free walk, and not a global
spectral one.
\end{rem}

\section{Quantum Poisson Summation Formula}\label{sec:7}

By taking the continuum scaling limit of the discrete trace formula for the Szegedy walk, we obtain the quantum analog of Corollary \ref{cor:poisson}. While the classical random walk corresponds to the heat equation (yielding the Gaussian Poisson summation formula), the quantum walk is governed by unitary wave propagation, which naturally yields the distributional Poisson summation formula known as the Dirac comb identity.

\begin{theorem}[Quantum Poisson Summation Formula (Dirac Comb)]\label{thm:dirac-comb}
For the symmetric reflected Szegedy walk ($p=q=1/2$), the asymptotic evaluation of the discrete trace formula \eqref{QWDiscreteTraceFormulaFixedPoint} under the ballistic continuum limit yields the Dirac comb identity. 
Specifically, scaling $N, n \to \infty$ while maintaining a constant velocity parameter $v := n / (N+1)$, we obtain in the sense of distributions:
\begin{align}\label{QuantumPoissonFormula}
    \sum_{m=-\infty}^\infty e^{i m \pi v} = 2 \sum_{k=-\infty}^\infty \delta(v - 2k).
\end{align}
\end{theorem}

\begin{proof}
We define a sequence of discrete distributions $\nu_N \in \mathcal{D}'(0, \infty)$ associated with the trace of the transition matrix by scaling the discrete time $n$ with the velocity parameter $v = n/(N+1)$:
\begin{align*}
    \nu_N := \frac{1}{N+1} \sum_{n=1}^\infty \mathrm{Tr}((\mathbf{U}^{(N)})^n) \delta_{n/(N+1)},
\end{align*}
where $\delta_{v_0}$ is the Dirac delta distribution centered at $v_0$. We evaluate the limit of $\nu_N$ as $N \to \infty$ for any test function $\phi \in C_c^\infty(0, \infty)$ in two distinct ways.

First, we use the explicit spectral decomposition. For $p=q=1/2$, the eigenvalues of $\mathbf{U}^{(N)}$ are exactly $\mu_0 = 1$, $\mu_{N+1} = -1$, and $\mu_{\pm m} = e^{\pm i \frac{m\pi}{N+1}}$ for $1 \le m \le N$.
The trace is expressed as:
\begin{align*}
    \mathrm{Tr} ((\mathbf{U}^{(N)})^n) &= 
     1 + (-1)^n + 2 \sum_{m=1}^N \cos \left( n \frac{m \pi}{N+1} \right) = \sum_{m=0}^{2N+1} e^{i n \frac{m \pi}{N+1}}. \\
\end{align*}
Evaluating the measure $\nu_N$ against the test function $\phi$, we obtain:
\begin{align*}
\begin{aligned}
     \langle \nu_N, \phi \rangle &= \frac{1}{N+1} \sum_{n=1}^\infty \sum_{m=0}^{2N+1} e^{i n \frac{m\pi}{N+1}} \phi\left(\frac{n}{N+1}\right) \\
     &= \sum_{m=-N}^{N+1} \left( \frac{1}{N+1} \sum_{n=1}^\infty e^{i m \pi \frac{n}{N+1}} \phi\left(\frac{n}{N+1}\right) \right) \\
     &= \sum_{m \in \mathbb{Z} } 
     J_m (N)
\end{aligned}
    \end{align*}
where 
\begin{equation*}
    J_m(N) := 1_{[-N,N+1]}(m) \frac{1}{N+1} \sum_{n=1}^\infty e^{i m \pi \frac{n}{N+1}} \phi\left(\frac{n}{N+1}\right)
\end{equation*}
for $ m \in \mathbb{Z} $ and $ N \in \mathbb{N} $. 
For any fixed integer $m$, $J_m(N)$ 
for $ N > |m| $ is a Riemann sum for the integral of $f_m(v) = e^{i m \pi v} \phi(v)$ over the interval $[0, \infty)$ with step size $\Delta v = \frac{1}{N+1}$. 
Since $\phi \in C^\infty_c(0, \infty)$, it is Riemann integrable, and we have the pointwise limit:
\begin{equation*}
    \lim_{N \to \infty} J_m(N) = \int_0^\infty e^{i m \pi v} \phi(v) \, dv,
\end{equation*}
and also 
\begin{align*}
    |J_m (N)| \leq \frac{C}{m^2} \Vert \phi'' \Vert_{L^1(0,\infty)}
\end{align*}
for $ m \ne 0 $. 
Then, by the dominated convergence theorem, we have 
\begin{align*}
    \lim_{N \to \infty} \langle \nu_N, \phi \rangle = \sum_{m=-\infty}^\infty \int_0^\infty e^{i m \pi v} \phi(v) dv.
\end{align*}
This means $ \lim_{N \to \infty} \nu_N =\sum_{m=-\infty}^\infty e^{i m \pi v}$
in $\mathcal{D}'(0, \infty) $.

Second, we evaluate the identical trace using the spatial product expansion \eqref{QWDiscreteTraceFormulaFixedPoint}. 
Applying the formal logarithmic derivative operator $u \frac{d}{du} \mathrm{Log}$ to both sides of the exact determinant formula yields a generating function for the traces:
\begin{align*}
    \sum_{n=1}^\infty u^n \mathrm{Tr} ((\mathbf{U}^{(N)})^n) 
    = u \frac{d}{du} \mathrm{Log} \, \zeta_{\partial Y^{(N)}}(u) 
    + \sum_{\tilde{W} \in \tilde{\mathcal{G}}_{0,N+1}} u \frac{d}{du} \mathrm{Log} \, \zeta_{\tilde{W}, Y^{(N)}}(u).
\end{align*}
By extracting the exact coefficient of $u^n$ on both sides, we obtain the following spatial trace partition:
\begin{align}\label{SpatialPartitionQW}
    \mathrm{Tr} ((\mathbf{U}^{(N)})^n) = B_Y(n) + \sum_{k=-\infty}^\infty S_{\mathcal{T}_k}(n) + \sum_{k=-\infty}^\infty S_{R_k}(n).
\end{align}
Here, the boundary anomaly trace $B_Y(n)$ is uniquely defined as the $n$-th coefficient of the logarithmic derivative of the quantum boundary zeta function:
\begin{align*}
    \sum_{n=1}^\infty B_Y(n) u^n := u \frac{d}{du} \mathrm{Log} \left[ \frac{1-u^2}{(1+u^2)^2} \zeta_{\partial X^{(N)}} \left(\frac{2u}{1+u^2}\right) \right].
\end{align*}
The terms $ S_{\mathcal{T}_k}(n) $
and $ S_{\mathcal{R}_k}(n) $
for the translation elements $\mathcal{T}_k x = x + 2k(N+1)$ and the reflection elements $\mathcal{R}_k x = -x + 2k(N+1)$ for $k \in \mathbb{Z}$ that form the group $\tilde{\mathcal{G}}_{0,N+1}$ are 
expressed via $\mathbf{U}^{(\infty)}$
using the definition of the local bulk zeta function $\zeta_{\tilde{W}, Y^{(N)}}$ as:
\begin{align}\label{STSR}
    S_{\mathcal{T}_k}(n) &:= \sum_{x=1}^{N} \mathrm{Tr}_{\mathbb{C}^2} \langle \mathcal{T}_k x | (\mathbf{U}^{(\infty)})^n | x \rangle, \\
    S_{\mathcal{R}_k}(n) &:= \sum_{x=1}^{N} \mathrm{Tr}_{\mathbb{C}^2} \langle \mathcal{R}_k x | (\mathbf{U}^{(\infty)})^n | x \rangle.
\end{align}

Since 
when $p=q=1/2$ 
\begin{align*}
    \widehat{\mathbf{U}^{\infty}}(k) = \begin{pmatrix} 0 & e^{ik} \\ e^{-ik} & 0 \end{pmatrix} \begin{pmatrix} 0 & 1 \\ 1 & 0 \end{pmatrix} = \begin{pmatrix} e^{ik} & 0 \\ 0 & e^{-ik} \end{pmatrix}, 
\end{align*}
we have, by the inverse Fourier transform 
\eqref{InvFourier0}, 
\begin{align}\label{k-delta}
    \mathrm{Tr}_{\mathbb{C}^2} \langle y | (\mathbf{U}^{(\infty)})^n | x \rangle = \delta_{y, x+n} + \delta_{y, x-n},
\end{align}
where $\delta$ denotes the Kronecker delta.
By substituting 
\eqref{k-delta} to \eqref{STSR},
we obtain 
\begin{align*}
    \sum_{k=-\infty}^\infty S_{\mathcal{T}_k}(n) = \begin{cases}
        2N & \text{if } n \equiv 0 \pmod{2(N+1)}, \\
        0 & \text{otherwise},
    \end{cases}
\end{align*}
and 
\begin{align*}
    \sum_{k=-\infty}^\infty S_{\mathcal{R}_k}(n) = \begin{cases}
        0 & \text{if } n \text{ is odd, or } n \equiv 0 \pmod{2(N+1)}, \\
        2 & \text{if } n \text{ is even and } n \not\equiv 0 \pmod{2(N+1)}.
    \end{cases}
\end{align*}

Next, we directly calculate the boundary anomaly trace $B_Y(n)$ without relying on the total spectral trace. 
Let $z = \frac{2u}{1+u^2}$. 
By \eqref{TraceFormula00-boundary} applied at $x=y=0$ and at $x=y=N+1$, we obtain  
$$z \frac{d}{dz} \mathrm{Log} \, \zeta_{\partial X^{(N)}}(z) = 2 \sum_{k=-\infty}^\infty \sum_{n=1}^\infty z^n \mathbb{P}(X_n = 2k(N+1) \mid 0). $$
Using the free random walk generating function
$$\sum_{n=1}^\infty z^n \mathbb{P}(X_n = y) = \frac{1+u^2}{1-u^2} u^{|y|}$$ 
for $y \neq 0$, and $\frac{2u^2}{1-u^2}$ for $y=0$, we evaluate this classical term exactly:
\begin{align*}
    z \frac{d}{dz} \mathrm{Log} \, \zeta_{\partial X^{(N)}}(z) = \frac{4u^2}{1-u^2} + \frac{4(1+u^2)}{1-u^2} \sum_{k=1}^\infty u^{2k(N+1)}.
\end{align*}
Applying the chain rule $u \frac{dz}{du} = z \frac{1-u^2}{1+u^2}$ to the definition of $B_Y(n)$, we obtain 
\begin{align*}
    \sum_{n=1}^\infty B_Y(n) u^n &= u \frac{d}{du} \mathrm{Log} \left[ \frac{1-u^2}{(1+u^2)^2} \right] + \frac{1-u^2}{1+u^2} \left( z \frac{d}{dz} \mathrm{Log} \, \zeta_{\partial X^{(N)}}(z) \right) \\
    &= \left( \frac{-2u^2}{1-u^2} - \frac{4u^2}{1+u^2} \right) + \frac{1-u^2}{1+u^2} \left( \frac{4u^2}{1-u^2} + \frac{4(1+u^2)}{1-u^2} \sum_{k=1}^\infty u^{2k(N+1)} \right) \\
    &= \frac{-2u^2}{1-u^2} + 4 \sum_{k=1}^\infty u^{2k(N+1)}.
\end{align*}
By extracting the coefficients of $u^n$, we can determine $B_Y(n)$:
\begin{align*}
    B_Y(n) = \begin{cases}
        -2 + 4 = 2 & \text{if } n \equiv 0 \pmod{2(N+1)}, \\
        -2 & \text{if } n \text{ is even and } n \not\equiv 0 \pmod{2(N+1)}, \\
        0 & \text{if } n \text{ is odd}.
    \end{cases}
\end{align*}

Adding the three components $B_Y(n)$, $S_{\mathcal{T}}(n)$, and $S_{\mathcal{R}}(n)$ back into the spatial trace partition \eqref{SpatialPartitionQW},
we obtain 
\begin{align*}
    \mathrm{Tr} ((\mathbf{U}^{(N)})^n) = \begin{cases}
        2N + 2 & \text{if } n \equiv 0 \pmod{2(N+1)}, \\
        0 & \text{otherwise}.
    \end{cases}
\end{align*}
Applying the discrete measure $\nu_N$ to a test function $\phi \in C_c^\infty(0, \infty)$, we substitute this exact expansion:
\begin{align*}
    \langle \nu_N, \phi \rangle
    &= \frac{1}{N+1} \sum_{k=1}^\infty \left( 2N + 2 \right) \phi\left( \frac{2k(N+1)}{N+1} \right) = 2 \sum_{k=1}^\infty \phi(2k).
\end{align*}
Because this equality holds strictly for all $N \ge 2$, it trivially passes to the continuum limit:
\begin{align*}
    \lim_{N \to \infty} \langle \nu_N, \phi \rangle = 2 \sum_{k=1}^\infty \phi(2k) = \left\langle 2 \sum_{k=-\infty}^\infty \delta(v - 2k), \phi \right\rangle.
\end{align*}
This establishes the continuous Dirac comb identity \eqref{QuantumPoissonFormula}.
\end{proof}

\begin{rem}[Heat Kernel vs. Wave Kernel]
This result contrasts with the random walk case. 
The symmetric reflected random walk corresponds to a diffusive limit ($t = n/(N+1)^2$), converging to the fundamental solution of the Heat equation on $S^1$ and thereby recovering the classic Gaussian Poisson summation formula. Conversely, the quantum walk maintains unitarity without dissipation; its ballistic scaling ($v = n/(N+1)$) maps directly to the Wave equation on $S^1$. The emergence of the Dirac comb identity demonstrates that the continuous Heat and Wave kernel trace formulas are simply distinct scaling limits of the same universal discrete trace formula over the reflection group.

The wave kernel $e^{it\sqrt{-\Delta}}$ is the fundamental solution of the wave equation $\partial_t^2 u = \Delta u$ and plays a central role in the spectral geometry of manifolds. While the heat kernel describes the dissipation of probability mass, the wave kernel represents the coherent propagation of waves. A celebrated result in this field is the Poisson relation for the wave trace \cite{Duistermaat1975, Chavel1984}, which states that the trace of the wave kernel is a distribution whose singular support is contained within the set of lengths of periodic geodesics. In our context, the symmetric Szegedy walk behaves as a discrete-time wave propagator. 
Its ballistic scaling limit  recovers the wave trace on a circle, where the ``singular support" manifests as the Dirac comb at the resonance points $v=2k$. This connection emphasizes that the Konno-Sato correspondence is a discrete manifestation of the duality between the eigenvalue spectrum and the length spectrum in Riemannian geometry.
\end{rem}

\section{Double Scaling Limit and a Local Zeta Factorization of the Massive Dirac-Type Trace Formula}\label{sec:8}

In this section, we consider a \textit{double scaling limit} where the spatial asymmetry (drift) and the spectral parameter are scaled simultaneously with the system size $N$.

\subsection{Scaling Hypotheses and Notation}\label{sec:8.1}

We define the effective mass $M \in \mathbb{R}$ and the continuous spectral parameter $s \in \mathbb{C}$ ($\mathrm{Re}(s) > 0$) by scaling the transition probabilities and the spectral variable $u$ as follows:
\begin{equation}\label{eq:scaling}
    p = \frac{1}{2} + \frac{M}{2N}, \quad q = \frac{1}{2} - \frac{M}{2N}, \quad u = e^{-s/N}.
\end{equation}
In this regime, the mass gap $\Delta = \arccos(2\sqrt{pq})$ scales as $O(M/N)$, which allows the boundary reflections to remain coherent even as $N \to \infty$.

Throughout this section we use the following notation.
As in Section \ref{sec:5} we write
$ \alpha := p/q $, 
and, as in the Konno--Sato correspondence \eqref{KSforUNPN},
\begin{align}\label{eq:z-and-xi}
    z := \frac{2u}{1+u^2}, \qquad
    \xi_N := \frac{1+u^2}{4u\sqrt{pq}}, \qquad \text{so that} \quad
    2z\sqrt{pq} = \xi_N^{-1}.
\end{align}
We let $\lambda_\pm := \xi_N \pm \sqrt{\xi_N^2-1}$ be the characteristic roots attached to $\xi_N$, so that
\begin{align}\label{eq:lambda-basic}
    \lambda_+\lambda_- = 1, \qquad \lambda_+ + \lambda_- = 2\xi_N,
\end{align}
the branch being fixed by $\lambda_- = e^{-\gamma_N}$ with $\gamma_N := \mathrm{arcosh}(\xi_N)$, $\mathrm{Re}(\gamma_N)>0$; in particular $|\lambda_-| < 1$.
Finally, for the massive momentum we write
\begin{align}\label{eq:varkappa}
    \varkappa := \sqrt{s^2+M^2}, \qquad \mathrm{Re}(\varkappa) > 0
\end{align}
(principal branch), so that $\varkappa^2 - M^2 = s^2$.

Note that \eqref{eq:scaling} requires $N>|M|$, which we assume from now
on; then $p,q\in(0,1)$ and $2\sqrt{pq}=\sqrt{1-M^2/N^2}\in(0,1]$.

Finally we fix the region of the spectral parameter in which the double
scaling limit will be taken. Put
\begin{align}\label{eq:SM}
    \mathcal{S}_M:=\bigl\{\, s\in\mathbb{C}\ :\
      \mathrm{Re}(s)>0 \ \text{ and }\ \mathrm{Re}(\varkappa)>|M| \,\bigr\}.
\end{align}
The first condition is what makes $|u|=|e^{-s/N}|<1$; the second one is
exactly the condition under which the local zeta functions of the odd
words converge, as will be seen in Step 2 of the proof of Theorem
\ref{thm:double_scaling}. Writing $s=\sigma+i\omega$ and using
$\mathrm{Re}(\varkappa)^2-\mathrm{Im}(\varkappa)^2=\sigma^2-\omega^2+M^2$
together with $\mathrm{Re}(\varkappa)\,\mathrm{Im}(\varkappa)=\sigma\omega$,
one checks that, for $M\neq0$,
\begin{align*}
    \mathcal{S}_M
    =\bigl\{\,\sigma+i\omega\ :\ \sigma>0,\ \
       \sigma^{2}\left(M^{2}+\omega^{2}\right)>M^{2}\omega^{2}\,\bigr\},
\end{align*}
equivalently $\sigma^{-2}<\omega^{-2}+M^{-2}$ when $\omega\neq0$, while
$\mathcal{S}_0=\{\mathrm{Re}(s)>0\}$. Thus $\mathcal{S}_M$ is an open
connected set which contains the half plane $\{\mathrm{Re}(s)>|M|\}$ and
the positive real axis, whose boundary is asymptotic to the line
$\mathrm{Re}(s)=|M|$ as $|\mathrm{Im}(s)|\to\infty$, and which increases
to the whole half plane $\{\mathrm{Re}(s)>0\}$ as $M\to0$.

We shall also use the following elementary description of $\lambda_-$ as a
function of $u$. Let $J(v):=\frac12\left(v+v^{-1}\right)$ be the Joukowski
map, which is a biholomorphism of the punctured disc
$\mathbb{D}^{*}=\{0<|v|<1\}$ onto $\mathbb{C}\setminus[-1,1]$. By
\eqref{eq:z-and-xi} we have $\xi_N=a\,J(u)$ with
$a:=(2\sqrt{pq})^{-1}\geq1$, and by \eqref{eq:lambda-basic}
$J(\lambda_-)=\xi_N$ with $|\lambda_-|<1$; hence
\begin{align*}
    \lambda_-=J^{-1}\!\left(a\,J(u)\right),\qquad 0<|u|<1 ,
\end{align*}
so that $\lambda_-$ is a biholomorphism of $\mathbb{D}^{*}$ onto
$\mathbb{D}^{*}\setminus\bigl([\varrho,1)\cup(-1,-\varrho]\bigr)$, where
\begin{align}\label{eq:rho-slit}
    \varrho:=J^{-1}(a)=a-\sqrt{a^{2}-1}
    =\frac{1-|p-q|}{2\sqrt{pq}}=\min(\alpha,\alpha^{-1})^{1/2}.
\end{align}
In particular $\xi_N\notin[-1,1]$, so $\gamma_N=\mathrm{arcosh}(\xi_N)$ is
well defined, and $\lambda_-$ is holomorphic and injective on
$\mathbb{D}^{*}$.

Finally, we single out the subdomain of $\mathbb{D}^{*}$ on which the local
zeta functions of the odd words will be shown to converge:
\begin{align}\label{eq:OmegaN}
    \Omega_N:=\bigl\{\,u\in\mathbb{C}\ :\ 0<|u|<1,\ \
      |\lambda_-(u)|^{2}<\min(\alpha,\alpha^{-1})\,\bigr\}
    =\lambda_-^{-1}\bigl(\{0<|v|<\varrho\}\bigr).
\end{align}
The second equality is \eqref{eq:rho-slit}, and $\{0<|v|<\varrho\}$ is
contained in the image $\mathbb{D}^{*}\setminus\bigl([\varrho,1)\cup(-1,-\varrho]\bigr)$
of $\lambda_-$; since $\lambda_-$ is injective and holomorphic on
$\mathbb{D}^{*}$, the set $\Omega_N$ is a connected open set, and it contains
a punctured neighbourhood of $u=0$ because $\lambda_-(u)\to0$ as $u\to0$.
Note that $\Omega_N$ depends only on $p$ and $q$, hence on $N$ through
\eqref{eq:scaling}, and that $\Omega_N=\mathbb{D}^{*}$ precisely when $p=q$,
in which case $\varrho=1$ and the two slits degenerate. For $p\neq q$ the
inclusion $\Omega_N\subsetneq\mathbb{D}^{*}$ is strict, but $\Omega_N$ always
contains the real punctured interval $(-1,1)\setminus\{0\}$, since
$\lambda_-$ maps it onto $(-\varrho,\varrho)\setminus\{0\}$ with
$\lambda_-(u)\uparrow\varrho$ only as $u\uparrow1$.

\subsection{Global and Local Zeta Functions in the Limit}\label{sec:8.2}

To establish a continuum Walk/Zeta correspondence, we first justify the geometric definition of the global zeta function in the continuous regime.
In classical spectral geometry, the Selberg or Ruelle-type dynamical zeta function for a manifold is defined via an Euler product over all primitive periodic orbits (prime geodesics) $\gamma$:
\begin{align*}
    Z(s) = \prod_{\gamma} \left( 1 - e^{-s \cdot l(\gamma)} \right)^{-1},
\end{align*}
where $l(\gamma)$ denotes the geometric length of the orbit.

In our macroscopic continuum limit, the discrete path graph collapses onto a one-dimensional orbifold (an interval with reflecting boundaries). This geometry possesses exactly one unique primitive periodic orbit $\gamma_0$, which corresponds to a single unoriented complete round trip between the boundaries, yielding the macroscopic length $l(\gamma_0) = 2$. Furthermore, for a relativistic field with an effective mass $M$, the standard wave propagation parameter $s$ is altered by the dispersion relation into the massive momentum parameter $\varkappa=\sqrt{s^2+M^2}$.
Thus,
for the massive one-dimensional system, we naturally obtain the definition of the continuous global zeta function:
\begin{align*}
    Z^{\mathrm{cont}}(s, M) := \left( 1 - e^{-l(\gamma_0)\varkappa} \right)^{-1} = \left( 1 - e^{-2\sqrt{s^2+M^2}} \right)^{-1}.
\end{align*}
We refer to this as a \emph{massive Dirac-type} zeta function because the spectral parameter $s$ is replaced by
$\varkappa=\sqrt{s^2+M^2}$, 
which is the Euclidean massive relativistic dispersion relation. 
In the present scaling limit, the two conjugate spectral branches of the Szegedy walk play the role of the positive- and negative-energy branches of a one-dimensional Dirac system, while $M$ opens a spectral gap between them. 
Accordingly, a traversal of length $l$ carries the propagation factor $ e^{-l\varkappa}$, and the primitive round trip $\gamma_0$, of length $l(\gamma_0)=2$, produces the Euler factor displayed above. When $M=0$, one has $\varkappa=s$, and the factor reduces to the corresponding massless one. 
The qualifier ``Dirac-type'' is introduced here at the level of
the dispersion relation and the resulting periodic-orbit factor.
It is justified a posteriori in Section \ref{sec:8.6}, where the
limit is identified with the zeta-regularized determinant of
$\mathcal{D}^{2}+s^{2}$ on $[0,1]$ with Dirichlet ends.

Further, by observing
\begin{align*}
    \log Z^{\mathrm{cont}}(s, M) = -\log\left(1 - e^{-2\varkappa}\right) = \sum_{k=1}^\infty \frac{e^{-2k\varkappa}}{k},
\end{align*}
we define the local zeta function of our massive one-dimensional system by
\begin{align}\label{eq:continuous_local_zeta}
\zeta_{\gamma_0^k}^{\mathrm{cont}}(s, M) := \exp \left( \frac{e^{-k l(\gamma_0) \varkappa}}{k} \right) = \exp \left( \frac{e^{-2k \sqrt{s^2 + M^2}}}{k} \right).
\end{align}
This definition
is justified by
Theorem \ref{thm:double_scaling} below,
where $\zeta_{\gamma_0^k}^{\mathrm{cont}}(s, M) $ is obtained as the limit of a local zeta function of the Szegedy walk.

\subsection{The Reflection Group and its Word Decomposition}

As established in Section \ref{sec:4}, the generalized reflection operators $\tilde{\mathcal{M}}_0$ and $\tilde{\mathcal{M}}_{N+1}$ are involutions ($\tilde{\mathcal{M}}_0^2 = \tilde{\mathcal{M}}_{N+1}^2 = id$) acting on the state space. Thus, for any $p, q \in (0,1)$, the group $\tilde{\mathcal{G}}_{0,N+1} = \langle \tilde{\mathcal{M}}_0, \tilde{\mathcal{M}}_{N+1} \rangle$ is algebraically isomorphic to the infinite dihedral group $\mathrm{D}_\infty$. 
We classify its elements by their alternating word structure:
\begin{enumerate}
    \item \textbf{Translations (Even-length words):} For $k \in \mathbb{N}$, we set
    \begin{align*}
        \tilde{\mathcal{T}}_k := (\tilde{\mathcal{M}}_{N+1}\tilde{\mathcal{M}}_0)^k, \qquad
        \tilde{\mathcal{T}}_{-k} := (\tilde{\mathcal{M}}_0\tilde{\mathcal{M}}_{N+1})^k ,
    \end{align*}
    and $\tilde{\mathcal{T}}_0 := id$. These elements algebraically represent discrete cycles bouncing between the two boundaries $k$ times. We write $\tilde{\mathfrak{T}} := \{ \tilde{\mathcal{T}}_k : k \in \mathbb{Z}\setminus\{0\} \}$.
    \item \textbf{Reflections (Odd-length words):} For $k \in \mathbb{Z}_{\ge 0}$, we set
    \begin{align*}
        \tilde{\mathcal{R}}_{0,k} := \tilde{\mathcal{M}}_0 (\tilde{\mathcal{M}}_{N+1}\tilde{\mathcal{M}}_0)^k, \qquad
        \tilde{\mathcal{R}}_{N+1,k} := \tilde{\mathcal{M}}_{N+1} (\tilde{\mathcal{M}}_0\tilde{\mathcal{M}}_{N+1})^k,
    \end{align*}
    and let $\tilde{\mathfrak{R}}$ denote the set of all such elements.
\end{enumerate}
Every element of $\tilde{\mathcal{G}}_{0,N+1}$ belongs to exactly one of $\{id\}$, $\tilde{\mathfrak{T}}$ and $\tilde{\mathfrak{R}}$, so that the trace formula \eqref{QWDiscreteTraceFormulaFixedPoint} of Section \ref{sec:6} reads
\begin{align}\label{eq:geometric_decomposition_proof}
\begin{aligned}
   & \det(I_{2N+2} - u\mathbf{U}^{(N)})^{-1} \\
    &= \zeta_{\partial Y^{(N)}}(u)\, \zeta_{id, Y^{(N)}}(u) \left( \prod_{\tilde{\mathcal{R}} \in \tilde{\mathfrak{R}}} \zeta_{\tilde{\mathcal{R}}, Y^{(N)}}(u) \right) \left( \prod_{k=1}^\infty \zeta_{\tilde{\mathcal{T}}_k, Y^{(N)}}(u) \zeta_{\tilde{\mathcal{T}}_{-k}, Y^{(N)}}(u) \right).
\end{aligned}
\end{align}

\subsection{Exact Evaluation of the Local Zeta Functions}

The lemmas of this subsection evaluate every factor of \eqref{eq:geometric_decomposition_proof}, except the boundary factor $\zeta_{\partial Y^{(N)}}(u)$, in closed form. 
All the identities below are exact and do \emph{not} presuppose the scaling
\eqref{eq:scaling}: Lemmas \ref{lem:freekernel}, \ref{lem:images} and
\ref{lem:translations} hold for every $N\ge2$, every $p,q\in(0,1)$ with
$p+q=1$ and every $u$ with $0<|u|<1$, while Lemma \ref{lem:reflections} requires in addition $u\in\Omega_N$, the
domain \eqref{eq:OmegaN} of Section \ref{sec:8.1}, which for $p\neq q$ is a
proper subdomain of $\{0<|u|<1\}$.

It is convenient to introduce the logarithmic resolvent of the symmetrized free random walk,
\begin{align*}
    \mathbf{L}(z) := \sum_{n=1}^\infty \frac{z^n}{n} \bigl( \tilde{\mathbf{P}}^{(\infty)} \bigr)^n ,
\end{align*}
so that, by the definition of the local bulk zeta functions in Section \ref{sec:6},
\begin{align}\label{eq:zeta-as-kernel}
    \log \zeta_{\tilde{W}, \tilde{\mathbf{P}}^{(N)}}(z)
    = \sum_{x=1}^{N} \langle x | \tilde{W}^* \mathbf{L}(z) | x \rangle
    = \sum_{x=1}^{N} \sum_{y \in \mathbb{Z}} \langle x | \tilde{W}^* | y \rangle \, \langle y | \mathbf{L}(z) | x \rangle ,
    \,\, \tilde{W} \in \tilde{\mathcal{G}}_{0,N+1}.
\end{align}

\begin{lemma}[Free logarithmic kernel]\label{lem:freekernel}
For $|z|$ sufficiently small,
\begin{align}\label{eq:freekernel}
    \langle y | \mathbf{L}(z) | x \rangle =
    \begin{cases}
        \dfrac{\lambda_-^{\,|x-y|}}{|x-y|}, & x \neq y, \\[2mm]
        \log\bigl( 1 + \lambda_-^2 \bigr), & x = y .
    \end{cases}
\end{align}
\end{lemma}

\begin{proof}
The operator $\tilde{\mathbf{P}}^{(\infty)}$ is translation invariant with symbol $\widehat{\tilde{\mathbf{P}}^{(\infty)}}(k) = 2\sqrt{pq}\cos k$, so that, exactly as in the proof of Lemma \ref{local-konno-sato},
\begin{align*}
    \langle y | \mathbf{L}(z) | x \rangle
    = - \int_0^{2\pi} e^{ik(y-x)}\, \mathrm{Log}\bigl( 1 - \xi_N^{-1}\cos k \bigr) \frac{\mathrm{d}k}{2\pi},
\end{align*}
where we used \eqref{eq:z-and-xi}.
Writing $1 - \xi_N^{-1}\cos k = -\frac{1}{2\xi_N} e^{-ik}\bigl( e^{2ik} - 2\xi_N e^{ik} + 1 \bigr)$ and factoring the quadratic, whose roots are $\lambda_-$ and $\lambda_-^{-1} = \lambda_+$ by \eqref{eq:lambda-basic}, we obtain
\begin{align*}
    1 - \xi_N^{-1} \cos k
    = \frac{1}{1+\lambda_-^2} \bigl( 1 - \lambda_- e^{ik} \bigr)\bigl( 1 - \lambda_- e^{-ik} \bigr),
\end{align*}
where we have used $2\xi_N \lambda_- = \lambda_-(\lambda_++\lambda_-) = 1 + \lambda_-^2$.
Since $|\lambda_-|<1$, we may expand $\mathrm{Log}(1-\lambda_- e^{\pm ik}) = - \sum_{m \ge 1} \frac{\lambda_-^m}{m} e^{\pm ikm}$ and integrate term by term. The constant $\mathrm{Log}(1+\lambda_-^2)^{-1}$ contributes only to the diagonal $x=y$, and the two series contribute $\lambda_-^{|x-y|}/|x-y|$ off the diagonal and nothing on it. This is \eqref{eq:freekernel}.
\end{proof}

\begin{lemma}[Exact image structure of the symmetrized reflections]\label{lem:images}
With $\alpha=p/q$ one has the identity
\begin{align}\label{eq:key-identity}
    \frac{q-p}{pq} = \frac{q}{p} - \frac{p}{q} = \alpha^{-1} - \alpha .
\end{align}
Moreover, for every $x \ge 1$ and every $y \in \mathbb{Z}$,
\begin{align}\label{eq:M0-column}
\begin{aligned}
   & \langle x | \tilde{\mathcal{M}}_0^* | y \rangle =
    \begin{cases}
        w_m, & y = -x-2m, \quad m \in \mathbb{Z}_{\ge 0}, \\
        0, & \text{otherwise},
    \end{cases}
    \end{aligned}
\end{align}
where $ w_0 := \alpha $, and $ w_m := (\alpha - \alpha^{-1})\,\alpha^{m} \quad (m \ge 1) $,
and, for every $x \le N$ and every $y \in \mathbb{Z}$,
\begin{align}\label{eq:MN1-column}
\begin{aligned}
  &  \langle x | \tilde{\mathcal{M}}_{N+1}^* | y \rangle =     \begin{cases}
        v_m, & y = 2(N+1)-x+2m, \quad m \in \mathbb{Z}_{\ge 0}, \\
        0, & \text{otherwise},
    \end{cases}
\end{aligned}
\end{align}
where $ v_0 := \alpha^{-1}$, and $ v_m := (\alpha^{-1} - \alpha)\,\alpha^{-m} \quad (m \ge 1)$. 
Their generating functions satisfy
\begin{align}\label{eq:WV}
    W(s) := \sum_{m \ge 0} w_m s^m = \frac{\alpha - s}{1 - \alpha s}, \qquad 
    V(s) := \sum_{m \ge 0} v_m s^m = \frac{1 - \alpha s}{\alpha - s},
\end{align}
for sufficiently small $ \sigma $
so that $ W(\sigma) V(\sigma) \equiv 1 $.
\end{lemma}

\begin{proof}
The identity \eqref{eq:key-identity} is $\frac{q^2-p^2}{pq} = \frac{(q-p)(q+p)}{pq}$ together with $p+q=1$.

Since $\tilde{\mathcal{M}}_K = \mathcal{D}_\infty \mathcal{M}_K \mathcal{D}_\infty^{-1}$ with $(\mathcal{D}_\infty f)(x) = \alpha^{-x/2} f(x)$, the matrix of $\tilde{\mathcal{M}}_K$ is $\langle i | \tilde{\mathcal{M}}_K | j \rangle = \tilde{m}_{i-K,\,j-K}$, where
$   \tilde{m}_{i,j} := \alpha^{-(i-j)/2} m_{i,j} $,
and consequently $\langle x | \tilde{\mathcal{M}}_K^* | y \rangle = \tilde{m}_{y-K,\,x-K}$.
Reading off the definition \eqref{matrixM} of $\mathsf{M} = (m_{i,j})$:
\begin{itemize}
    \item $m_{i,-i} = \alpha^{i-1}$ for $i \ge 1$ gives $\tilde{m}_{i,-i} = \alpha^{-i}\alpha^{i-1} = \alpha^{-1}$, and symmetrically $\tilde{m}_{-i,i} = \alpha$;
    \item for $i > 0 \ge j$ with $i+j = 2k \ge 2$ we have $m_{i,j} = p^{-j-1}q^{j-1}(q-p)$, whence
    \begin{align*}
        \tilde{m}_{i,j} = \alpha^{-(i-j)/2} p^{-j-1}q^{j-1}(q-p) 
        = \frac{q-p}{pq}\,\alpha^{-k} = (\alpha^{-1}-\alpha)\,\alpha^{-k},
    \end{align*}
    a function of $i+j$ only;
    \item symmetrically, for $i < 0 \le j$ with $i+j = -2k \le -2$, $\tilde{m}_{i,j} = (\alpha-\alpha^{-1})\,\alpha^{k}$.
\end{itemize}
For \eqref{eq:M0-column} take $K=0$ and read the column $j = x \ge 1$ of $\tilde{m}_{y,x}$: the only admissible entries are $y = -x$, of weight $\tilde{m}_{-x,x} = \alpha$, and $y < 0 \le x$ with $y+x = -2m \le -2$, of weight $(\alpha-\alpha^{-1})\alpha^{m}$.
For \eqref{eq:MN1-column} take $K = N+1$ and read the column $j = x-(N+1) \le -1$ of $\tilde{m}_{y-(N+1),\,x-(N+1)}$: the admissible entries are $y-(N+1) = (N+1)-x$, of weight $\alpha^{-1}$, and $y-(N+1) > 0 \ge x-(N+1)$ with $\bigl( y-(N+1) \bigr) + \bigl( x-(N+1) \bigr) = 2m \ge 2$, of weight $(\alpha^{-1}-\alpha)\alpha^{-m}$.

Finally, by summing the geometric series and using \eqref{eq:key-identity},
\begin{align*}
    W(s) &= \alpha + (\alpha-\alpha^{-1}) \frac{\alpha s}{1-\alpha s}
     = \frac{\alpha-s}{1-\alpha s}, \\
    V(s) &= \alpha^{-1} + (\alpha^{-1}-\alpha) \frac{\alpha^{-1}s}{1-\alpha^{-1}s}
   = \frac{1-\alpha s}{\alpha- s},
\end{align*}
whence $W(\sigma)V(\sigma) \equiv 1$.
\end{proof}

\begin{rem}\label{rem:not-reflections}
Formulas \eqref{eq:M0-column} and \eqref{eq:MN1-column} hold for \emph{all} $x > 0$, resp.\ all $x < N+1$, not merely for bulk points; this is what makes the composition in Lemma \ref{lem:translations} purely combinatorial. Note also that $W(1) = V(1) = -1$: for $p \neq q$ the operators $\tilde{\mathcal{M}}_K$ are genuinely not point reflections, and their ``tails'' $(w_m)_{m\ge1}$, $(v_m)_{m\ge1}$ are not negligible --- each has size $O(1/N)$ under \eqref{eq:scaling}, but $O(N)$ of them lie within the range of the kernel \eqref{eq:freekernel}, so that their total contribution is $O(1)$. What happens for the even words is not that the tails become small, but that they cancel identically; see Lemma \ref{lem:translations} below. 
\end{rem}

\begin{lemma}[The even words act as exact translations on the bulk]\label{lem:translations}
For every $k \in \mathbb{Z} \setminus \{0\}$ and every $x \in \{1, \dots, N\}$,
\begin{align}\label{eq:Texact}
    \langle x | \tilde{\mathcal{T}}_k^* | y \rangle = \delta_{y,\, x + 2k(N+1)} \qquad (y \in \mathbb{Z}).
\end{align}
Consequently, for all $p,q \in (0,1)$ with $p+q=1$ and all $N \ge 2$,
\begin{align*}
    \log \zeta_{\tilde{T}_k, Y^{(N)}}(u) = \log \zeta_{\tilde{T}_k, \tilde{\mathbf{P}}^{(N)}}(z)
    = \frac{N}{2|k|(N+1)} \lambda_-^{\,2|k|(N+1)},
\end{align*}
and
\begin{align}\label{eq:prodTexact}
    \prod_{\tilde{T} \in \tilde{\mathcal{T}}} \zeta_{\tilde{T}, Y^{(N)}}(u)
    = \bigl( 1 - \lambda_-^{2(N+1)} \bigr)^{-\frac{N}{N+1}} .
\end{align}
\end{lemma}

\begin{proof}
Let $k \ge 1$, so that $\tilde{\mathcal{T}}_k^* = (\tilde{\mathcal{M}}_0^* \tilde{\mathcal{M}}_{N+1}^*)^k$.
Start from $x \in \{1,\dots,N\}$ and apply the $2k$ factors from the left.
By \eqref{eq:M0-column}, an application of $\tilde{\mathcal{M}}_0^*$ to a point $> 0$ produces points $< 0$; by \eqref{eq:MN1-column}, an application of $\tilde{\mathcal{M}}_{N+1}^*$ to a point $< N+1$ produces points $> N+1$. Hence the hypotheses of Lemma \ref{lem:images} are met at every step, and after the $2k$ steps the endpoint is
\begin{align*}
    y = x + 2k(N+1) + 2m, \qquad m = m_1 + \cdots + m_{2k} \quad (m_i \in \mathbb{Z}_+),
\end{align*}
carrying the weight $\prod_{i \text{ odd}} w_{m_i} \prod_{i \text{ even}} v_{m_i}$.
For each fixed $m$ only finitely many $(m_1,\dots,m_{2k})$ occur, so that the matrix element is the finite sum
\begin{align*}
    \langle x | \tilde{\mathcal{T}}_k^* | x + 2k(N+1) + 2m \rangle = c^{(k)}_m,
    \quad
    \sum_{m \ge 0} c^{(k)}_m \sigma^m = \bigl( W(\sigma) V(\sigma) \bigr)^k = 1
\end{align*}
as an identity of formal power series, by \eqref{eq:WV}. Comparing coefficients gives $c^{(k)}_0 = 1$ and $c^{(k)}_m = 0$ for $m \ge 1$, which is \eqref{eq:Texact}. The case $k \le -1$ is identical with the roles of $\tilde{\mathcal{M}}_0^*$ and $\tilde{\mathcal{M}}_{N+1}^*$ interchanged, and yields $y = x - 2|k|(N+1)$.

Inserting \eqref{eq:Texact} and Lemma \ref{lem:freekernel} into \eqref{eq:zeta-as-kernel} gives
\begin{align*}
    \log \zeta_{\tilde{T}_k, \tilde{\mathbf{P}}^{(N)}}(z)
    = \sum_{x=1}^{N} \frac{\lambda_-^{\,2|k|(N+1)}}{2|k|(N+1)}
    = \frac{N}{2|k|(N+1)} \lambda_-^{\,2|k|(N+1)},
\end{align*}
and the passage from $\tilde{\mathbf{P}}^{(N)}$ to $Y^{(N)}$ is Lemma \ref{local-konno-sato} together with $c(\tilde{T}_k) = 0$. Summing over $k \in \mathbb{Z}\setminus\{0\}$,
\begin{align*}
    \log \prod_{\tilde{T} \in \tilde{\mathcal{T}}} \zeta_{\tilde{T}, Y^{(N)}}(u)
    = 2 \sum_{k=1}^\infty \frac{N}{2k(N+1)} \lambda_-^{\,2k(N+1)}
    = - \frac{N}{N+1} \log \bigl( 1 - \lambda_-^{2(N+1)} \bigr),
\end{align*}
which is \eqref{eq:prodTexact}.
\end{proof}

\begin{lemma}[The odd words]\label{lem:reflections}
For every $k \ge 0$ and every $x \in \{1,\dots,N\}$, the only non-vanishing matrix elements of the odd words are
\begin{align}\label{eq:Rcolumns}
\begin{aligned}
    \langle x | \tilde{\mathcal{R}}_{0,k}^* | y \rangle &= w_m
      &&\text{for } y = -\bigl( x + 2k(N+1) \bigr) - 2m, \quad m \ge 0, \\
    \langle x | \tilde{\mathcal{R}}_{N+1,k}^* | y \rangle &= v_m
      &&\text{for } y = 2(k+1)(N+1) - x + 2m, \quad m \ge 0 .
\end{aligned}
\end{align}
Let moreover $u\in\Omega_N$, the domain \eqref{eq:OmegaN}. Then the product
$\prod_{\tilde{\mathcal{R}}\in\tilde{\mathfrak{R}}}\zeta_{\tilde{\mathcal{R}},Y^{(N)}}(u)$
converges absolutely, the double series
\begin{align}\label{eq:R-doublesum}
    \log \prod_{\tilde{\mathcal{R}} \in \tilde{\mathfrak{R}}} \zeta_{\tilde{\mathcal{R}}, Y^{(N)}}(u)
    = \sum_{\substack{n \ge 1 \\ (N+1) \nmid n}} \ \sum_{m \ge 0} (w_m + v_m)\,
        \frac{\lambda_-^{\,2(n+m)}}{2(n+m)}
\end{align}
converges absolutely and locally uniformly on $\Omega_N$, and, in closed form,
\begin{align}\label{eq:R-exact}
    \prod_{\tilde{R} \in \tilde{\mathcal{R}}} \zeta_{\tilde{R}, Y^{(N)}}(u)
    = \frac{\bigl( 1 - \lambda_-^{2(N+1)} \bigr)^{\frac{1}{N+1}}}{1-\lambda_-^2}\, e^{\,\mathcal{E}_N},
\end{align}
where
\begin{align}\label{eq:def-EN}
    \mathcal{E}_N := \frac{(\alpha-1)^2}{2} \int_0^{\lambda_-^2}
        \frac{(1+\tau)^2}{(1-\alpha\tau)(\alpha-\tau)}
        \left( \frac{1}{1-\tau} - \frac{\tau^{N}}{1-\tau^{N+1}} \right) \mathrm{d}\tau ,
\end{align}
the integral being taken along the segment from $0$ to $\lambda_-^2$. In
particular both sides of \eqref{eq:R-exact} are holomorphic on $\Omega_N$,
and $\mathcal{E}_N = 0$ if $p=q$.
\end{lemma}

\begin{proof}
Since $\tilde{\mathcal{R}}_{0,k}^* 
= \tilde{\mathcal{T}}_k^* \tilde{\mathcal{M}}_0^*$ and $\tilde{\mathcal{R}}_{N+1,k}^* 
= \tilde{\mathcal{T}}_{-k}^* \tilde{\mathcal{M}}_{N+1}^*$, Lemma \ref{lem:translations} reduces the computation to a single application of \eqref{eq:M0-column}, resp.\ \eqref{eq:MN1-column}; the sign hypotheses are met because $x + 2k(N+1) > 0$ and $x - 2k(N+1) < N+1$. This is \eqref{eq:Rcolumns}.

By Lemma \ref{lem:freekernel} the corresponding spatial distances are $|y-x| = 2\bigl(k(N+1)+x+m\bigr)$ and $|y-x| = 2\bigl((k+1)(N+1)-x+m\bigr)$. As $k$ runs over $\mathbb{Z}_+$ and $x$ over $\{1,\dots,N\}$, each of $n = k(N+1)+x$ and $n = (k+1)(N+1)-x$ runs exactly once over $\{ n \ge 1 : (N+1) \nmid n \}$. Summing over $x$ and over all odd words, and using $c(\tilde{\mathcal{R}})=0$ in Lemma \ref{local-konno-sato}, gives \eqref{eq:R-doublesum}.

We next check that everything just written converges, which is where the
hypothesis $u\in\Omega_N$ enters. By Lemma \ref{lem:images} and
\eqref{eq:key-identity}, the weights of the odd words satisfy, 
for $ m \geq 1$, 
\begin{align*}
    |w_m|\,|\lambda_-|^{2m}=\bigl|\alpha-\alpha^{-1}\bigr|
\left(\alpha\,|\lambda_-|^{2}\right)^{m},
    \quad
    |v_m|\,|\lambda_-|^{2m}=\bigl|\alpha-\alpha^{-1}\bigr|
\left(\alpha^{-1}|\lambda_-|^{2}\right)^{m},
\end{align*}
and both geometric ratios are $<1$ exactly when
$|\lambda_-|^{2}<\min(\alpha,\alpha^{-1})$, that is, exactly when
$u\in\Omega_N$. Since in addition $|\lambda_-|<1$ and $2(n+m)\geq2$, the
double series in \eqref{eq:R-doublesum} is dominated by
\begin{align*}
    \bigl(\alpha+\alpha^{-1}\bigr)
    \sum_{n\geq1}|\lambda_-|^{2n}
    \sum_{m\geq0}\left[\left(\alpha|\lambda_-|^{2}\right)^{m}
                      +\left(\alpha^{-1}|\lambda_-|^{2}\right)^{m}\right]
    <\infty ,
\end{align*}
locally uniformly on $\Omega_N$; in particular the individual series
defining the $\log\zeta_{\tilde{\mathcal{R}},Y^{(N)}}(u)$ converge absolutely,
their sum over $\tilde{\mathcal{R}}\in\tilde{\mathfrak{R}}$ may be rearranged
at will, and the product
$\prod_{\tilde{\mathcal{R}}}\zeta_{\tilde{\mathcal{R}},Y^{(N)}}(u)$ converges
absolutely. This justifies \eqref{eq:R-doublesum} and, by Fubini's theorem,
the interchange of summation and integration carried out below. Finally,
$\min(\alpha,\alpha^{-1})\leq1$, so on $\Omega_N$ the segment
$[0,\lambda_-^{2}]$ lies in the disc $\{|\tau|<\min(\alpha,\alpha^{-1})\}$,
on which the integrand of \eqref{eq:def-EN} is holomorphic: its only
singularities in $\mathbb{C}$ are the simple poles $\tau=\alpha^{\pm1}$ and
the zeros of $1-\tau^{N+1}$, all of modulus at least
$\min(\alpha,\alpha^{-1})$. Holomorphy of both sides of \eqref{eq:R-exact}
on $\Omega_N$ follows, $\Omega_N$ being open and $\lambda_-$ holomorphic
there.

Inserting $\frac{\lambda_-^{2r}}{2r} = \frac{1}{2}\int_0^{\lambda_-^2} \tau^{r-1}\mathrm{d}\tau$ with $r = n+m$ and summing the two resulting series,
\begin{align*}
    \log\prod_{\tilde{\mathcal{R}} \in \tilde{\mathfrak{R}}} \zeta_{\tilde{\mathcal{R}}, Y^{(N)}}(u)
    = \frac{1}{2} \int_0^{\lambda_-^2}
      \left( \frac{\tau}{1-\tau} - \frac{\tau^{N+1}}{1-\tau^{N+1}} \right)
      \bigl( W(\tau) + V(\tau) \bigr) \frac{\mathrm{d}\tau}{\tau} .
\end{align*}
By \eqref{eq:WV},
\begin{align*}
    W(\tau) + V(\tau)
    = 2 + \frac{\bigl[ (\alpha-\tau)-(1-\alpha\tau) \bigr]^2}{(1-\alpha\tau)(\alpha-\tau)}
    = 2 + \frac{(\alpha-1)^2 (1+\tau)^2}{(1-\alpha\tau)(\alpha-\tau)} .
\end{align*}
The contribution of the constant $2$ is
\begin{align*}
    \int_0^{\lambda_-^2} \left( \frac{1}{1-\tau} - \frac{\tau^{N}}{1-\tau^{N+1}} \right) \mathrm{d}\tau
    = -\log(1-\lambda_-^2) + \frac{1}{N+1}\log\bigl( 1-\lambda_-^{2(N+1)} \bigr),
\end{align*}
and the remaining contribution is $\mathcal{E}_N$ as defined in \eqref{eq:def-EN}. Exponentiating gives \eqref{eq:R-exact}. Finally $(\alpha-1)^2 = 0$ when $p=q$.
\end{proof}

\begin{rem}
The tails of the odd words do \emph{not} cancel: the factor $e^{\mathcal{E}_N}$ in \eqref{eq:R-exact} is not negligible in the double scaling limit, and contributes the finite factor $\varkappa^2/s^2$ to \eqref{eq:massive_asymptotic} below. This is in sharp contrast with the even words, for which the cancellation is exact by Lemma \ref{lem:translations}. Both phenomena are governed by the same identity \eqref{eq:key-identity}.
\end{rem}

\subsection{Convergence to the Massive Dirac-Type Trace Formula}

\begin{theorem}[Convergence to the Continuum Massive Dirac-Type Trace Formula]\label{thm:double_scaling}
Let $M\in\mathbb{R}$ be fixed and let $N>|M|$, and consider the double
scaling limit $N\to\infty$ with $u=e^{-s/N}$ and
$2\sqrt{pq}=\sqrt{1-M^{2}/N^{2}}$. Let $s$ range in the region
$\mathcal{S}_M$ of \eqref{eq:SM}, and put $\varkappa=\sqrt{s^{2}+M^{2}}$,
$\mathrm{Re}(\varkappa)>0$, as in \eqref{eq:varkappa}. 
Then the global zeta
function of the Szegedy walk converges to the continuous zeta
function of the massive Dirac type: the spectral determinant yields
\begin{align}\label{eq:massive_spectral}
    \lim_{N \to \infty} \det(I_{2N+2} - u\mathbf{U}^{(N)})^{-1}
    = \frac{\varkappa}{s} e^{s-\varkappa} Z^{\mathrm{cont}}(s, M).
\end{align}

Furthermore, the geometric Euler product over the translation set $\tilde{\mathfrak{T}}$ reconstructs the continuous global zeta function via the term-by-term local correspondence
\begin{align}\label{eq:massive_translation}
    \lim_{N \to \infty} \zeta_{\tilde{\mathcal{T}}_k, Y^{(N)}}(u)\, \zeta_{\tilde{\mathcal{T}}_{-k}, Y^{(N)}}(u) = \zeta_{\gamma_0^k}^{\mathrm{cont}}(s, M)
    \qquad (k \in \mathbb{N}),
\end{align}
which leads to
\begin{align}\label{eq:massive_geometric_T}
    \lim_{N \to \infty} \prod_{\tilde{\mathcal{T}} \in \tilde{\mathfrak{T}}} \zeta_{\tilde{\mathcal{T}}, Y^{(N)}}(u)
    = Z^{\mathrm{cont}}(s, M) = \prod_{k=1}^\infty \zeta_{\gamma_0^k}^{\mathrm{cont}}(s, M).
\end{align}
The bulk vacuum contribution ($id$) converges to the continuous free field factor:
\begin{align}\label{eq:massive_identity}
    \lim_{N \to \infty} \zeta_{id, Y^{(N)}}(u) = e^{s-\varkappa}.
\end{align}
The odd-length reflections and the boundary anomaly are given exactly by
\begin{align}\label{eq:massive_reflection_exact}
\prod_{\tilde{\mathcal{R}} \in \tilde{\mathfrak{R}}} \zeta_{\tilde{\mathcal{R}}, Y^{(N)}}(u)
      &= \frac{ \bigl(1 - \lambda_-^{2(N+1)}\bigr)^{\frac{1}{N+1}} }{ 1 - \lambda_-^2 }\, e^{\mathcal{E}_N},
    \\ \label{eq:massive_boundary_exact}
    \zeta_{\partial Y^{(N)}}(u)
      &= \frac{(1-\lambda_-^2)^2}{1-u^2} \bigl( 1 - \lambda_-^{2(N+1)} \bigr)^{-\frac{2}{N+1}} e^{-\mathcal{E}_N},
\end{align}
with $\mathcal{E}_N$ as in \eqref{eq:def-EN}. In the thermodynamic limit their individual contributions diverge and decay, respectively:
\begin{align}\label{eq:massive_asymptotic}
    \lim_{N\to\infty} \mathcal{E}_N = \log \frac{\varkappa^2}{s^2},
    \quad
    \prod_{\tilde{\mathcal{R}} \in \tilde{\mathfrak{R}}} \zeta_{\tilde{\mathcal{R}}, Y^{(N)}}(u) \sim \frac{N\varkappa}{2s^2},
    \quad
    \zeta_{\partial Y^{(N)}}(u) \sim \frac{2s}{N}
    \quad (N \to \infty).
\end{align}
However, their singularities cancel each other in the product, yielding the finite boundary scattering phase
\begin{align}\label{eq:massive_reflection_limit}
    \lim_{N \to \infty} \left[ \zeta_{\partial Y^{(N)}}(u) \prod_{\tilde{\mathcal{R}} \in \tilde{\mathfrak{R}}} \zeta_{\tilde{\mathcal{R}}, Y^{(N)}}(u)\right]
    = \lim_{N \to \infty} \frac{2s}{N}\cdot\frac{N\varkappa}{2s^2} = \frac{\varkappa}{s}.
\end{align}
\end{theorem}

\begin{rem}\label{rem:domain}
The restriction $s\in\mathcal{S}_M$ enters only through the odd words. The
limits \eqref{eq:massive_spectral}, \eqref{eq:massive_translation},
\eqref{eq:massive_geometric_T} and \eqref{eq:massive_identity} remain valid
on the whole half plane $\{\mathrm{Re}(s)>0\}$: their proofs (Steps 1, 3 and
4 below) use only the closed forms
\eqref{eq:exact_discrete_factorization}, \eqref{eq:prodTexact} and
\eqref{eq:zeta-id-exact}, all of which are holomorphic on $\{0<|u|<1\}$,
together with \eqref{eq:gamma-limit}. By contrast, the factor
$e^{\mathcal{E}_N}$ of \eqref{eq:massive_reflection_exact} is defined by a series which converges only on the smaller domain
$\Omega_N$ of \eqref{eq:OmegaN} (Lemma \ref{lem:reflections}), and
$\{u=e^{-s/N}\}\subset\Omega_N$ for large $N$ precisely when
$\mathrm{Re}(\varkappa)>|M|$. Since the limits in
\eqref{eq:massive_asymptotic} and \eqref{eq:massive_reflection_limit} are
holomorphic on $\{\mathrm{Re}(s)>0\}$ as well, one expects these identities
to persist there by analytic continuation; we do not pursue this, as the
factors themselves have to be re-defined off $\mathcal{S}_M$.
\end{rem}

\begin{proof}
\textit{Step 1: the exact discrete factorization and its limit.}
By the expression \eqref{explicitUN} of $\det(I_{2N+2}-u\mathbf{U}^{(N)})$ we may represent it using the Chebyshev polynomial of the second kind, $U_N(\xi_N) = 2^N \prod_{m=1}^N (\xi_N - \cos\frac{m\pi}{N+1})$; in terms of the characteristic roots $\lambda_\pm$,
\begin{align*}
    U_N(\xi_N) = \frac{\lambda_+^{N+1}}{\lambda_+-\lambda_-} \left( 1 - \lambda_-^{2(N+1)} \right).
\end{align*}
Since $\det(I_{2N+2}-u\mathbf{U}^{(N)}) = (1-u^2)(2u\sqrt{pq})^N U_N(\xi_N)$, we obtain
\begin{align*}
    \det(I_{2N+2} - u\mathbf{U}^{(N)})^{-1}
    = C_N(u) \left( 1 - \lambda_-^{2(N+1)} \right)^{-1}
\end{align*}
with 
\begin{align}\label{eq:exact_discrete_factorization}
      C_N(u) := \frac{(2u\sqrt{pq})^{-N} (\lambda_+-\lambda_-) \lambda_+^{-(N+1)}}{1-u^2}.
\end{align}

We now evaluate the double scaling limit for $s \in \mathbb{C}$ with $\mathrm{Re}(s)>0$. Noting that
\begin{align*}
    \xi_N = \frac{\cosh(s/N)}{\sqrt{1-M^2/N^2}},
\end{align*}
and recalling that $\gamma_N = \mathrm{arcosh}(\xi_N)$ is taken with $\mathrm{Re}(\gamma_N)>0$, so that $\lambda_\pm = e^{\pm\gamma_N}$ and $|\lambda_-|<1$, we expand
\begin{align*}
    \lim_{N \to \infty} N^2 (\cosh \gamma_N - 1)
    = \lim_{N \to \infty} N^2 \left( \frac{\cosh(s/N) - \sqrt{1-M^2/N^2}}{\sqrt{1-M^2/N^2}} \right)
    = \frac{\varkappa^2}{2}.
\end{align*}
Using $\cosh\gamma_N - 1 = 2\sinh^2(\gamma_N/2)$ and $\lim_{w\to0}\frac{\sinh w}{w}=1$, we deduce $\lim_{N\to\infty}(N\gamma_N)^2 = \varkappa^2$; since both $\mathrm{Re}(N\gamma_N)>0$ and $\mathrm{Re}(\varkappa)>0$, the branches match and
\begin{align}\label{eq:gamma-limit}
    \lim_{N\to\infty} N\gamma_N = \varkappa,
    \qquad\text{hence}\qquad
    \lim_{N\to\infty} \lambda_\pm^{\,N} = \lim_{N\to\infty} \lambda_\pm^{\,N+1} = e^{\pm\varkappa}.
\end{align}
These convergences are uniform on compact subsets of $\{\mathrm{Re}(s)>0\}$, all functions involved being holomorphic there. In particular
\begin{align}\label{eq:Z-limit}
    \lim_{N\to\infty} \left( 1 - \lambda_-^{2(N+1)} \right)^{-1} = \left( 1 - e^{-2\varkappa} \right)^{-1} = Z^{\mathrm{cont}}(s,M),
\end{align}
while, from $(2u\sqrt{pq})^{-N} = e^{s}(1-M^2/N^2)^{-N/2} \to e^{s}$, $N(\lambda_+-\lambda_-) = 2N\sinh\gamma_N \to 2\varkappa$ and $N(1-u^2) \to 2s$,
\begin{align}\label{eq:prefactor_limit_final}
    \lim_{N\to\infty} C_N(u)
    = \lim_{N\to\infty} \frac{N(\lambda_+-\lambda_-)}{N(1-u^2)} \, (2u\sqrt{pq})^{-N} \lambda_+^{-(N+1)}
    = \frac{2\varkappa}{2s} \cdot e^{s} \cdot e^{-\varkappa}
    = \frac{\varkappa}{s} e^{s-\varkappa}.
\end{align}
Combining \eqref{eq:exact_discrete_factorization}, \eqref{eq:Z-limit} and \eqref{eq:prefactor_limit_final} establishes \eqref{eq:massive_spectral}.

\textit{Step 2: validity of the factorization at $u=e^{-s/N}$.}
The product expansion \eqref{eq:geometric_decomposition_proof} was
established in Section \ref{sec:6} for $|u|$ sufficiently small, and we
first extend it to the domain $\Omega_N$ of \eqref{eq:OmegaN}. On $\Omega_N$
every factor on the right-hand side of
\eqref{eq:geometric_decomposition_proof} is holomorphic: this is Lemma
\ref{lem:reflections} for the odd words, while for the even words
$\sum_{k\geq1}\frac{N}{k(N+1)}|\lambda_-|^{2k(N+1)}$ converges locally
uniformly on all of $\mathbb{D}^{*}$ and controls \eqref{eq:prodTexact}, and
for the identity factor Lemmas \ref{local-konno-sato} and
\ref{lem:freekernel} give
$\zeta_{id,Y^{(N)}}(u)=\bigl((1+\lambda_-^{2})/(1+u^{2})\bigr)^{N}$, again
holomorphic on $\mathbb{D}^{*}$. Both sides of
\eqref{eq:geometric_decomposition_proof} are therefore holomorphic on the
connected open set $\Omega_N$ and agree near the origin, so they agree on
all of $\Omega_N$.

It remains to check that $u=e^{-s/N}\in\Omega_N$ for large $N$. Since
$|\lambda_-|^{2}=e^{-2\mathrm{Re}(\gamma_N)}$ and
$\min(\alpha,\alpha^{-1})=e^{-|\log\alpha|}$, the defining inequality of
\eqref{eq:OmegaN} reads
\begin{align*}
    2N\,\mathrm{Re}(\gamma_N)>N\,|\log\alpha| .
\end{align*}
By \eqref{eq:gamma-limit} the left-hand side converges to
$2\,\mathrm{Re}(\varkappa)$, and
$N\log\alpha=N\log\frac{1+M/N}{1-M/N}\to2M$, so the right-hand side
converges to $2|M|$; both convergences are uniform on compact subsets of
$\mathcal{S}_M$, on which $2\,\mathrm{Re}(\varkappa)>2|M|$ strictly. Hence,
for every compact $K\subset\mathcal{S}_M$ there is $N_K$ such that
$e^{-s/N}\in\Omega_N$ for all $s\in K$ and all $N\geq N_K$. All the
identities below are therefore valid at $u=e^{-s/N}$, for $s\in\mathcal{S}_M$
and $N$ large.

\textit{Step 3: the translations.}
By Lemma \ref{lem:translations} the elements of $\tilde{\mathfrak{T}}$ act on the bulk as exact translations by $2k(N+1)$, so that, with no approximation,
\begin{align*}
    \log \zeta_{\tilde{\mathcal{T}}_k, Y^{(N)}}(u) = \log \zeta_{\tilde{\mathcal{T}}_{-k}, Y^{(N)}}(u)
    = \frac{N}{2k(N+1)} \lambda_-^{\,2k(N+1)} \qquad (k \in \mathbb{N}).
\end{align*}
By \eqref{eq:gamma-limit}, $\lambda_-^{\,2k(N+1)} = e^{-2k(N+1)\gamma_N} \to e^{-2k\varkappa}$, while $N/(N+1) \to 1$; hence
\begin{align*}
    \lim_{N\to\infty} \log \bigl( \zeta_{\tilde{\mathcal{T}}_k, Y^{(N)}}(u) \zeta_{\tilde{\mathcal{T}}_{-k}, Y^{(N)}}(u) \bigr)
    = \frac{e^{-2k\varkappa}}{k},
\end{align*}
which is \eqref{eq:massive_translation} by \eqref{eq:continuous_local_zeta}. Taking the product over $k \ge 1$, or equivalently letting $N \to \infty$ in the exact formula \eqref{eq:prodTexact}, gives
\begin{align*}
    \lim_{N\to\infty} \prod_{\tilde{\mathcal{T}} \in \tilde{\mathfrak{T}}} \zeta_{\tilde{\mathcal{T}}, Y^{(N)}}(u)
    = \exp \left( \sum_{k=1}^\infty \frac{e^{-2k\varkappa}}{k} \right)
    = \bigl( 1 - e^{-2\varkappa} \bigr)^{-1} = Z^{\mathrm{cont}}(s,M),
\end{align*}
which is \eqref{eq:massive_geometric_T}.

\textit{Step 4: the bulk vacuum.}
Here $c(id) = N$, so Lemma \ref{local-konno-sato} gives $\zeta_{id, Y^{(N)}}(u) = (1+u^2)^{-N} \zeta_{id, \tilde{\mathbf{P}}^{(N)}}(z)$, while the diagonal case of Lemma \ref{lem:freekernel} gives $\log \zeta_{id, \tilde{\mathbf{P}}^{(N)}}(z) = N \log(1+\lambda_-^2)$. Since $1+\lambda_-^2 = \lambda_-(\lambda_-+\lambda_+) = 2\xi_N \lambda_- = \lambda_- \frac{1+u^2}{2u\sqrt{pq}}$,
\begin{align}\label{eq:zeta-id-exact}
    \zeta_{id, Y^{(N)}}(u) = \left( \frac{1+\lambda_-^2}{1+u^2} \right)^{N}
    = \bigl( 2u\sqrt{pq}\, \lambda_+ \bigr)^{-N}
    = (2u\sqrt{pq})^{-N} e^{-N\gamma_N} .
\end{align}
By $(2u\sqrt{pq})^{-N} \to e^{s}$ and \eqref{eq:gamma-limit}, this converges to $e^{s-\varkappa}$, which is \eqref{eq:massive_identity}.

\textit{Step 5: the odd words.}
Formula \eqref{eq:massive_reflection_exact} is Lemma \ref{lem:reflections}.
We compute $\lim_{N\to\infty}\mathcal{E}_N$. If $M=0$ then $\alpha=1$ and
$\mathcal{E}_N\equiv0=\log(\varkappa^{2}/s^{2})$, so we assume $M\neq0$ and
fix a compact set $K\subset\mathcal{S}_M$; all estimates below are uniform
for $s\in K$ and $N$ large.

The integrand of \eqref{eq:def-EN} is holomorphic on the disc
$\{|\tau|<\min(\alpha,\alpha^{-1})\}$, whose only singularities on
$\mathbb{C}$ are the simple poles $\tau=\alpha^{\pm1}$ and the zeros of
$1-\tau^{N+1}$ on $|\tau|=1$. By Step 2 the segment $[0,\lambda_-^{2}]$ lies
in that disc, and so does the curve
$\{\tau=e^{-2\gamma_N r}:r\geq1\}$, since
$\left|e^{-2\gamma_N r}\right|=|\lambda_-|^{2r}\leq|\lambda_-|^{2}$ for
$r\geq1$. As the disc is simply connected we may integrate along the latter:
substituting $\tau=e^{-2\gamma_N r}$, $\mathrm{d}\tau=-2\gamma_N
e^{-2\gamma_N r}\,\mathrm{d}r$, and reversing the orientation,
\begin{align*}
    \begin{aligned}
&\mathcal{E}_N=\int_1^{\infty}h_N(r)\,\mathrm{d}r, \,\,\text{where} \\
  &  h_N(r):=(\alpha-1)^{2}\,\gamma_N\,
      \frac{(1+\tau)^{2}\,e^{-2\gamma_N r}}{(1-\alpha\tau)(\alpha-\tau)}
      \left(\frac{1}{1-\tau}-\frac{\tau^{N}}{1-\tau^{N+1}}\right)
      \bigg|_{\tau=e^{-2\gamma_N r}} .
    \end{aligned}
\end{align*}
Write $\alpha=e^{\theta_N}$, so that $N\theta_N\to2M$, and recall
$N\gamma_N\to\varkappa$ from \eqref{eq:gamma-limit}. For each fixed $r\geq1$
we have $\tau\to1$, and
\begin{align*}
  &  (\alpha-1)^{2}\sim\theta_N^{2},\quad
    (1+\tau)^{2}\to4,\quad
    1-\tau\sim2\gamma_N r,\\
   & \alpha-\tau\sim\theta_N+2\gamma_N r,\quad
    1-\alpha\tau\sim2\gamma_N r-\theta_N,
\end{align*}
while $\tau^{N}/(1-\tau^{N+1})=O(1)$ so that the second bracket contributes
$O(\gamma_N)=O(1/N)$ relative to the first. Hence
\begin{align*}
    h_N(r)\longrightarrow
    \frac{2M^{2}}{r\left(\varkappa^{2}r^{2}-M^{2}\right)}
    \qquad(N\to\infty),\quad r\geq1 .
\end{align*}
For the domination, note that $\mathrm{Re}(\varkappa)>|M|>0$ on $K$ implies
that $\cos\arg(\gamma_N)=\mathrm{Re}(\gamma_N)/|\gamma_N|$ is bounded below
by a positive constant, whence $|1-e^{-2\gamma_N r}|\geq c\min(1,|\gamma_N|r)$
and $|\alpha^{\pm1}-\tau|\geq c\,\min\{1,|\gamma_N|r+|\theta_N|\}$ for all
$r\geq1$ and all large $N$; consequently $|h_N(r)|\leq C r^{-3}$ on
$[1,\infty)$. Dominated convergence and
\begin{align*}
    \frac{\mathrm{d}}{\mathrm{d}r}
    \log\frac{\varkappa^{2}r^{2}-M^{2}}{\varkappa^{2}r^{2}}
    =\frac{2M^{2}}{r\left(\varkappa^{2}r^{2}-M^{2}\right)},
    \qquad
    \left|\frac{M^{2}}{\varkappa^{2}r^{2}}\right|<1\ \ (r\geq1),
\end{align*}
the latter because $\mathrm{Re}(\varkappa)>|M|$ forces $|\varkappa|>|M|$, now
give
\begin{align*}
\begin{aligned}
    \lim_{N\to\infty}\mathcal{E}_N
    &=\int_1^{\infty}\frac{2M^{2}\,\mathrm{d}r}{r\left(\varkappa^{2}r^{2}-M^{2}\right)}
    =\left[\log\frac{\varkappa^{2}r^{2}-M^{2}}{\varkappa^{2}r^{2}}\right]_{r=1}^{r=\infty} \\
    &=-\log\frac{\varkappa^{2}-M^{2}}{\varkappa^{2}}
    =\log\frac{\varkappa^{2}}{s^{2}},
    \end{aligned}
\end{align*}
the principal branch being used throughout since
$1-M^{2}/(\varkappa r)^{2}$ stays in the disc of radius $1$ centred at $1$.
Since $N(1-\lambda_-^{2})=N(1-e^{-2\gamma_N})\to2\varkappa$ and
$(1-\lambda_-^{2(N+1)})^{1/(N+1)}\to1$,
\eqref{eq:massive_reflection_exact} gives
\begin{align*}
\prod_{\tilde{\mathcal{R}} \in \tilde{\mathfrak{R}}} \zeta_{\tilde{\mathcal{R}}, Y^{(N)}}(u)
    \sim \frac{N}{2\varkappa} \cdot \frac{\varkappa^2}{s^2} = \frac{N\varkappa}{2s^2},
\end{align*}
the second assertion of \eqref{eq:massive_asymptotic}.

\textit{Step 6: the boundary anomaly and conclusion.}
The boundary factor $\zeta_{\partial Y^{(N)}}(u)$ admits no independent closed form, and is determined by the exact factorization \eqref{eq:geometric_decomposition_proof}. Dividing \eqref{eq:exact_discrete_factorization} by the identity factor \eqref{eq:zeta-id-exact} and using $\lambda_+\lambda_-=1$,
\begin{align*}
    \frac{C_N(u)}{\zeta_{id, Y^{(N)}}(u)}
    = \frac{(\lambda_+-\lambda_-)\lambda_+^{-1}}{1-u^2}
    = \frac{1-\lambda_-^2}{1-u^2},
\end{align*}
so that, dividing further by the translation factor \eqref{eq:prodTexact},
\begin{align}\label{eq:bdry-quotient}
\begin{aligned}
    \zeta_{\partial Y^{(N)}}(u) \prod_{\tilde{\mathcal{R}} \in \tilde{\mathfrak{R}}} \zeta_{\tilde{\mathcal{R}}, Y^{(N)}}(u)
    &= \frac{C_N(u)\bigl(1-\lambda_-^{2(N+1)}\bigr)^{-1}}
           {\zeta_{id,Y^{(N)}}(u)\bigl(1-\lambda_-^{2(N+1)}\bigr)^{-\frac{N}{N+1}}} \\
    &= \frac{1-\lambda_-^2}{1-u^2}\bigl( 1 - \lambda_-^{2(N+1)} \bigr)^{-\frac{1}{N+1}}.
\end{aligned}
\end{align}
Dividing \eqref{eq:bdry-quotient} by \eqref{eq:massive_reflection_exact} yields the exact closed form \eqref{eq:massive_boundary_exact} of the boundary anomaly, and hence, by Step 5 and $N(1-u^2)\to 2s$,
\begin{align*}
    \zeta_{\partial Y^{(N)}}(u)
    \sim \frac{(2\varkappa/N)^2}{2s/N} \cdot \frac{s^2}{\varkappa^2}
    = \frac{2s}{N},
\end{align*}
the third assertion of \eqref{eq:massive_asymptotic}. Finally, letting $N \to \infty$ in \eqref{eq:bdry-quotient} and using $N(1-\lambda_-^2)\to2\varkappa$, $N(1-u^2)\to2s$ and $(1-\lambda_-^{2(N+1)})^{1/(N+1)}\to1$, we obtain \eqref{eq:massive_reflection_limit}. Multiplying \eqref{eq:massive_reflection_limit}, \eqref{eq:massive_identity} and \eqref{eq:massive_geometric_T} recovers \eqref{eq:massive_spectral}, in agreement with Step 1.
\end{proof}

\subsection{Identification of the Limit as a Spectral Determinant}
\label{sec:8.6}

The three limits \eqref{eq:massive_geometric_T}, \eqref{eq:massive_identity}
and \eqref{eq:massive_reflection_limit} were obtained separately, one for
each class of words. We now observe that their product is a classical
object, and that this identifies each of the three factors.

Let $\mathsf{D}$ denote the massive Dirac operator on the interval $[0,1]$,
\begin{align}\label{eq:dirac-op}
    \mathsf{D}:=-i\sigma_1\frac{\mathrm{d}}{\mathrm{d}x}+M\sigma_3,
    \qquad
    \mathsf{D}^{2}=-\frac{\mathrm{d}^{2}}{\mathrm{d}x^{2}}+M^{2},
\end{align}
so that $\mathsf{D}^{2}+s^{2}=-\frac{\mathrm{d}^{2}}{\mathrm{d}x^{2}}+\varkappa^{2}$
with $\varkappa=\sqrt{s^{2}+M^{2}}$ as in \eqref{eq:varkappa}
and we write $\mathsf{D}_{0}$ for $\mathsf{D}$ at $M=0$, so that
$\mathsf{D}_{0}^{2}+s^{2}=-\mathrm{d}^{2}/\mathrm{d}x^{2}+s^{2}$
We denote by $\det_\zeta$ the zeta-regularized determinant, defined as usual
by $\det_\zeta A=\exp(-\zeta_A'(0))$ with
$\zeta_A(z)=\sum_j\lambda_j^{-z}$. 
The determinant we need is the classical
Gelfand--Yaglom evaluation \cite{GelfandYaglom1960}; for the
zeta-regularized form used here see
\cite{LevitSmilansky1977, Forman1987, BFK1992} and, for the explicit
one-dimensional formulas, \cite[Ch.~2]{Kirsten2002},
\cite{KirstenMcKane2003} or the survey \cite{Dunne2008}. In the case at
hand it is immediate:

\begin{lemma}\label{lem:gy}
Let $m\in\mathbb{C}$ with $m^{2}\neq-(n\pi)^{2}$ for every $n\geq1$, the
right-hand side below being read as its limit $2$ when $m=0$. Then, 
\begin{align}\label{eq:gelfand-yaglom}
    \det\nolimits_\zeta\Bigl(-\frac{\mathrm{d}^{2}}{\mathrm{d}x^{2}}+m^{2}\Bigr)_{\mathrm{Dir},[0,1]}
    =\frac{2\sinh m}{m}.
\end{align}
\end{lemma}

\begin{proof}
The Dirichlet spectrum of $-\mathrm{d}^{2}/\mathrm{d}x^{2}$ on $[0,1]$ is
$\{(n\pi)^{2}\}_{n\geq1}$, so its zeta function is
$\zeta(z)=\pi^{-2z}\zeta_{R}(2z)$ with $\zeta_{R}$ the Riemann zeta
function. Using $\zeta_{R}(0)=-\frac12$ and
$\zeta_{R}'(0)=-\frac12\log(2\pi)$,
\begin{align*}
    \zeta'(0)=-2\log\pi\cdot\zeta_{R}(0)+2\zeta_{R}'(0)
      =\log\pi-\log(2\pi)=-\log 2 ,
\end{align*}
whence $\det_\zeta(-\mathrm{d}^{2}/\mathrm{d}x^{2})_{\mathrm{Dir},[0,1]}=2$.
Write $\zeta_{m}(z)=\sum_{n\geq1}\bigl((n\pi)^{2}+m^{2}\bigr)^{-z}$, so that
$\zeta_{0}=\zeta$, and put $a_{n}:=m^{2}/(n\pi)^{2}$. For $\mathrm{Re}(z)$
large,
\begin{align*}
    \zeta_{m}(z)-\zeta_{0}(z)
    =\sum_{n\geq1}(n\pi)^{-2z}\bigl[(1+a_{n})^{-z}-1\bigr],
\end{align*}
each summand vanishing at $z=0$ with derivative $-\log(1+a_{n})$ there.
Since $\sum_{n}|a_{n}|<\infty$, the series converges locally uniformly in a
neighbourhood of $z=0$ and defines a holomorphic function there, so the
difference may be differentiated term by term:
\begin{align*}
    \zeta_{m}'(0)-\zeta_{0}'(0)
    =-\sum_{n\geq1}\log\left(1+\frac{m^{2}}{n^{2}\pi^{2}}\right)
    =-\log\frac{\sinh m}{m},
\end{align*}
by the Weierstrass product for $\sinh$. Combining with
$\det_\zeta(-\mathrm{d}^{2}/\mathrm{d}x^{2})_{\mathrm{Dir},[0,1]}=2$ gives
\eqref{eq:gelfand-yaglom}.
\end{proof}

\begin{rem}
The right-hand side of \eqref{eq:gelfand-yaglom} is $2y(1)$, where $y$
solves $y''=m^{2}y$ with $y(0)=0$, $y'(0)=1$; this is the Gelfand--Yaglom
form of the determinant, and the constant $2$ is
$\det_\zeta(-\mathrm{d}^{2}/\mathrm{d}x^{2})_{\mathrm{Dir},[0,1]}$. It is
this Wronskian normalization that will be matched with the boundary factor
of the discrete decomposition in Remark \ref{rem:dictionary}.
\end{rem}

\begin{theorem}[The limit is a spectral determinant]\label{thm:identification}
In the double scaling limit of Theorem \ref{thm:double_scaling}, but for
every $s$ with $\mathrm{Re}(s)>0$,
\begin{align}\label{eq:identification}
    \lim_{N\to\infty}\det\bigl(I_{2N+2}-u\mathbf{U}^{(N)}\bigr)^{-1}
    =\frac{e^{s}}{s}\,
     \Bigl[\det\nolimits_\zeta\bigl(\mathsf{D}^{2}+s^{2}\bigr)_{\mathrm{Dir},[0,1]}\Bigr]^{-1}
    =\frac{e^{s}\,\varkappa}{2s\sinh\varkappa}.
\end{align}
Writing $\mathbf{U}^{(N)}_{M}$ for the walk with the scaling
\eqref{eq:scaling} at mass $M$, the two non-universal prefactors cancel in
the massive-to-massless ratio, which is therefore an identity between
spectral determinants alone:
\begin{align}\label{eq:identification-ratio}
    \lim_{N\to\infty}
    \frac{\det\bigl(I_{2N+2}-u\mathbf{U}^{(N)}_{M}\bigr)^{-1}}
         {\det\bigl(I_{2N+2}-u\mathbf{U}^{(N)}_{0}\bigr)^{-1}}
=\frac{\det_\zeta\bigl(\mathsf{D}_{0}^{2}+s^{2}\bigr)_{\mathrm{Dir},[0,1]}}      {\det_\zeta\bigl(\mathsf{D}^{2}+s^{2}\bigr)_{\mathrm{Dir},[0,1]}}
    =\frac{\varkappa\sinh s}{s\sinh\varkappa}.
\end{align}
\end{theorem}

\begin{proof}
By \eqref{eq:massive_spectral}, which by Remark \ref{rem:domain} is valid on
all of $\{\mathrm{Re}(s)>0\}$, the limit equals
$\frac{\varkappa}{s}e^{s-\varkappa}Z^{\mathrm{cont}}(s,M)$, and 
\begin{align*}
    e^{-\varkappa}Z^{\mathrm{cont}}(s,M)
    =\frac{e^{-\varkappa}}{1-e^{-2\varkappa}}
    =\frac{1}{e^{\varkappa}-e^{-\varkappa}}
    =\frac{1}{2\sinh\varkappa},
\end{align*}
which gives the second equality in \eqref{eq:identification}; the first is
\eqref{eq:gelfand-yaglom} with $m=\varkappa$, together with
$\mathcal{D}^{2}+s^{2}=-\mathrm{d}^{2}/\mathrm{d}x^{2}+\varkappa^{2}$.
Taking the quotient with the same expression at $M=0$, where
$\varkappa=s$, cancels $e^{s}/s$ and gives
\eqref{eq:identification-ratio}.
\end{proof}

\begin{rem}[What each factor of the discrete decomposition becomes]
\label{rem:dictionary}
Expanding the right-hand side of \eqref{eq:identification} as
\begin{align*}
    \Bigl[\det\nolimits_\zeta\bigl(\mathcal{D}^{2}+s^{2}\bigr)_{\mathrm{Dir}}\Bigr]^{-1}
    =\frac{\varkappa}{2\sinh\varkappa}
    =\underbrace{\varkappa\vphantom{\frac11}}_{\text{Wronskian normalization}}
     \cdot\underbrace{e^{-\varkappa}}_{\text{vacuum}}
     \cdot\underbrace{\bigl(1-e^{-2\varkappa}\bigr)^{-1}}_{\text{periodic orbits}},
\end{align*}
and comparing with \eqref{eq:geometric_decomposition_proof}, the three
classes of words of $\tilde{\mathcal{G}}_{0,N+1}$ are seen to converge to
the three canonical constituents of the Gelfand--Yaglom determinant:
\begin{itemize}
\item the even words $\tilde{\mathfrak{T}}$ converge to
  $\bigl(1-e^{-2\varkappa}\bigr)^{-1}$, the periodic-orbit (Euler) factor
  of the primitive round trip $\gamma_{0}$;
\item the identity converges to $e^{s-\varkappa}$, the vacuum factor
  $e^{-l(\gamma_{0})\varkappa/2}$ times the lattice counterterm $e^{s}$;
\item the odd words $\tilde{\mathfrak{R}}$, together with the boundary
  factor $\zeta_{\partial Y^{(N)}}$, converge to $\varkappa/s$, the
  Wronskian (Gelfand--Yaglom) normalization measured against its value at
  $M=0$.
\end{itemize}
In particular the boundary scattering phase $\varkappa/s$ of
\eqref{eq:massive_reflection_limit}, which \eqref{eq:bdry-quotient} produces as a quotient, is exactly the
Gelfand--Yaglom normalization of the limiting operator, measured against its
massless value. Equivalently, since $s=\sqrt{\varkappa^{2}-M^{2}}$,
\begin{align}\label{eq:jacobian}
    \frac{\varkappa}{s}=\frac{\mathrm{d}s}{\mathrm{d}\varkappa}\bigg|_{\varkappa^{2}=s^{2}+M^{2}},
\end{align}
the Jacobian of the massive dispersion relation, i.e.\ the density of
states of the massive problem relative to the massless one. This is the
precise sense in which the bulk--boundary separation of the discrete trace
formula survives the continuum limit: the bulk supplies the length spectrum
and the vacuum, while the boundary supplies the dispersion relation.
\end{rem}

\begin{rem}[The factor $1/s$]\label{rem:zero-mode}
The prefactor $e^{s}/s$ in \eqref{eq:identification} is not part of the
continuum determinant. The factor $e^{s}=\lim(2u\sqrt{pq})^{-N}$ is the
usual per-site lattice counterterm. The factor $1/s$ is the contribution of
the eigenvalue $\mu_{0}=1$ of $\mathbf{U}^{(N)}$, which is present for every
$N$ and every $p,q$: it contributes $(1-u)^{-1}$ to
$\det(I-u\mathbf{U}^{(N)})^{-1}$, and $N(1-u)\to s$. It is the zero mode of
the limiting operator, and it is the reason why
\eqref{eq:identification-ratio} rather than \eqref{eq:identification} is the
normalization-free form of the statement.
\end{rem}
\section{Concluding Remarks}\label{sec:9}

We have shown that the zeta function of a walk on a finite path graph with
two reflecting ends factorizes exactly into local contributions indexed by
the infinite dihedral group generated by the boundary reflections, and that
this factorization is stable under the continuum limits one may take.

Two constructions carry the argument. The generalized reflection principle
of Section \ref{sec:5} replaces the point reflection, which is available
only for $p=q$, by an involution of $\mathbb{R}^{\mathbb{Z}}$ adapted to the
drift; summing the free kernel over the group it generates gives the
discrete trace formula. The local Konno--Sato correspondence of Section
\ref{sec:6} then transports that formula to the Szegedy walk term by term,
so that the classical/quantum arithmetic relation is a statement about each
bulk factor separately, while the algebraic prefactor of the quantum
determinant is absorbed into a single boundary zeta function.

The three limits we computed each isolate a different part of the group. In
the thermodynamic limit the Mahler measure representation of the normalized
zeta function is precisely the local factor of the identity element, so that
the infinite-graph limit records only the free walk and forgets both the
non-trivial group elements and the boundary. In the scaling limits, the
translations and the reflections conspire to produce the Poisson summation
formula in the diffusive regime $n\asymp N^{2}$ and the Dirac comb identity
in the ballistic regime $n\asymp N$; the difference of time scale is exactly
the difference between the heat kernel and the wave kernel on $S^{1}$, and
the quantum identity holds already at finite $N$. In the double scaling
limit the asymmetry becomes a mass, and the decomposition is inherited: the
even words give the periodic-orbit Euler product of the one-dimensional
orbifold, the identity gives the free-field factor, and the odd words
together with the boundary give a finite scattering phase, with the
divergence of the former cancelled exactly by the decay of the latter.
Their product, moreover, is not merely finite but classical: it is the
zeta-regularized determinant of the massive Dirac operator squared on the
interval, so that the discrete decomposition is matched constituent by
constituent with the Gelfand--Yaglom factorization of that determinant.

What we find suggestive is that none of these continuum features has to be
imported. The length spectrum of the limiting orbifold is the word length in
$\mathcal{G}_{0,N+1}$; the heat/wave dichotomy is a choice of time scale in a
single discrete identity; the mass is the drift $p-q$; and the Euler product
of the limiting zeta function is the image of a product which is already
exact on the lattice. In this sense the discrete trace formula behaves as a
pre-geometric object, of which the continuum trace formulas are degenerations.

Several questions remain. 
The scattering phase $\varkappa/s$ has been
identified in Remark \ref{rem:dictionary} as the Gelfand--Yaglom
normalization of the limiting operator, but the boundary zeta function
$\zeta_{\partial Y^{(N)}}(u)$ itself is still characterized only through the
factorization, and an intrinsic closed form for it at finite $N$ would be
desirable. 
It would also be natural to ask which of the present structure
survives on trees, on finite graphs with several boundary components, or for
coins other than the Szegedy one, where the relevant group is no longer
dihedral.



\begin{thebibliography}{99}


\bibitem{BianePitmanYor2001}
Biane, P., Pitman, J. and Yor, M. (2001). 
Probability laws related to the Jacobi theta and Riemann zeta functions, and Brownian excursions. 
{\em Bulletin of the American Mathematical Society}, 38(4), 435-465.

\bibitem{BFK1992}
Burghelea, D., Friedlander, L. and Kappeler, T. (1992).
Mayer--Vietoris type formula for determinants of elliptic differential
operators. {\em J. Funct. Anal.}, 107(1), 34--65.

\bibitem{Chavel1984}
Chavel, I. (1984). {\em Eigenvalues in Riemannian Geometry}. Academic Press.

\bibitem{Deninger1992}
Deninger, C. (1992). Local L-factors of motives and regularized determinants. {\em Inventiones Mathematicae}, 107(1), 135-150.

\bibitem{Duistermaat1975}
Duistermaat, J. J. and Guillemin, V. W. (1975). The spectrum of positive elliptic operators and periodic bicharacteristics. {\em Inventiones Mathematicae}, 29(1), 39-79.

\bibitem{Dunne2008}
Dunne, G. V. (2008).
Functional determinants in quantum field theory.
{\em J. Phys. A: Math. Theor.}, 41(30), 304006.

\bibitem{Feller1948}
Feller, W. (1948). 
On the Kolmogorov-Smirnov limit theorems for empirical distributions. 
{\em The Annals of Mathematical Statistics}, 19(2), 177-189.

\bibitem{Forman1987}
Forman, R. (1987).
Functional determinants and geometry.
{\em Invent. Math.}, 88(3), 447--493.

\bibitem{GelfandYaglom1960}
Gelfand, I. M. and Yaglom, A. M. (1960).
Integration in functional spaces and its applications in quantum physics.
{\em J. Math. Phys.}, 1(1), 48--69.

\bibitem{IKS2012QIP}
Ide, Y., Konno, N. and Segawa, E. (2012).
Time averaged distribution of a discrete-time quantum walk on the path. {\em Quantum Inf. Process}, 11, 1207–1218 .

\bibitem{Ihara1966}
Ihara, Y. (1966). On discrete groups; subset of $SL(2, \mathbb{R})$. 
{\em J. Math. Soc. Japan}, 18(3), 219-235.

\bibitem{Kirsten2002}
Kirsten, K. (2002).
{\em Spectral Functions in Mathematics and Physics}.
Chapman \& Hall/CRC, Boca Raton.

\bibitem{KirstenMcKane2003}
Kirsten, K. and McKane, A. J. (2003).
Functional determinants by contour integration methods.
{\em Ann. Physics}, 308(2), 502--527.

\bibitem{Kolmogorov1933}
Kolmogorov, A. (1933). 
Sulla determinazione empirica di una legge di distribuzione. 
{\em Giornale dell'Istituto Italiano degli Attuari}, 4, 83-91.


\bibitem{KKS2021QIP}
Komatsu, T., Konno, N., and Sato, I. (2021). Grover/Zeta Correspondence based on the Konno–Sato theorem. {\em Quantum Inf. Process} 20, 268. 

\bibitem{KKS2023JSP}
Komatsu, T., Konno, N. and Sato, I.  (2023). 
Walk/Zeta Correspondence. {\em J. Stat. Phys}, 190, 36.


\bibitem{KS2012QIP}
Konno, N. and Sato, I. (2012). On the relation between quantum walks and zeta functions. {\em Quantum Inf. Process}, 
11, 341-349.

\bibitem{KonnoTamura2021}
Konno, N. and Tamura, S.(2021).
Walk/zeta correspondence for quantum and correlated random walks. {\em Yokohama Math. J.}, 67, 125–152. 

\bibitem{Kurokawa2005}
Kurokawa, N. (2005). 
Zeta functions over $\mathbb{F}_1$. {\em Proceedings of the Japan Academy, Series A, Mathematical Sciences}, 81(10), 180-184.

\bibitem{LevitSmilansky1977}
Levit, S. and Smilansky, U. (1977).
A theorem on infinite products of eigenvalues of Sturm--Liouville type
operators. {\em Proc. Amer. Math. Soc.}, 65(2), 299--302.


\bibitem{Luck1994}
L\"uck, W. (1994). Approximating $L^2$-invariants by their finite-dimensional analogues. {\em Geometric and Functional Analysis GAFA}, 4(4), 455-481.

\bibitem{Selberg1956}
Selberg, A. (1956). Harmonic analysis and discontinuous groups in weakly symmetric Riemannian spaces with applications to Dirichlet series. 
{\em Journal of the Indian Mathematical Society}, 20, 47-87.

\bibitem{Smirnov1939}
Smirnov, N. V. (1939). 
Sur le \text{ecarts} de la courbe de distribution empirique. 
{\em Matematicheskii Sbornik}, 48(1), 3-26.

\bibitem{SmithDiaconis1988}
Smith, L. and Diaconis, P. (1988). 
Honest Bernoulli excursions. 
{\em Journal of Applied Probability}, 25(3), 464-477.

\bibitem{Sze}
Szegedy, M. (2004). 
Quantum speed-up of Markov chain based algorithms. In Proceedings of the 45th Annual IEEE Symposium on Foundations of Computer Science, 32-41.

\end{thebibliography}
\end{document}